\documentclass[a4paper,12pt]{amsart}

\usepackage{amsmath}
\usepackage{amsfonts}
\usepackage{amssymb}
\usepackage{amsthm}
\usepackage[margin=3cm]{geometry}
\usepackage{latexsym}
\usepackage{graphicx}
\usepackage{url}
\usepackage{enumerate}
\usepackage{cite}       %

\usepackage{todonotes}

\usepackage{hyperref}

\newcommand{\doi}[1]{
    \href{http://dx.doi.org/#1}{\texttt{doi:#1}}
}

\usepackage{tikz}
\usetikzlibrary{decorations.pathmorphing}
\usetikzlibrary{decorations.pathreplacing}
\usetikzlibrary{patterns}
\usetikzlibrary{arrows}

\usepackage{tikz-cd}

\usepackage[T1]{fontenc}     %
\usepackage{lmodern}         %
\usepackage[utf8]{inputenc}  %

\newtheorem{definition}{Definition}[section]
\newtheorem{lemma}[definition]{Lemma}
\newtheorem{proposition}[definition]{Proposition}
\newtheorem{example}[definition]{Example}

\newtheorem{corollary}[definition]{Corollary}

\newtheorem{theorem}[definition]{Theorem}
\newtheorem{remark}[definition]{Remark}

\newtheorem{maintheorem}{Theorem}
\newtheorem{maincorollary}[maintheorem]{Corollary}

\newtheorem*{THEOREMA}{Theorem~\ref{thm:our-ostrowski-in-introduction}}
\newtheorem*{THEOREMB}{Theorem~\ref{thm:1-norm-statistics}}
\newtheorem*{THEOREMC}{Theorem~\ref{thm:nice-formula-for-order-ideals}}
\newtheorem*{THEOREMD}{Theorem~\ref{thm:nice-formula-for-qmatchings}}
\newtheorem*{COROLLARYE}{Corollary~\ref{cor:three-combinatorial-interpretation-of-q-rational}}

\newcommand{\N}{\mathbb{N}}
\newcommand{\Z}{\mathbb{Z}}
\newcommand{\Q}{\mathbb{Q}}
\newcommand{\R}{\mathbb{R}}

\newcommand{\GL}{\mathrm{GL}}

\newcommand{\xx}{{x}}

\newcommand{\M}{{\mathcal M}}
\newcommand{\Acal}{{\mathcal{A}}}
\newcommand{\Bcal}{{\mathcal{B}}}

\newcommand{\Fcal}{{\mathcal{F}}}
\newcommand{\Gcal}{{\mathcal{G}}}
\newcommand{\Jcal}{{\mathcal{J}}}
\newcommand{\Mcal}{{\mathcal{M}}}
\newcommand{\Pcal}{{\mathcal{P}}}

\newcommand{\Zcal}{{\mathcal{Z}}}

\newcommand{\bfrak}{{\mathfrak{b}}}

\newcommand{\0}{\mathtt{0}}
\newcommand{\1}{\mathtt{1}}

\newcommand{\carreplein}{{\perp}}
\newcommand{\carrevide}{{\parallel}}

\newcommand{\Acaleven}{{\Acal_{\text{even}}}}

\DeclareMathOperator{\Pcaleven}{{\Pcal_{\text{even}}}}
\DeclareMathOperator{\Pcalodd}{{\Pcal_{\text{odd}}}}
\DeclareMathOperator{\CFeven}{{CF_{\text{even}}}}
\DeclareMathOperator{\CFodd}{{CF_{\text{odd}}}}
\DeclareMathOperator{\region}{region}
\DeclareMathOperator{\area}{area}

\DeclareMathOperator{\rep}{rep}
\DeclareMathOperator{\val}{val}
\DeclareMathOperator{\shift}{\textsc{shift}}
\DeclareMathOperator{\pop}{\textsc{pop}}
\DeclareMathOperator{\conv}{conv}

\newcommand{\size}[1]{
    {|#1|}                 %
}

\def\enclosedAreaColor{green!10}
\def\basicMatchingColor{red}
\def\someMatchingColor{blue}
\tikzset{
circled node/.style={circle,fill, inner sep=0, minimum size=.3em},
empty circled node/.style={draw, shape=circle, inner sep=2pt, fill=white},
plein circled node/.style={draw, shape=circle, inner sep=2pt, fill=black!50},
thebasicmatching/.style={ultra thick,dotted,bend left=0,\basicMatchingColor},
somematching/.style={thick,\someMatchingColor},
}

\begin{document}

\title[$q$-analogs of rational numbers]
{$q$-analogs of rational numbers:\\ from Ostrowski numeration systems\\to perfect matchings}

\author[J.-C.~Aval]{Jean-Christophe Aval}
\address[J.-C.~Aval]{Univ. Bordeaux, CNRS, Bordeaux INP, LaBRI, UMR 5800, F-33400 Talence, France}
\email{aval@labri.fr}
\urladdr{https://www.labri.fr/perso/aval/}

\author[S.~Labb\'e]{S\'ebastien Labb\'e}
\address[S.~Labb\'e]{CNRS -- Université de Montréal CRM-CNRS, Montréal, Canada}
\email{sebastien.labbe@cnrs.fr}
\urladdr{http://www.slabbe.org/}


\date{\today}

\begin{abstract}
We consider the $q$-deformation of rational numbers introduced recently by Morier-Genoud and Ovsienko. We propose three enumerative interpretations of these $q$-rationals: in terms of a new version of Ostrowski's numeration system for integers, in terms of order ideals of fence posets and in terms of perfect matchings of snake graphs. %
    Contrary to previous results which are restricted to rational numbers greater than one, our interpretations work for all positive rational numbers and are based on a single combinatorial object for defining both the numerator and denominator. The proofs rest on order-preserving bijections between posets over these objects. We recover a formula for a $q$-analog of Markoff numbers. We also deduce a fourth interpretation given in terms of the integer points inside a polytope in $\mathbb{R}^k$ on both sides of a hyperplane where $k$ is the length of the continued fraction expansion.
\end{abstract}

\maketitle

\setcounter{tocdepth}{1}
\tableofcontents

\section{Introduction}

Fix a positive irrational number $\xx$ with continued fraction expansion
$[a_0; a_1, a_2, \dots]$ with $a_0\geq0$ and $a_n\geq1$ for every $n\geq1$.
Let $(p_n)$ and $(q_n)$ be the sequence of numerators and denominators of the
convergents $p_n/q_n$ to $\xx$ \cite[Theorem 149]{MR2445243}:
\begin{equation}\label{eq:CF-denominators}
\begin{aligned}
    p_i &= a_ip_{i-1} + p_{i-2},  &&\text{ with } p_{-1}=1 \text{ and } p_0=a_0,\\
    q_i &= a_iq_{i-1} + q_{i-2},  &&\text{ with } q_{-1}=0 \text{ and } q_0=1.
\end{aligned}
\end{equation}
It was proved by Ostrowski a century ago \cite{zbMATH02600866}
that every positive integer $n$ can be written uniquely as
\[
    n = \sum_{i=1}^{k}b_{i}q_{i}
\]
where the integer coefficients satisfy $0 \leq b_i \leq a_i$ and if $b_i = a_i$
then $b_{i-1} = 0$.
See also \cite{berthe_autour_2001,MR1997038, MR1970391, MR1970385} or the more recent
\cite{epifanio_sturmian_2012}.

To have uniqueness, different conditions can be imposed on the sequence $(b_i)_i$,
leading to different versions of the theorem. For example, two choices are
called \emph{first} and \emph{second} multiplicative sequences in the chapter
\cite{MR1970391} describing the substitutive structure of Sturmian sequences.
In this article, we propose a \emph{third} choice
leading to an alternating-sign analog of Ostrowski's theorem
using the sum of the numerator and denominator of the convergents.

Our choice of sequences can be stated as follows.
Let $a=[a_0; a_1, \dots, a_{k-1}]$
be the even or odd-length continued fraction expansion
of a positive rational number
where $a_0\geq0$ and $a_i\geq1$ for every index $i$ with $1\leq i<k$
\cite[Theorem 162]{MR2445243}.
    We say that a finite sequence of integers $(b_i)_{0\leq i< k}$ is
    \emph{admissible for $a$} if it satisfies
    \begin{itemize}
        \item $0 \leq b_i \leq a_i$ for every integer $i$ with $0\leq i<k$,
        \item if $i>0$ is odd and $b_i = a_i$, then $b_{i-1} = a_{i-1}$,
        \item if $i>0$ is even  and $b_i = 0$, then $b_{i-1} = 0$.
    \end{itemize}
    We reuse the vocabulary of ``admissible'' sequence from \cite{MR1020484,MR4836876}
    where it was used to describe another numeration system based on a single substitution.

Also, let $(r_i)_{0\leq i\leq k}$ be the sequence defined by
\[
    r_i=p_{i-1}+q_{i-1}
\]
for every integer $i$ with $0\leq i\leq k$
where $(p_i)_{0\leq i\leq k-1}$ and $(q_i)_{0\leq i\leq k-1}$
are the numerators and denominators of the convergents
defined as above from the finite continued fraction expansion $a$.
Note that we have $r_0=p_{-1}+q_{-1}=1$ and $r_1=p_0+q_0=a_0+1$.

\newcommand\MainTheoremA{
    Let $k\geq1$ be an integer and
    $a=[a_0; a_1, \dots, a_{k-1}]$ be the even or odd-length continued
    fraction expansion of a positive rational number.
    Let $\Zcal(a)=\Z\,\cap\,[0,r_{k})$ if $k$ is odd
    or $\Zcal(a)=\Z\cap[r_{k-1}-r_{k},r_{k-1})$ if $k$ is even.
    Every integer $n\in \Zcal(a)$ can be written uniquely as
    \[
        n = \sum_{i=0}^{k-1}(-1)^ib_{i}r_{i}
    \]
    where  $(b_i)_{0\leq i\leq k-1}$ is an admissible sequence for $a$.
}

    Using this choice of admissible sequences, we propose the following
    alternating-sign version of Ostrowski's theorem.
    Note that it is distinct from the alternate Ostrowski numeration
    system proposed in \cite{bourla_ostrowski_2016}.

\begin{maintheorem}\label{thm:our-ostrowski-in-introduction}
    \MainTheoremA
\end{maintheorem}

Given a positive rational number $\xx$ with continued fraction
expansion $a=[a_0; a_1, \dots, a_{k-1}]$, it is natural to define
the set of admissible sequences for $a$:
\[
        \Bcal(a)
        =
        \{(b_i)_{0\leq i\leq k-1} \mid
            (b_i)_{0\leq i\leq k-1}
            \text{ is admissible for } a
        \}.
\]
Theorem~\ref{thm:our-ostrowski-in-introduction}
implies that $n\mapsto(b_i)_{0\leq i\leq k-1}$
defines a bijection $\rep_a:\Zcal(a)\to\Bcal(a)$
representing every integer in the interval $\Zcal(a)$ as an admissible sequence for
$a$.

When $a=[1;1,\dots,1]$ is a convergent to the golden ratio,
we recover similar values %
as in the negaFibonacci number system proposed by Donald Knuth
\cite{zbMATH05597000} in Section
7.1.3 \textit{Bitwise tricks \& techniques}.
But, a difference is that admissible sequences allow using two consecutive
Fibonacci numbers, whereas negaFibonacci numeration system forbids doing this;
see Example~\ref{ex:negaFibonacci}.

We split the set $\Bcal(a)$ into two as the disjoint union
$\Bcal(a)= \Bcal^\bullet(a) \cup \Bcal^\circ(a)$
where
    \begin{align*}
        \Bcal^\bullet(a)
        &=
        \{(b_i)_{0\leq i\leq k-1}
        \in\Bcal(a)
            \mid
            b_0=a_0=0<b_1=a_1
            \text{ or }
            0<b_0
        \},\\
        \Bcal^\circ(a)
        &=
        \{(b_i)_{0\leq i\leq k-1}
        \in\Bcal(a)
            \mid
            b_0=a_0=0\leq b_1<a_1
            \text{ or }
        0=b_0<a_0
    \}.
\end{align*}
When $k$ is even, we can show that the cardinality of the two sets are equal
respectively to the numerator and denominator of the rational number
represented by $a$, that is,
\[
        \#\Bcal^\bullet(a) = p_{k-1}
        \quad
        \text{ and }
        \quad
        \#\Bcal^\circ(a) = q_{k-1}.
\]
This is a consequence of the following stronger result which states that
the 1-norm statistics over the set of admissible sequences for $a$
can be computed from a product of the following two matrices:
\[
    L_q=
\left(\begin{array}{rr}
q & 0 \\
q & 1
\end{array}\right)
\qquad
\text{ and }
\qquad
    R_q=
\left(\begin{array}{rr}
q & 1 \\
0 & 1
\end{array}\right)
\]
where $q$ is some indeterminate.

\newcommand\MainTheoremB{
Let $a=[a_0;\dots,a_{2\ell-1}]$ be the even-length continued fraction expansion
of a positive rational number.
The set of admissible sequences for $a$ satisfies
    \begin{equation*}
    \left(\begin{array}{r}
        \sum_{b\in \Bcal^\bullet(a)} q^{\Vert b\Vert_1}\\
          \sum_{b\in \Bcal^\circ  (a)} q^{\Vert b\Vert_1}
    \end{array}\right)
        =
        \left(\begin{smallmatrix} 1 & 0 \\ 0 & q \end{smallmatrix}\right)^{-1}
        R_q^{a_0}L_q^{a_1}\cdots R_q^{a_{2\ell-2}}L_q^{a_{2\ell-1}}
        \left(\begin{smallmatrix} 1 \\ 0 \end{smallmatrix}\right).
    \end{equation*}
where
$ \Vert b\Vert_1 = b_0+b_1+\dots+b_{2\ell-1} $
    for every admissible sequence $(b_i)_{0\leq i< 2\ell}\in\Bcal(a)$.
}

\begin{maintheorem}\label{thm:1-norm-statistics}
\MainTheoremB
\end{maintheorem}

We prove in Section~\ref{sec:convexity} that the set $\Bcal(a)$ of admissible
sequences is a convex subset in $\Z^{2\ell}$
whose convex hull is a polytope
and that $\Bcal^\bullet(a)$ and
$\Bcal^\circ(a)$ are the two sets of integer points in the polytope
on both sides of a hyperplane.

A $q$-analog of all rational numbers was proposed in \cite{MR4073883}.
It is given as a ratio of two polynomials in the indeterminate $q$.
When the rational number is larger than one, they also gave a combinatorial
interpretation of the coefficients of the numerator and denominator
polynomials in terms of the subsets of closures of an oriented path graph.
Closures can equivalently be defined as the lower order ideals of a fence poset
\cite{MR4266256,MR4499341}. This is the terminology that we use in this work.

We extend the definition of fence posets proposed in
\cite{MR4266256,MR4499341} in order to describe all of them, not only those
that start with an up step.
This allows to interpret simultaneously the numerator and the denominator
of a $q$-rational number on a single fence poset.
Also, this provides an enumerative interpretation of the $q$-analog of
every positive rational number without the condition $>1$ assumed in \cite{MR4073883}.

To every positive rational number $x$,
whose even-length continued fraction expansion is
$a=[a_0;a_1,\dots,a_{2\ell-1}]$,
we associate a fence poset $\Fcal(\xx)$ containing
$a_0+a_1+\dots+a_{2\ell-1}$ elements
and $a_0+a_1+\dots+a_{2\ell-1}-1$ covering relations;
please refer to the precise definition in
Section~\ref{sec:bijection-order-ideals-admissible}.
To each such fence poset $\Fcal(\xx)$, we denote
its set of lower order ideals by $\Jcal(\xx)$.

The set $\Jcal(\xx)$ of lower order ideals
is naturally partitioned into a disjoint union
$\Jcal(\xx)= \Jcal^\bullet(\xx) \cup \Jcal^\circ(\xx)$
according to whether the left-most vertex of the fence poset is in the order
ideal ($\bullet$) or not ($\circ$).
We prove the following theorem
about the rank polynomial of the lower order ideals.

\newcommand\MainTheoremC{
    For every positive rational number $x\in\Q_{>0}$,
whose even-length continued fraction expansion is $[a_0;a_1,\dots,a_{2\ell-1}]$,
the cardinality statistics over the set $\Jcal(\xx)$ of order ideals
of the fence poset $\Fcal(\xx)$ satisfies
\begin{equation*}
\left(\begin{array}{r}
        \sum_{I\in \Jcal^\bullet(\xx)} q^{\size{I}}\\
        \sum_{I\in \Jcal^\circ  (\xx)} q^{\size{I}}
\end{array}\right)
    =
    \left(\begin{smallmatrix} 1 & 0 \\ 0 & q \end{smallmatrix}\right)^{-1}
    R_q^{a_0}L_q^{a_1}\cdots R_q^{a_{2\ell-2}}L_q^{a_{2\ell-1}}
    \left(\begin{smallmatrix} 1 \\ 0 \end{smallmatrix}\right).
\end{equation*}
}

\begin{maintheorem}\label{thm:nice-formula-for-order-ideals}
\MainTheoremC
\end{maintheorem}

Snake graphs have been introduced to provide combinatorial formulas
for elements in cluster algebras of surface type
\cite{MR2661414,MR2807089,MR3034481,zbMATH06144657},
see also \cite{zbMATH07181526}.
Their relation with continued fraction expansion of rational numbers
greater than one is known \cite{zbMATH06559878,MR4058266}.

To every positive rational number $\xx$,
whose even-length continued fraction expansion is $[a_0;a_1,\dots,a_{2\ell-1}]$,
we associate a snake graph $\Gcal(\xx)$ embedded in the plane $\R^2$ containing
$a_0+a_1+\dots+a_{2\ell-1}$ unit squares.
This is one more unit square than the usual definition \cite{MR3778183}.
This choice allows us to associate bijectively a snake graph to all positive rational
numbers, not only those that are larger than one.

The bijection from positive rational numbers to the set of all snake graphs makes sense:
if $\xx=\frac{r}{s}$ for some coprime integers
$r,s\geq1$, we show that the snake graph $\Gcal(\xx)$ contains exactly $r+s$
perfect matchings.
Moreover, the set of perfect matchings $\Mcal(\xx)$
of the snake graph $\Gcal(\xx)$
is naturally partitioned into a disjoint union
$\Mcal(\xx)= \Mcal^\carreplein(\xx) \cup \Mcal^\carrevide(\xx)$ according to
the parity of the length of the snake graph and whether
the first edge of the matching is vertical or horizontal;
see Definition~\ref{def:matching-dichothomoy}. We prove that
the cardinalities of the two sets
\[
        \#\Mcal^\carreplein(\xx) = r
        \quad
        \text{ and }
        \quad
        \#\Mcal^\carrevide(\xx) = s
\]
are the numerator and denominator of the rational number $\xx$.

Again this follows from a stronger result involving
Morier-Genoud--Ovsienko's $q$-analogs of rational numbers.
We use a statistics on perfect matchings counting
the area of the region enclosed by the symmetric difference $m\Delta\bfrak$
of each perfect matching $m$ with some fixed canonical \emph{basic} perfect
matching $\bfrak$.
Precise definitions are given in Section~\ref{sec:snake-graphs-matchings}.
The region enclosed by the symmetric difference of two perfect matchings of
a snake graph was used previously to define the height monomials of a cluster
algebra.
This combinatorial interpretation was used to prove
the positivity conjecture (nonnegativity of the coefficients of the cluster
expansions of cluster variables with respect to any seed) for cluster algebras
coming from a surface \cite{MR2807089,zbMATH06144657}.

We prove that the area statistics of the perfect matchings of a snake graph
is computed from the continued fraction expansion of the rational number
defining the snake graph.

\newcommand\MainTheoremD{
For every positive rational number $x\in\Q_{>0}$,
whose even-length continued fraction expansion is $[a_0;a_1,\dots,a_{2\ell-1}]$,
the area statistics over the set of perfect matchings
of the snake graph $\Gcal(\xx)$ satisfies
\begin{equation*}
\left(\begin{array}{r}
        \sum_{m\in \Mcal^\carreplein(\xx)} q^{\area(m\Delta\bfrak)}\\
        \sum_{m\in \Mcal^\carrevide (\xx)} q^{\area(m\Delta\bfrak)}
\end{array}\right)
    =
    \left(\begin{smallmatrix} 1 & 0 \\ 0 & q \end{smallmatrix}\right)^{-1}
    R_q^{a_0}L_q^{a_1}\cdots R_q^{a_{2\ell-2}}L_q^{a_{2\ell-1}}
    \left(\begin{smallmatrix} 1 \\ 0 \end{smallmatrix}\right).
\end{equation*}
where $\bfrak$ is the basic perfect matching of $\Gcal(\xx)$.
}

\begin{maintheorem}\label{thm:nice-formula-for-qmatchings}
\MainTheoremD
\end{maintheorem}

Note that a perfect matching $m$ belongs to
$\Mcal^\carreplein(\xx)$ if and only if
the first unit square of the snake graph is in the region
enclosed by the cycles of $m\Delta\bfrak$;
see Remark~\ref{rem:carreplein-interpretation}.
In particular, $\area(m\Delta\bfrak)\geq1$ when $m\in\Mcal^\carreplein(\xx)$.

The proof of Theorem~\ref{thm:nice-formula-for-order-ideals} is
based on Proposition~\ref{prop:equivalent-conditions-Fq}
where equivalent conditions are given for a pair of polynomial functions
to be described by product of $L_q$ and $R_q$ matrices.
Then, Theorem~\ref{thm:1-norm-statistics}
and Theorem~\ref{thm:nice-formula-for-qmatchings} are deduced
from bijections.

\begin{figure}[h]
\begin{center}
\begin{tikzcd}[ampersand replacement=\&]
    \& [0mm,between origins]
    \Q_{>0}
    \arrow{r}{\CFeven}
    \& \Acaleven
    \arrow{r}{W}
	\& \{\0,\1\}^*
    \arrow{r}{\theta}
    \& \{\0,\1\}^*\\[-5mm]
    \& x
    \arrow[draw=none]{u}[sloped,auto=false]{\in}
    \arrow[mapsto]{r}{}
    \arrow[mapsto]{dd}{}
    \& a
    \arrow[draw=none]{u}[sloped,auto=false]{\in}
    \arrow[mapsto]{r}{}
    \arrow[mapsto]{dd}{\Bcal}
    \& W(\xx)
    \arrow[draw=none]{u}[sloped,auto=false]{\in}
    \arrow[mapsto]{r}{}
    \arrow[mapsto]{d}{F}
    \& \theta(W(\xx))
    \arrow[draw=none]{u}[sloped,auto=false]{\in}
    \arrow[mapsto]{d}{G}
    \\
    \& \& \& \Fcal(\xx)
    \arrow[mapsto]{d}{J}
    \& \Gcal(\xx)
    \arrow[mapsto]{d}{M}\\
    \Z
    \arrow[draw=none]{r}[sloped,auto=false]{\supset}
    \& \Zcal(a)
    \& \Bcal(a)
    \arrow[swap]{l}{\val_a}
    \arrow[mapsto]{dr}{}
    \& \Jcal(\xx)
    \arrow[swap]{l}{\Psi}
    \arrow[mapsto]{d}{}
    \& \Mcal(\xx)
    \arrow[swap]{l}{\Phi}
    \arrow[mapsto]{dl}{}
    \\[-3mm]
    \& \& \&
    {[x]_q}
\end{tikzcd}
\end{center}
    \caption{The big picture of bijections presented in this article.}
    \label{fig:the-big-picture}
\end{figure}
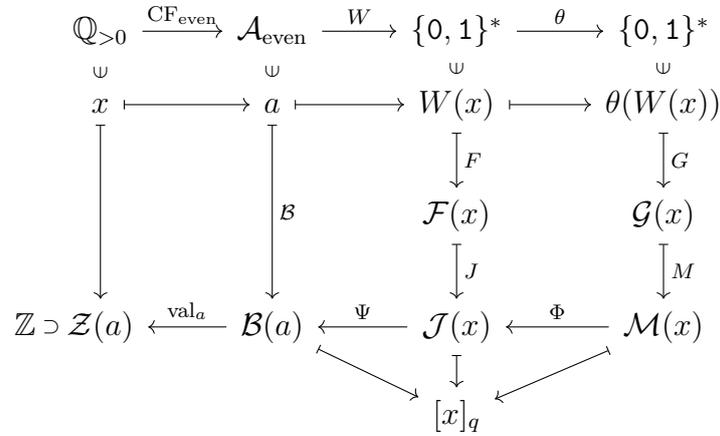

There are two levels of bijections.
On the first level, we represent positive
rational numbers as binary words
using a bijection $W\circ\CFeven:\Q_{>0}\to\{\0,\1\}^*$
where $\CFeven$ returns the even-length continued fraction expansion
of a positive rational number.
We also have an involution $\theta:\{\0,\1\}^*\to\{\0,\1\}^*$:
\begin{center}
\begin{tikzcd}[ampersand replacement=\&]
    \Q_{>0}
    \arrow{r}{\CFeven}
	\& \Acaleven
    \arrow{r}{W}
	\& \{\0,\1\}^*
    \arrow{r}{\theta}
	\& \{\0,\1\}^*.
\end{tikzcd}
\end{center}

The second level of bijections depends on a
chosen rational number $\xx>0$
whose even-length continued fraction expansion is $a=[a_0;a_1,\dots,a_{2\ell-1}]$.
We prove the existence of three bijections
$\Phi:\Mcal(\xx)\to\Jcal(\xx)$,
$\Psi:\Jcal(\xx)\to\Bcal(a)$ and
$\val_a:\Bcal(a)\to \Zcal(a)$
involving
the set $\Mcal(\xx)$ of perfect matchings of the snake graph $\Gcal(\xx)$,
the set $\Jcal(\xx)$ of order ideals of the fence poset $\Fcal(\xx)$,
the set $\Bcal(a)$ of admissible sequences for $a$
and
an interval of integers $\Zcal(a)\subset\Z$ that depend on $a$:
\begin{center}
\begin{tikzcd}[ampersand replacement=\&]
       \Mcal(\xx) \arrow{r}{\Phi}
    \& \Jcal(\xx) \arrow{r}{\Psi}
    \& \Bcal(a) \arrow{r}{\val_a}
    \& \Zcal(a)\subset\Z.
\end{tikzcd}
\end{center}
The second level of bijections depends on the first since
the snake graph $\Gcal(\xx)$ is defined from the word $\theta\circ W\circ\CFeven(\xx)$
and the fence poset $\Fcal(\xx)$ is defined from the word $W\circ\CFeven(\xx)$.
The three bijections $\Phi$, $\Psi$ and $\val_a$
are proved in
Theorem~\ref{thm:bijection-admissible-to-interval-integers},
Theorem~\ref{thm:bijection-order-ideals-to-integersequences} and
Theorem~\ref{theo:bij}.
The bijections preserve statistics that are used to define polynomials
in the indeterminate $q$.
In particular, $\Phi$ and $\Psi$ are order-preserving
for some natural partial orderings defined on $\Mcal(\xx)$, $\Jcal(\xx)$ and $\Bcal(a)$.
The big picture of all bijections of the two levels is shown in
Figure~\ref{fig:the-big-picture}
and are illustrated in
Figure~\ref{fig:big-picture-for-4-over-5} for the rational number $4/5$ whose
even-length continued fraction expansion is $[0;1,3,1]$.

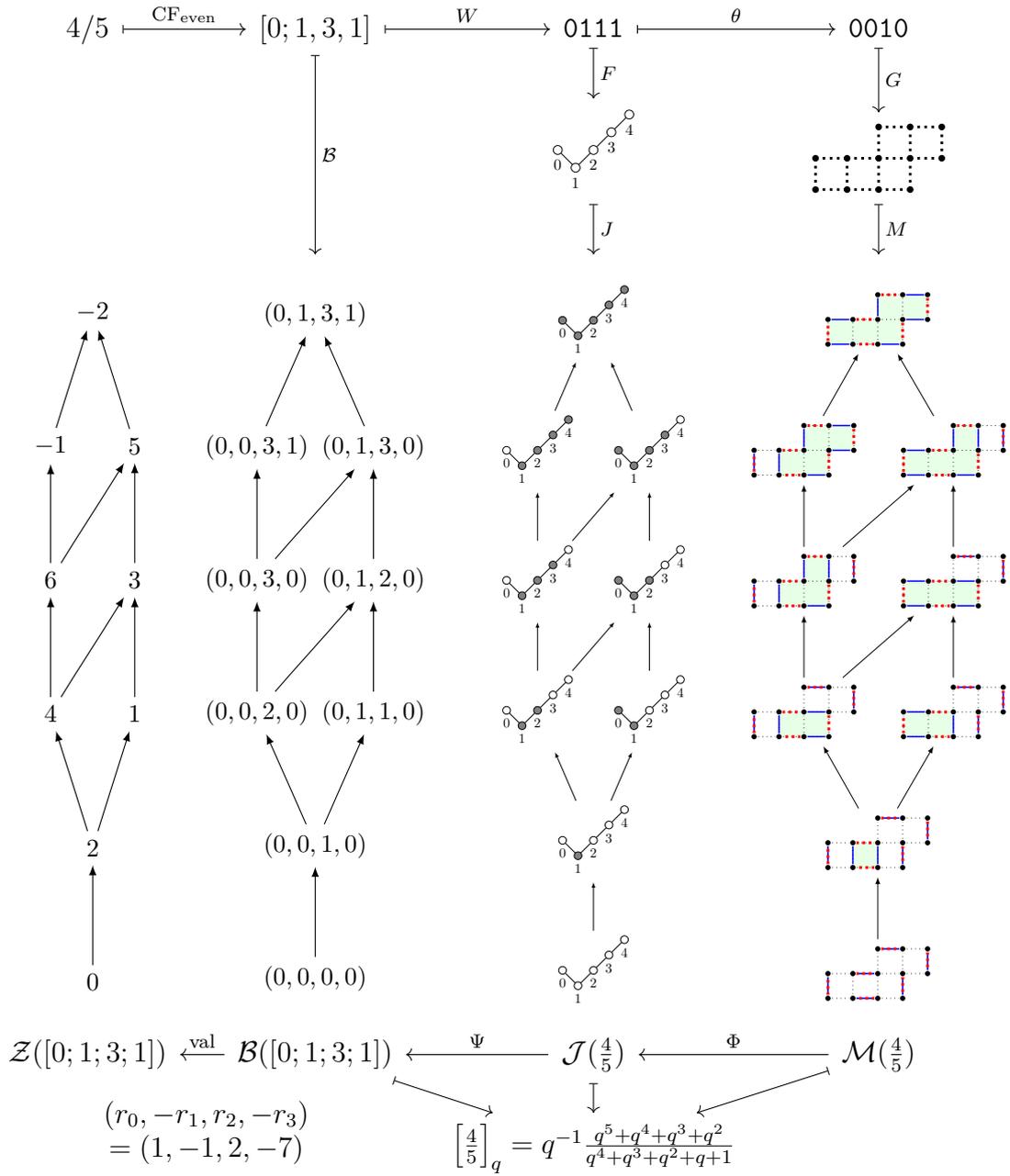
\begin{figure}
    \begin{center}
        \begin{tikzcd}[ampersand replacement=\&]
                      4/5           \arrow[mapsto]{r}{\CFeven}
            \&[-15mm] {[0;1,3,1]}   \arrow[mapsto]{r}{W}   \arrow[mapsto]{dd}{\Bcal}
            \&[-15mm] \0\1\1\1      \arrow[mapsto]{r}{\theta} \arrow[mapsto]{d}{F}
            \&[-15mm] \0\0\1\0                                \arrow[mapsto]{d}{G}
            \\[-1mm]
            \&
            \& \includegraphics[height=12mm]{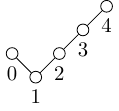} %
			   \arrow[mapsto]{d}{J}
            \& \includegraphics[height=10mm]{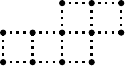} %
               \arrow[mapsto]{d}{M}
            \\
              \includegraphics[height=11cm]{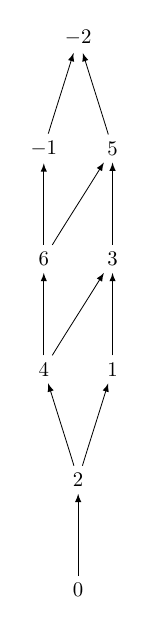}
            \&\includegraphics[height=11cm]{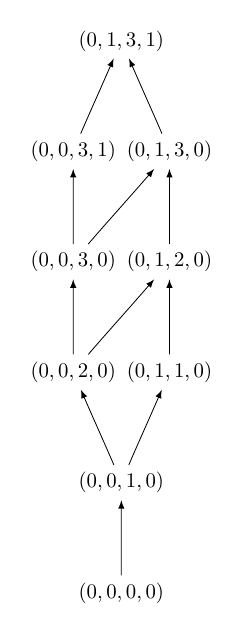}
            \&\includegraphics[height=11cm]{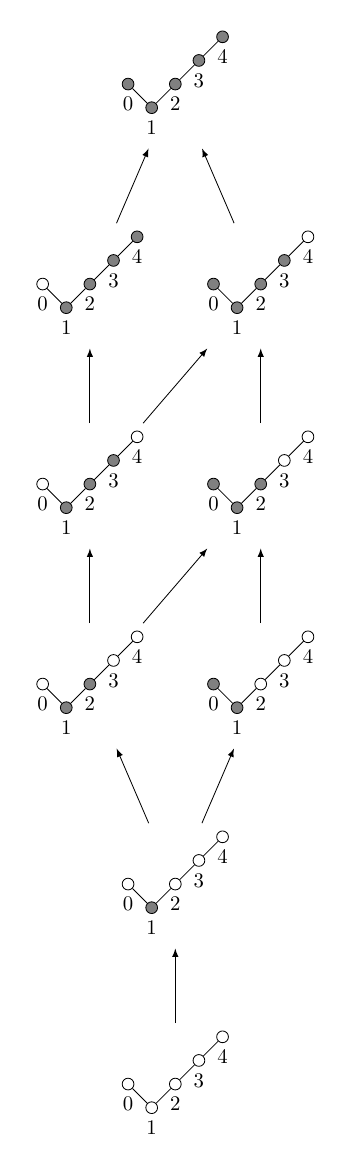}
            \&\includegraphics[height=11cm]{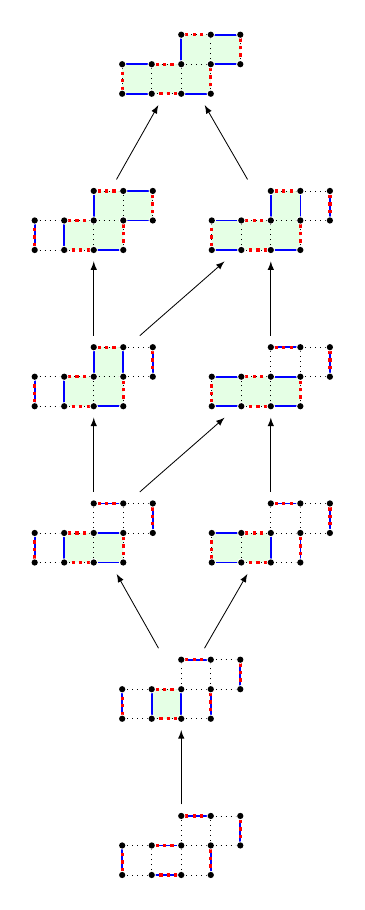}
            \\[-9mm]
               \Zcal([0;1;3;1])
						\arrow[phantom,xshift=17mm,yshift=-4mm]{d}
                            {\begin{array}{c}(r_0,-r_1,r_2,-r_3)\\= (1,-1,2,-7)\end{array}}
            \& \Bcal([0;1;3;1]) \arrow[swap]{l}{\val}%
               \arrow[mapsto]{dr}{}
            \& \Jcal(\frac{4}{5}) \arrow[swap]{l}{\Psi}
               \arrow[mapsto]{d}{}
            \& \Mcal(\frac{4}{5}) \arrow[swap]{l}{\Phi}
               \arrow[mapsto]{dl}{}
            \\[-3mm]
            \phantom{x} \& \&
            {
            \left[\frac{4}{5}\right]_q
                    = q^{-1}\frac{q^5+q^4+q^3+q^2}
                                {q^4+q^3+q^2+q+1}
            }
        \end{tikzcd}
    \end{center}
    \caption{The admissible sequences
            $\Bcal(\frac{4}{5})$ for $\frac{4}{5}=[0;1,3,1]$,
			the order ideals $\Jcal(\frac{4}{5})$ of the
            fence poset $\Fcal(\frac{4}{5})$,
            and the set of perfect matchings $\Mcal(\frac{4}{5})$ of the snake
            graph $\Gcal(\frac{4}{5})$ are isomorphic posets.
            The poset is partitioned into two disjoint poset intervals
            whose rank polynomials are the numerator and denominator
            of $[4/5]_q$ up to a scalar $q^{-1}$.
            }
    \label{fig:big-picture-for-4-over-5}
\end{figure}

As a corollary of these three bijections,
we obtain three different enumerative interpretations of the
Morier-Genoud and Ovsienko's
$q$-analog of positive rational numbers \cite{MR4073883}.
The definition of this $q$-analog is recalled Section~\ref{sec:Morier-Genoud-Ovsienko-q-analog}.

\newcommand\MainCorollaryE{
    Let $\xx>0$ be a positive rational number
    whose even-length continued fraction expansion
    is $a=[a_0;\dots,a_{2\ell-1}]$.
    Then,
    the set $\Bcal(a)$ of admissible sequences for $a$,
    the set $\Mcal(\xx)$ of perfect matchings of the snake graph $\Gcal(\xx)$ and
    the set $\Jcal(\xx)$ of order ideals of the fence poset $\Fcal(\xx)$
    give three enumerative interpretations
    of the Morier-Genoud--Ovsienko $q$-analog of the rational number $\xx$:
\[
\left[x\right]_q
=
\frac{q^{-1}\sum_{b\in \Bcal^\bullet(a)} q^{\Vert b\Vert_1}}
            {\sum_{b\in \Bcal^\circ  (a)} q^{\Vert b\Vert_1}}
=
\frac{q^{-1}\sum_{I\in \Jcal^\bullet(\xx)} q^{\size{I}}}
            {\sum_{I\in \Jcal^\circ  (\xx)} q^{\size{I}}}
=
\frac{q^{-1}\sum_{m\in \Mcal^\carreplein(\xx)} q^{\area(m\Delta\bfrak)}}
            {\sum_{m\in \Mcal^\carrevide (\xx)} q^{\area(m\Delta\bfrak)}}
\]
where $\bfrak$ is the basic perfect matching of the snake graph $\Gcal(\xx)$
and the fractions on the right-hand sides are reduced.
}

\begin{maincorollary}\label{cor:three-combinatorial-interpretation-of-q-rational}
\MainCorollaryE
\end{maincorollary}

Our interpretation in terms of order ideals of fence posets is more general
than the one proposed in \cite{MR4073883} as it is not restricted to rational
numbers larger than one.
Similarly, our interpretation in terms of perfect matchings of snake graphs
works for all positive rational numbers and, contrary to \cite{MR3778183,MR4058266},
it is not restricted to rational numbers larger than one.
Also, our interpretation of the $q$-analog of rational numbers in terms of
admissible sequences is new.

An interesting aspect of
Corollary~\ref{cor:three-combinatorial-interpretation-of-q-rational},
compared to previous combinatorial interpretations provided in
\cite{MR3778183,MR4058266,MR4073883,musiker_higher_2023},
is that the numerator polynomial and denominator polynomial are obtained
from a single fence poset and a single snake graph.
This seems to be an original result even in the case of $q=1$ providing three
combinatorial interpretations of the numerator and denominator of every
positive rational number.
The usual interpretation involving different combinatorial objects
for the numerator and denominator can be deduced as a consequence
of our main results; see Corollary~\ref{cor:the-usual-combinatorial-interpretation}.

\subsection*{Structure of the article}
We recall the continued fraction expansion of rational numbers
and their encodings into binary words in Section~\ref{sec:continued-fraction}.
Then, we recall Morier-Genoud--Ovsienko's $q$-analog of rational numbers 
in Section~\ref{sec:Morier-Genoud-Ovsienko-q-analog}.
In Section~\ref{sec:bijection-admissible-to-integers},
we prove the bijection between admissible sequences
and integers in an interval (Theorem~\ref{thm:our-ostrowski-in-introduction}).
In Section~\ref{sec:bijection-order-ideals-admissible},
we prove the order-preserving bijection from fence posets order ideals
to admissible sequences.
In Section~\ref{sec:proof-of-Theorems-B-and-D},
we prove Theorem~\ref{thm:nice-formula-for-order-ideals}
and deduce Theorem~\ref{thm:1-norm-statistics}.
In Section~\ref{sec:snake-graphs-matchings},
we recall the definition of snake graphs, their perfect matchings
and the area statistics.
In Section~\ref{sec:bijection-perfect-matchings-to-order-ideals},
we prove the order-preserving bijection from the set of perfect matchings of a snake graph to
the set of order ideals of a fence poset.
In Section~\ref{sec:proof-of-thm:nice-formula-for-qmatchings},
we prove Theorem~\ref{thm:nice-formula-for-qmatchings}
and Corollary~\ref{cor:three-combinatorial-interpretation-of-q-rational}.
In Section~\ref{sec:Markoff},
we express how to recover Markoff numbers and a $q$-analog of them
from these statistics.
Finally, we show in Section~\ref{sec:convexity} that
the set of admissible sequences is the set 
of integer points inside of a polytope.

\section{Continued fractions}
\label{sec:continued-fraction}

In this section, we recall well-known notions about the continued fraction
expansion of real numbers; see \cite{MR2445243} for more details.
We also present an encoding of rational numbers as binary words
which will be useful thereafter.

\subsection{Continued fraction expansions of rational numbers}

Theorem 162 from \cite{MR2445243} says that
a rational number can be expressed as a finite
continued fraction expansion in just two ways:
one of odd-length and the other of even-length.
In one form, the last partial quotient is 1,
in the other, it is greater than 1.
For example,
\[
    \frac{22}{7} = 3 + \frac{1}{7}
    = 3 + \frac{1}{
      6 + \frac{1}{1
      }}
      \qquad
      \text{ and }
      \qquad
    \frac{7}{22} = 0 + \frac{1}{3 + \frac{1}{7}}
    = 0 + \frac{1}{
      3 + \frac{1}{
      6 + \frac{1}{1
      }}}.
\]
The above is often more compactly written with square brackets as
\[
    \frac{22}{7}=[3;7] = [3;6,1]
      \qquad
      \text{ and }
      \qquad
    \frac{7}{22}=[0;3,7] = [0;3,6,1].
\]
In this work, we have a preference for using the
even-length expansions:
\[
\Acaleven=\{
    [a_0;a_1,\dots,a_{2\ell-1}]\in\N^{2\ell}
\colon
\ell\geq1,
a_0\geq0\text{ and }
a_i>0 \text{ when } 0<i<2\ell
\}
\]
where the $a_i$ are called \emph{partial quotients}.
Thus, it leads to the following bijection
\[
\begin{array}{rccl}
\CFeven:&\Q_{>0}&\to&\Acaleven\\
&\xx
&\mapsto
&[a_0;a_1,\dots,a_{2\ell-1}]
\end{array}
\]
where $[a_0;a_1,\dots,a_{2\ell-1}]$,
for some integer $\ell\geq1$,
is the unique even-length continued fraction expansion of the positive rational
number $\xx$.

\subsection{Continued fraction convergents in matrix form}

Let $x\in\Q_{>0}$
and \[\CFeven(x)=[a_0;a_1,\dots,a_{2\ell-1}]\] be
its even-length continued fraction expansion
for some integer $\ell\geq1$.
If $x=r/s$ where $r,s\geq1$ are coprime integers,
then it can be deduced \cite[Theorem 149]{MR2445243}
that
\begin{equation}\label{eq:CF-expansion-LR}
    \begin{aligned}
        \left(\begin{array}{c} r \\ s \end{array}\right)
        &=
        R^{a_0}L^{a_1}\cdots R^{a_{2\ell-2}}L^{a_{2\ell-1}}
         \left(\begin{smallmatrix} 1 \\ 0 \end{smallmatrix}\right)\\
        &=
        R^{a_0}L^{a_1}\cdots R^{a_{2\ell-2}}L^{a_{2\ell-1}-1}
         \left(\begin{smallmatrix} 1 \\ 1 \end{smallmatrix}\right)
    \end{aligned}
\end{equation}
where $L$ and $R$ are the elementary matrices
\[
L=
\left(\begin{array}{rr}
1 & 0 \\
1 & 1
\end{array}\right)
\qquad
\text{ and }
\qquad
R=
\left(\begin{array}{rr}
1 & 1 \\
0 & 1
\end{array}\right).
\]

\subsection{Binary encodings of rational numbers}
\label{sec:binary-encoding-of-QQ}

The free monoid $\{\0,\1\}^*$ is the set of all finite words
over the alphabet $\{\0,\1\}$ together with a binary operation
of concatenation. The neutral element for the concatenation is the
word of length zero, and denoted $\varepsilon$. If $w=ps\in\{\0,\1\}^*$
is the concatenation of two other words $p,s\in\{\0,\1\}^*$,
then we say that $p$ is a \emph{prefix} of $w$
and $s$ is a \emph{suffix} of $w$.

A rational number can be represented as a binary word over the alphabet $\{\0,\1\}$
using its even-length continued fraction expansion as follows:
\[
\begin{array}{rccl}
W:&\Acaleven &\to&\{\0,\1\}^*\\
&[a_0;a_1,\dots,a_{2\ell-1}]
&\mapsto
&\1^{a_0}\0^{a_1}\dots \1^{a_{2\ell-2}}\0^{a_{2\ell-1}-1}.
\end{array}
\]
Here is the image of $W$ for some rational numbers:
\[
    \def\arraycolsep{4mm}
\begin{array}{c|l|ll}
    \xx      &  \CFeven(\xx) & W(\CFeven(\xx))\\[2pt]
    \hline
    &&&\\[-3mm]
    1               & [0;1]     & \1^0\0^{1-1}          &=\varepsilon   \\
    2               & [1;1]     & \1^1\0^{1-1}          &=\1            \\
    1/2             & [0;2]     & \1^0\0^{2-1}          &=\0\\
    1/3             & [0;3]     & \1^0\0^{3-1}          &=\0\0\\
    2/3             & [0;1,1,1] & \1^0\0^{1}\1^1\0^{1-1}&=\0\1\\
    3/2             & [1;2]     & \1^1\0^{2-1}          &=\1\0\\
    3               & [2;1]     & \1^2\0^{1-1}          &=\1\1
\end{array}
\]
Notice that $W([a_0;\dots,a_{2\ell-1}])$
is of length $a_0+\dots+a_{2\ell-1}-1$.

In fact, every binary word can be expressed as the image of an even-length
continued fraction expansion under the map $W$.
\begin{lemma}\label{lem:W-Acaleven-to-binarywords-bijection}
$W:\Acaleven\to\{\0,\1\}^*$ is a bijection.
\end{lemma}

\begin{proof}
    Consider the following map
\[
\begin{array}{rccl}
Y:&\{\0,\1\}^*&\to&\Acaleven\\
&w
&\mapsto
&[a_0;\dots,a_{2\ell-1}]
\end{array}
\]
    where $(a_i)_{0\leq i<2\ell}$ is the unique sequence of integers
    of even length
    such that
    \[
w= \1^{a_0}\0^{a_1}\dots \1^{a_{2\ell-2}}\0^{a_{2\ell-1}-1}
\]
with $a_0\geq 0$ and $a_i\geq1$ for every $i\neq0$.
We observe that $Y\circ W$ and $W\circ Y$ are identity mappings.
Therefore, $Y$ is the inverse map of $W$ and we conclude that
    $W:\Acaleven\to\{\0,\1\}^*$ is invertible,
and thus is a bijection.
\end{proof}

Since $\CFeven:\Q_{>0}\to\Acaleven$ is a well-known bijection
often kept implicit, we also denote the composition
$W\circ\CFeven:\Q_{>0}\to\{\0,\1\}^*$ as $W$ in this article.
We hope the reader accepts this slight abuse of notation.
Thus, we have

\begin{lemma}
    $W:\Q_{>0}\to\{\0,\1\}^*$ is a bijection.
\end{lemma}

\begin{proof}
    Follows from Lemma~\ref{lem:W-Acaleven-to-binarywords-bijection}
    since $\CFeven:\Q_{>0}\to\Acaleven$ is also a bijection.
\end{proof}

\subsection{Two involutions on rational numbers}
The inverse of rational numbers
corresponds, under the map $W$, to a natural involution
on the free monoid $\{\0,\1\}^*$ which swaps the letters.
On $\{\0,\1\}^*$, let $w \mapsto \overline{w}$ be defined as
$\overline{w} = \overline{w_1} \cdots \overline{w_k}$ if $w = w_1 \cdots w_k$
where
$\overline{\0}=\1$
and
$\overline{\1}=\0$.

\begin{lemma}\label{lem:map-W-commutes-inverse-bar}
    For every $x\in\Q_{>0}$, $\overline{W(x)}=W(x^{-1})$.
    In other words, the following diagram is commutative:
\begin{center}
\begin{tikzcd}[ampersand replacement=\&]
    \Q_{>0}
    \arrow{r}{W}
    \arrow[swap]{d}{x\mapsto x^{-1}}
	\& \{\0,\1\}^* \arrow{d}{w\mapsto \overline{w}}
    \\
    \Q_{>0} \arrow{r}{W}
    \& \{\0,\1\}^*
\end{tikzcd}
\end{center}
\end{lemma}

\begin{proof}
    Let $[a_0;\dots,a_{2\ell-1}]\in\Q_{>0}$.
    If $a_0>0$ and $a_{2\ell-1}>1$, we have
    \begin{align*}
    \overline{W[a_0;\dots,a_{2\ell-1}]}
        &=\overline{\1^{a_0}\0^{a_1}\dots \1^{a_{2\ell-2}}\0^{a_{2\ell-1}-1}}\\
        &=\0^{a_0}\1^{a_1}\dots \0^{a_{2\ell-2}}\1^{a_{2\ell-1}-1}\\
        &=\1^0\0^{a_0}\1^{a_1}\dots \0^{a_{2\ell-2}}\1^{a_{2\ell-1}-1}\0^{1-1}\\
        &=W[0;a_0,\dots,a_{2\ell-1}-1,1]\\
        &=W[0;a_0,\dots,a_{2\ell-1}]\\
        &=W[a_0;\dots,a_{2\ell-1}]^{-1}.
    \end{align*}
    Other cases ($a_0=0$ or $a_{2\ell-1}=1$) are dealt similarly and are left to the reader.
\end{proof}

Also, we define the reversal $w \mapsto \widetilde{w}$ on $\{\0,\1\}^*$ as
$\widetilde{w} = w_k \cdots w_1$ if $w = w_1 \cdots w_k$.
The reversal is an involution whose fixed points are called palindromes.
Finally, another involution on $\{\0,\1\}^*$ is $w \mapsto \widehat{w}$ defined as
$\widehat{w} = \overline{w_k} \cdots \overline{w_1}$ if $w = w_1 \cdots w_k$.
It satisfies $\widehat{w}
              =\widetilde{\overline{w}}
              =\overline{\widetilde{w}}$.
It turns out that the map $w\;\mapsto \widehat{w}$ is conjugate
to the involution $\tau:\Acaleven\to\Acaleven$ defined as
\[
    \tau([a_0;\dots,a_{2\ell-1}])=[a_{2\ell-1}-1,a_{2\ell-2},\dots,a_1,a_{0}+1].
\]

\begin{lemma}\label{lem:map-W-commutes-the-widehat}
    For every $a\in\Acaleven$, $\widehat{W(a)}=W(\tau(a))$.
    In other words, the following diagram is commutative:
\begin{center}
\begin{tikzcd}[ampersand replacement=\&]
    \Acaleven
    \arrow{r}{W}
    \arrow[swap]{d}{\tau}
	\& \{\0,\1\}^* \arrow{d}{w\;\mapsto \widehat{w}}
    \\
    \Acaleven \arrow{r}{W}
    \& \{\0,\1\}^*
\end{tikzcd}
\end{center}
\end{lemma}

\begin{proof}
    Let $[a_0;a_1,\dots,a_{2\ell-1}]\in\Acaleven$.
    We have
    \begin{align*}
    \widehat{W[a_0;a_1,\dots,a_{2\ell-1}]}
        &=\widehat{\1^{a_0}\0^{a_1}\dots \1^{a_{2\ell-2}}\0^{a_{2\ell-1}-1}}\\
        &=\overline{\0^{a_{2\ell-1}-1} \1^{a_{2\ell-2}}\dots\0^{a_1}\1^{a_0}}\\
        &=\1^{a_{2\ell-1}-1} \0^{a_{2\ell-2}}\dots\1^{a_1}\0^{a_0}\\
        &=W[a_{2\ell-1}-1;a_{2\ell-2}\dots,a_1,a_0+1].
    \end{align*}
\end{proof}

\section{Morier-Genoud--Ovsienko's \texorpdfstring{$q$}{q}-analog of rational numbers}
\label{sec:Morier-Genoud-Ovsienko-q-analog}

Recently,
Morier-Genoud--Ovsienko defined
a $q$-analog for the rational numbers using the matrices
\[
    L_q=
\left(\begin{array}{rr}
q & 0 \\
q & 1
\end{array}\right)
\qquad
\text{ and }
\qquad
    R_q=
\left(\begin{array}{rr}
q & 1 \\
0 & 1
\end{array}\right)
\]
where $q$ is some indeterminate \cite{MR4073883}.
The Morier-Genoud--Ovsienko $q$-analog of the rational number $\frac{r}{s}$
is the rational function
\[
    \left[\frac{r}{s}\right]_q
    :=
    \frac{R(q)}{S(q)}
\]
where the numerator and denominator polynomials are obtained
by replacing matrices $L$ and $R$ in \eqref{eq:CF-expansion-LR}
by their $q$-analogs $L_q$ and $R_q$
and computing the following vector
(see Proposition 4.3 and Proposition 4.5 in \cite{MR4073883}):
\begin{equation}\label{eq:CF-expansion-LqRq}
    \begin{aligned}
    \left(\begin{array}{c} R(q) \\ S(q) \end{array}\right)
        &= q^{-1} R_q^{a_0}L_q^{a_1}\cdots R_q^{a_{2\ell-2}}L_q^{a_{2\ell-1}}
                \left(\begin{smallmatrix} 1 \\ 0 \end{smallmatrix}\right)\\
        &= R_q^{a_0}L_q^{a_1}\cdots R_q^{a_{2\ell-2}}L_q^{a_{2\ell-1}-1}
                \left(\begin{smallmatrix} 1 \\ 1 \end{smallmatrix}\right).
    \end{aligned}
\end{equation}
In \eqref{eq:CF-expansion-LqRq},
it can be shown that
$R(q)$ and $S(q)$ are polynomials
and the constant coefficient of the denominator is $S(0)=1$.

\begin{example}
For example,
\[
    \frac{7}{2}
    = 3 + \frac{\displaystyle 1}{\displaystyle 2 }
    = [3;2],
    \qquad
        R_q^{3}L_q^{2-1}
         \left(\begin{array}{c} 1 \\ 1 \end{array}\right)
        =
    \left(\begin{array}{c} {q^{4} + q^{3} + 2q^{2} + 2q + 1} \\ {q + 1} \end{array}\right)
\]
and
\[
    \left[\frac{7}{2}\right]_q
    =
    \frac{q^{4} + q^{3} + 2q^{2} + 2q + 1}{q + 1}.
\]
\end{example}

Theorem 4 in \cite{MR4073883} gives a combinatorial
interpretation of the coefficients of $R(q)$ and $S(q)$
in terms of order ideals within two distinct fence posets
constructed from the partial quotients of the continued fraction expansion of
$\frac{r}{s}$.
In their theorem, it is implicitly assumed
that $a_0\neq 0$, that is, $\frac{r}{s}>1$.
However, Equation~\eqref{eq:CF-expansion-LqRq} works when $a_0=0$.
Thus, it provides a definition of $q$-analog for all
positive rational numbers including those in the interval $0<\frac{r}{s}<1$.
\begin{example}\label{ex:2-sur-7}
For example,
\[
    \frac{2}{7}
    = 0
    + \frac{1}{3
    + \frac{1}{1
    + \frac{1}{1
    }}}
    = [0;3,1,1],
    \qquad
        R_q^{0}L_q^{3}R_q^{1}L_q^{1-1}
         \left(\begin{array}{c} 1 \\ 1 \end{array}\right)
        =
    \left(\begin{array}{c} {q^{4} + q^{3}} \\ {q^4+2q^3+2q^2+q+1} \end{array}\right)
\]
and
\[
    \left[\frac{2}{7}\right]_q
    =
    \frac{q^4+q^3}{q^4+2q^3+2q^2+q+1}.
\]
\end{example}

As observed by \cite{MorierGenoud2019} and \cite{MR4073883},
the identity
\begin{equation}\label{eq:identity-q-rationals}
    \left[\frac{r}{s}+1\right]_q
    = q\left[\frac{r}{s}\right]_q + 1.
\end{equation}
allows to extend the definition to all rational numbers including negative ones.

\begin{remark}
In Example~\ref{ex:2-sur-7}, the constant term of the numerator $R(q)=q^4+q^3$ is 0.
This is not a contradiction with Corollary 1.7 (iii) of \cite{MR4073883}
that states that the constant terms of both $R(q)$ and $S(q)$ are equal to 1,
because \cite{MR4073883} implicitly assumes that $\frac{r}{s}>1$.
\end{remark}

\subsection*{An active subject}
Since the discovery of $q$-analog of rational numbers by
Morier-Genoud and Ovsienko, the subject has sparkled in many research directions.
During the preparation of this article, contributions were made that are close to this article.
The use of double dimer cover of graphs was considered in \cite{musiker_super_2025}.
Also, perfect matchings with weighted edges
and $q$-deformed Markov numbers were considered in
\cite{evans_q-deformed_2025}.
Then, an interpretation of
$q$-deformed rationals in terms of
hyperbinary partitions was proposed in
\cite{mcconville_hyperbinary_2025}.
Equivariant modular functions are considered in \cite{topkara_equivariant_2025}.
Finally, as observed by James Propp\footnote{
    \url{https://mathenchant.wordpress.com/2025/07/17/when-999-isnt-1/}},
there are three rules that allow to compute $[\xx]_q$ for any rational number $\xx\in\Q$:
    $[0]_q = 0$,
    $[\xx+1]_q = q[\xx]_q + 1$
    and
    $[\xx]_q [-\xx^{-1}]_q = -q^{-1}$.

\section{A bijection from admissible sequences to integers in an interval}
\label{sec:bijection-admissible-to-integers}

Let $\xx$ be a positive irrational number
with continued fraction expansion
$a=[a_0; a_1, a_2, \dots]$.
Let $(p_n)$ and $(q_n)$ be the sequence of numerators and denominators of the
convergents $p_n/q_n$ to $\xx$ satisfying \eqref{eq:CF-denominators}.
Let $(r_i)_{i\geq0}$ be the sequence defined as the sum of the numerator and
denominator of the convergents:
\[
    r_i=p_{i-1}+q_{i-1}.
\]
It satisfies the following recurrence:
\begin{equation}\label{eq:the-r-sequence}
    r_i = a_{i-1}r_{i-1} + r_{i-2} \quad\text{ for every } i>0,
\qquad
\text{ with }
r_{-1}=r_0=1.
\end{equation}
Let
    \begin{align*}
        \Bcal_{k}(a)
        &=
        \{(b_i)_{0\leq i< k} \mid
            (b_i)_{0\leq i< k}
            \text{ is admissible for } [a_0;\dots,a_{k-1}]
        \}
    \end{align*}
be the set of admissible sequences of length $k$.

\begin{lemma}\label{lem:cardinality-admissible-sequences}
    The set of admissible sequences of length $k$ has cardinality
    \[
        \#\Bcal_{k}(a)  = r_{k}.
    \]
\end{lemma}

\begin{proof}
    We proceed by induction.
    We have $\#\Bcal_{0}(a) = 1 = r_{0}$
    and $\#\Bcal_{1}(a) = a_1 + 1 = r_{1}$.
    Let $M=0$ if $k$ is odd or $M=a_{k-1}$ if $k$ is even.
    The following union is disjoint:
    \begin{align*}
        \Bcal_{k}(a)
        &=
        \{(b_i)_{0\leq i< k}\in\Bcal_{k}(a)
          \mid b_{k-2}=b_{k-1}=M\}
        \cup
        \{(b_i)_{0\leq i< k}\in\Bcal_{k}(a)
          \mid b_{k-1}\neq M\}.
    \end{align*}
    Therefore,
    \begin{align*}
        \#\Bcal_{k}(a)
        &=
        \#\{(b_i)_{0\leq i< k}\in\Bcal_{k}(a)
          \mid b_{k-2}=b_{k-1}=M\}\\
        &\phantom{=}+
        \#\{(b_i)_{0\leq i< k}\in\Bcal_{k}(a)
          \mid b_{k-1}\neq M\}\\
        &=
        \#\Bcal_{k-2}(a)
        + a_{k-1}\#\Bcal_{k-1}(a)\\
        &= r_{k-2} + a_{k-1}r_{k-1}
         = r_{k}.
    \end{align*}
    This concludes the proof.
\end{proof}

\begin{lemma}\label{lem:upper-bounds-on-the-sum}
    Let $a=[a_0; a_1, \dots, a_{k-1}]$ be the continued fraction expansion
    of a positive rational number.
    Let $n= \sum_{i=0}^{k-1} (-1)^{i}b_{i}r_{i}$
    where $(b_i)_{0\leq i< k}$ is an admissible
    sequence for $a$. Then,
    \[
        n <
        \begin{cases}
            r_{k-1} & \text{ if $k$ is even},\\
            r_k     & \text{ if $k$ is odd}.
        \end{cases}
    \]
\end{lemma}

\begin{proof}
    The upper bounds can be proved with the following telescoping sum.
    We have
    \begin{align*}
        n%
        &\leq \sum_{0\leq 2j <k} a_{2j}r_{2j}
        = \sum_{0\leq 2j <k} (r_{2j+1}-r_{2j-1})
        =
        \begin{cases}
            r_{k-1} - r_{-1} & \text{ if $k$ is even},\\
            r_k - r_{-1} & \text{ if $k$ is odd}.
        \end{cases}
    \end{align*}
    The strict inequality follows from the fact that $r_{-1}=1$.
\end{proof}

\begin{lemma}\label{lem:existence-rep-in-interval}
    Let $[a_0; a_1, \dots, a_{k-1}]$ be the continued fraction expansion
    of a positive rational number.
    \begin{enumerate}
        \item If $k$ is odd, then every integer $n$ in the interval
                $0\leq n< r_{k}$
                can be written as
                $
                    n = \sum_{i=0}^{k-1} (-1)^{i}b_{i}r_{i}
                $
                where $(b_i)_{0\leq i\leq k-1}$ is an admissible sequence.
        \item If $k$ is even, then every integer $n$ in the interval
                $r_{k-1}-r_{k}\leq n< r_{k-1}$
                can be written as
                $
                    n = \sum_{i=0}^{k-1} (-1)^{i}b_{i}r_{i}
                $
                where $(b_i)_{0\leq i\leq k-1}$ is an admissible sequence.
    \end{enumerate}
\end{lemma}

\begin{proof}
    The proof of existence of the admissible sequence is done by induction.
    We prove that if (1) holds for some odd $k$, then (2) hold for $k+1$,
    and that if (2) holds for some even $k$, then (1) hold for $k+1$.

    First, we prove that (1) holds for $k=1$.
    Recall that we have $r_{-1}=r_0=1$.
    Let $k=1$ and $n$ be an integer in the interval $0\leq n<r_{k}=r_1=a_0r_0+r_{-1}=a_0+1$.
    Let $b_0=n$. We have $0\leq b_0\leq a_0$. Thus, the sequence $(b_0)$ is admissible
    and satisfies
    \[
        n = b_0 = (-1)^{0}b_0r_0.
    \]

    Suppose that (1) holds for some odd integer $k\geq1$.
    We want to show that (2) holds for $k+1$.
    Let $n$ be an integer in the interval
    $r_{k}-r_{k+1}\leq n< r_{k}$.
    Let $Q$ be the quotient and $R$ be the remainder
    of the division of $n$ by $r_{k}$.
    We have
    \begin{align*}
        0
        &=    \left\lfloor
        \frac{r_{k}-1}
             {r_{k}}
            \right\rfloor
        \geq Q
        = \left\lfloor
        \frac{n}
             {r_{k}}
            \right\rfloor
        \geq
            \left\lfloor
        \frac{r_{k}-r_{k+1}}
             {r_{k}}
            \right\rfloor\\
        &=
            \left\lfloor
        \frac{r_{k}-(a_{k}r_{k}+r_{k-1})}
             {r_{k}}
            \right\rfloor
        = - a_{k}
            +
            \left\lfloor
        \frac{r_{k}-r_{k-1}}
             {r_{k}}
            \right\rfloor
        = - a_{k}.
    \end{align*}
    Let $b_{k}=-Q$ so that
    $0\leq b_{k}\leq a_{k}$.
    By definition of the division, we have
    $n=Qr_{k} + R$
    where the remainder satisfies $0\leq R<r_{k}$.
    From (1), the integer $R$
    can be written uniquely as
    $
        \sum_{i=0}^{k-1} (-1)^{i}b_{i}r_{i}
    $
    where $(b_i)_{0\leq i\leq k-1}$ is an admissible sequence.
    Thus,
    \begin{align*}
        n
        &=Qr_{k} + R
         =-b_{k}r_{k} + \sum_{i=0}^{k-1} (-1)^{i}b_{i}r_{i}
         =\sum_{i=0}^{k} (-1)^{i}b_{i}r_{i}.
    \end{align*}
    It remains to show that $(b_i)_{0\leq i\leq k}$ is an admissible sequence,
    that is, that $b_{k}=a_{k}$ implies that $b_{k-1}=a_{k-1}$.
    Suppose that $b_{k}=a_{k}$. Thus, $Q=-a_{k}$.
    We have
    \[
        r_{k}
        > R
        =
        n - Q r_{k}
        =
        n+a_{k}r_{k}
        \geq
        r_{k}-r_{k+1} +a_{k}r_{k}
        =
        r_{k}-r_{k-1}.
    \]
    Let $R'=\sum_{i=0}^{k-2} (-1)^{i}b_{i}r_{i}$.
    From Lemma~\ref{lem:upper-bounds-on-the-sum}, we have $R'<r_{k-2}$ since $k-1$ is even.
    Therefore,
    \[
        b_{k-1}r_{k-1}
        = R - R'
        > r_{k}-r_{k-1} - r_{k-2}
        = a_{k-1}r_{k-1}+r_{k-2}-r_{k-1} - r_{k-2}
        = (a_{k-1}-1)r_{k-1}.
    \]
    This implies that $b_{k-1}=a_{k-1}$ in the sum $R=\sum_{i=0}^{k-1} (-1)^{i}b_{i}r_{i}$.
    We conclude that (2) holds for $k+1$.

    Suppose that (2) holds for some even integer $k\geq1$.
    We want to show that (1) holds for $k+1$.
    Let $n$ be an integer in the interval
    $0\leq n< r_{k+1}$.
    Let $b_{k}$ be the quotient and $R$ be the remainder
    of the division of $n-(r_{k-1}-r_{k})$ by $r_{k}$.
    We have
    \begin{align*}
        0\leq b_{k} &= \left\lfloor
        \frac{n-(r_{k-1}-r_{k})}
             {r_{k}}
            \right\rfloor
        \leq
            \left\lfloor
        \frac{r_{k+1}{-}1-(r_{k-1}-r_{k})}
             {r_{k}}
            \right\rfloor\\
        &=
            \left\lfloor
        \frac{a_{k}r_{k}+r_{k-1}-(r_{k-1}-r_{k}){-}1}
             {r_{k}}
            \right\rfloor
        =
            \left\lfloor
        \frac{a_{k}r_{k}+r_{k}-1}
             {r_{k}}
            \right\rfloor
        = a_{k}.
    \end{align*}
    By definition of the division, we have
    $n-(r_{k-1}-r_{k})=b_{k}r_{k} + R$
    where the remainder satisfies $0\leq R<r_{k}$.
    Thus,
    \[
        r_{k-1}-r_{k} \leq R+ r_{k-1}-r_{k}  < r_{k-1}.
    \]
    From (2), the integer $R+ r_{k-1}-r_{k}$
    can be written uniquely as
    $
        \sum_{i=0}^{k-1} (-1)^{i}b_{i}r_{i}
    $
    where $(b_i)_{0\leq i\leq k-1}$ is an admissible sequence.
    Therefore,
    \[
        n
        = b_{k}r_{k} + R+(r_{k-1}-r_{k})
        = b_{k}r_{k} + \sum_{i=0}^{k-1} (-1)^{i}b_{i}r_{i}
        = \sum_{i=0}^{k} (-1)^{i}b_{i}r_{i}.
    \]
    It remains to show that $(b_i)_{0\leq i\leq k}$ is an admissible sequence,
    that is, that $b_{k}=0$ implies that $b_{k-1}=0$.
    Suppose that $b_{k}=0$.
    We have
    $n = -b_{k-1}r_{k-1} + R'$
    where $R'=\sum_{i=0}^{k-2} (-1)^{i}b_{i}r_{i}$.
    From Lemma~\ref{lem:upper-bounds-on-the-sum}, we have $R'<r_{k-1}$ since $k-1$ is odd.
    Therefore, using $n\geq0$,
    \[
        b_{k-1}r_{k-1} = R' - n < r_{k-1} + 0 = r_{k-1}.
    \]
    Thus, we conclude that $b_{k-1}=0$
    and $(b_i)_{0\leq i\leq k}$ is an admissible sequence.
    We conclude that (1) holds for $k+1$.
\end{proof}

\begin{table}
\[
\footnotesize
\begin{array}{c|cccc}
       &  a_0  & a_1   & a_2 \\
       &  2  &  2    & 2   \\
\hline
\hline
       &  p_{-1}/q_{-1}  & p_0/q_0  & p_1/q_1  \\
       &  1/0  &  2/1    & 5/2  \\
\hline
\hline
       &  r_0  & -r_1   & r_2 \\
       &  1  &  -3   & 7   \\
\hline
\hline
  n    &  b_0  & b_1   & b_2 \\
\hline
  0   & 0  & 0  &  0\\
  1   & 1  & 0  &  0\\
  2   & 2  & 0  &  0\\
  3   & 2  & 2  &  1\\
  4   & 0  & 1  &  1\\
  5   & 1  & 1  &  1\\
  6   & 2  & 1  &  1\\
  7   & 0  & 0  &  1\\
  8   & 1  & 0  &  1\\
  9   & 2  & 0  &  1\\
  10  & 2  & 2  &  2\\
  11  & 0  & 1  &  2\\
  12  & 1  & 1  &  2\\
  13  & 2  & 1  &  2\\
  14  & 0  & 0  &  2\\
  15  & 1  & 0  &  2\\
  16  & 2  & 0  &  2\\
\end{array}
\qquad
\begin{array}{c|cccc}
       &  a_0  & a_1   & a_2 & a_3 \\
       &  2  &  2    & 2   & 2 \\
\hline
\hline
       &  p_{-1}/q_{-1}  & p_0/q_0  & p_1/q_1   & p_2/q_2 \\
       &  1/0  &  2/1    & 5/2   & 12/5 \\
\hline
\hline
       &  r_0  & -r_1   & r_2 & -r_3 \\
       &  1  &  -3   & 7   & -17 \\
\hline
\hline
  n    &  b_0  & b_1   & b_2 & b_3 \\
\hline
  -24  &  2  &  2    & 2   & 2 \\
  -23  &  0  &  1    & 2   & 2 \\
  -22  &  1  &  1    & 2   & 2 \\
  -21  &  2  &  1    & 2   & 2 \\
  -20  &  0  &  0    & 2   & 2 \\
  -19  &  1  &  0    & 2   & 2 \\
  -18  &  2  &  0    & 2   & 2 \\
  -17  &  0  &  0    & 0   & 1 \\
  -16  &  1  &  0    & 0   & 1 \\
  -15  &  2  &  0    & 0   & 1 \\
  -14  &  2  &  2    & 1   & 1 \\
  -13  &  0  &  1    & 1   & 1 \\
  -12  &  1  &  1    & 1   & 1 \\
  -11  &  2  &  1    & 1   & 1 \\
  -10  &  0  &  0    & 1   & 1 \\
  -9   &  1  &  0    & 1   & 1 \\
  -8   &  2  &  0    & 1   & 1 \\
  -7   &  2  &  2    & 2   & 1 \\
  -6   &  0  &  1    & 2   & 1 \\
  -5   &  1  &  1    & 2   & 1 \\
  -4   &  2  &  1    & 2   & 1 \\
  -3   &  0  &  0    & 2   & 1 \\
  -2   &  1  &  0    & 2   & 1 \\
  -1   &  2  &  0    & 2   & 1 \\
  0    &  0  &  0    & 0   & 0 \\
  1    &  1  &  0    & 0   & 0 \\
  2    &  2  &  0    & 0   & 0 \\
  3    &  2  &  2    & 1   & 0 \\
  4    &  0  &  1    & 1   & 0 \\
  5    &  1  &  1    & 1   & 0 \\
  6    &  2  &  1    & 1   & 0 \\
  7    &  0  &  0    & 1   & 0 \\
  8    &  1  &  0    & 1   & 0 \\
  9    &  2  &  0    & 1   & 0 \\
  10   &  2  &  2    & 2   & 0 \\
  11   &  0  &  1    & 2   & 0 \\
  12   &  1  &  1    & 2   & 0 \\
  13   &  2  &  1    & 2   & 0 \\
  14   &  0  &  0    & 2   & 0 \\
  15   &  1  &  0    & 2   & 0 \\
  16   &  2  &  0    & 2   & 0 \\
\end{array}
\]
\caption{
    Let $a=[2;2,2]$ be the continued fraction expansion of $12/5$.  The table
    on the left shows the admissible representation $(b_0,b_1,b_2)$ of every
    integer $n$ in the interval $0 \leq n < r_3=17$.
    They satisfy $n=b_0r_0-b_1r_1+b_2r_2$.
    Let $a=[2;2,2,2]$ be the continued fraction expansion of $29/12$.
    The table on the right shows the admissible representation $(b_0,b_1,b_2,b_3)$ of every integer
    $n$ in the interval $-24=r_3-r_4 \leq n < r_3=17$.
    They satisfy $n=b_0r_0-b_1r_1+b_2r_2-b_3r_3$.
    }
\label{table:2222}
\end{table}

\begin{theorem}\label{thm:bijection-admissible-to-interval-integers}
    Let $a=[a_0; a_1, \dots, a_{k-1}]$ be a (even or odd-length) continued
    fraction expansion of a positive rational number.
    The map
    \[
    \begin{array}{rccl}
        \val_a:&\Bcal(a)&\to& \Zcal(a) \\
        &(b_i)_{0\leq i\leq k-1}&\mapsto &\sum_{i=0}^{k-1} (-1)^{i}b_{i}r_{i}
    \end{array}
    \]
    is a bijection whose range is the following interval of integers
    \[
        \Zcal(a) =
        \begin{cases}
            \Z\cap [0,r_{k})               & \text{ if $k$ is odd,}\\
            \Z\cap [r_{k-1}-r_{k},r_{k-1}) & \text{ if $k$ is even.}
        \end{cases}
    \]
\end{theorem}

\begin{proof}
    Let $\Zcal(a)=\Z\cap[0,r_{k})$ if $k$ is odd or
    $\Zcal(a)=\Z\cap[r_{k-1}-r_{k},r_{k-1})$ if $k$ is even.
    From the existence of the admissible sequence proved above
    in Lemma~\ref{lem:existence-rep-in-interval},
    the map $\val_a:\Bcal(a)\to \Zcal(a)$ is onto.
    The cardinality of the range of $\val_a$ is $\#\Zcal(a)=r_{k}$.
    From Lemma~\ref{lem:cardinality-admissible-sequences},
    we also have that the set of admissible sequences for $a$ of length $k$ is
    $\#\Bcal(a)  = r_{k}$.
    Thus, the map $\val_a:\Bcal(a)\to \Zcal(a)$ is also injective
    and it is a bijection.
\end{proof}

We denote the representation of integers in the interval $\Zcal(a)$
by the function $\rep_a:\Zcal(a)\to\Bcal(a)$
defined as the inverse of the map $\val_a$, that is,
$\rep_a = \val_a^{-1}$.

\begin{THEOREMA}
    \MainTheoremA
\end{THEOREMA}

\begin{proof}
    The statement follows from the fact that
    the map $\val_a:\Bcal(a)\to \Zcal(a)$ proposed in
    Theorem~\ref{thm:bijection-admissible-to-interval-integers}
    is a bijection.
\end{proof}

Theorem~\ref{thm:our-ostrowski-in-introduction}
is illustrated in Table~\ref{table:2222}
using the odd-length continued fraction expansion $a=[2;2,2]$ of $12/5$
and
using the even-length continued fraction expansion $a=[2;2,2,2]$ of $29/12$.
The following example illustrates the case of a convergent to the golden ratio.

\begin{example}\label{ex:negaFibonacci}
With $a=[1;1,1,1,1,1]$ of even length 6, we have
$(r_i)_i=(1, -2, 3, -5, 8, -13)$. Writing
least-significant first, we obtain the following representations
for integers in the interval $[-8, 13)$:
    \[
\begin{array}{c|c}
    n & \rep_a(n)\\
    \hline \\[-4mm]
  -8 & 111111 \\[-.5mm]
  -7 & 001111 \\[-.5mm]
  -6 & 101111 \\[-.5mm]
  -5 & 000011 \\[-.5mm]
  -4 & 100011 \\[-.5mm]
  -3 & 111011 \\[-.5mm]
  -2 & 001011
\end{array}
\qquad
\qquad
\begin{array}{c|c}
  n & \rep_a(n)\\
    \hline \\[-4mm]
  -1 & 101011 \\[-.5mm]
  0  & 000000 \\[-.5mm]
  1  & 100000 \\[-.5mm]
  2  & 111000 \\[-.5mm]
  3  & 001000 \\[-.5mm]
  4  & 101000 \\[-.5mm]
  5  & 111110
\end{array}
\qquad
\qquad
\begin{array}{c|c}
  n & \rep_a(n)\\
    \hline \\[-4mm]
  6  & 001110\\[-.5mm]
  7  & 101110\\[-.5mm]
  8  & 000010\\[-.5mm]
  9  & 100010\\[-.5mm]
  10 & 111010\\[-.5mm]
  11 & 001010\\[-.5mm]
  12 & 101010
\end{array}
    \]
\end{example}

The set of admissible sequences is naturally endowed with a partial order.
Its definition is made before Theorem~\ref{thm:bijection-order-ideals-to-integersequences}.
Additional results can be obtained on these admissible sequences
related to discrete geometry and numeration systems.

In Section~\ref{sec:convexity}, we prove that the set $\Bcal(a)$ of admissible
sequences for $a$ is a convex set of integer points (Theorem~\ref{thm:convex-hull}).
Thus $\Bcal(a)$ is the set of integer points inside of a polytope.

Using the infinite continued fraction expansion of a positive irrational number,
Theorem~\ref{thm:our-ostrowski-in-introduction} can be extended
to all positive integers using odd-length admissible sequences
and
to all integers in $\Z$ using even-length admissible sequences.
These two alternating-sign versions of Ostrowski's theorem seem to be new and
will be developed in a separate work from this article.

\section{A bijection from fence poset order ideals to admissible sequences}
\label{sec:bijection-order-ideals-admissible}

In \cite{MR4073883}, Morier-Genoud and Ovsienko
proposed a $q$-analog of rational numbers and
gave an enumerative interpretation of rational numbers larger than 1
through objects named
closures of oriented path graphs.
While making progress \cite{MR4266256} and solving
\cite{MR4499341} a conjecture of Morier-Genoud and
Ovsienko about the unimodality of the involved polynomials, the authors
used the equivalent, but maybe more common, terminology of fence posets and their
order ideals. In this section, we follow this choice.

The aim of this section is to relate the combinatorics of admissible sequences
to order ideals of fence posets.
As we shall see, we extend the definition of fence posets proposed in
\cite{MR4266256,MR4499341} in order to describe all of them, not only those
that start with an up step.
This allows to interpret simultaneously the numerator and the denominator
of a $q$-rational number on a single fence poset.
Also, this gives an enumerative interpretation of the $q$-analog of
every positive rational number without the condition $>1$ assumed in \cite{MR4073883}.

\subsection{From rational numbers to fence posets}\label{sec:rational-to-fence-poset}

A \emph{partially ordered set} $P$ (or \emph{poset}, for short)
is a set together with a binary relation denoted $\vartriangleleft$,
satisfying three axioms: reflexivity, antisymmetry and transitivity;
see \cite{MR2868112}.
If $s,t\in P$ with $s\neq t$, then we say that $t$ covers $s$ or $s$ is covered by $t$,
if $s\vartriangleleft t$ and
no element $u\in P\setminus\{s,t\}$ satisfies $s \vartriangleleft u \vartriangleleft t$.
The \emph{Hasse diagram}
of a finite poset $P$ is the graph whose vertices are the elements of $P$, whose edges are the
cover relations, and such that if $s \vartriangleleft t$ then $t$ is drawn above $s$
(i.e., with a higher vertical coordinate).

A \emph{fence poset} is a poset whose Hasse diagram is a path graph,
that is, a connected acyclic directed graph whose degree of every vertex is at
most~2.
In \cite{MR4266256}, a fence poset was associated
to every composition $(a_0,a_1,\dots,a_{k-1})$ made of positive integers,
thus forcing the first step of the fence poset to always be an up step.
The same definition was used in \cite{MR4499341},
but it was convenient for them to
``\textit{allow the first part of the composition to be zero}''
in their proof of Morier-Genoud--Ovsienko's rank-unimodality conjecture
in order to represent all fence posets.
This suggests that it may be more convenient to represent fence posets
using the continued fraction expansion of
positive rational numbers.
This is what we do in this section.

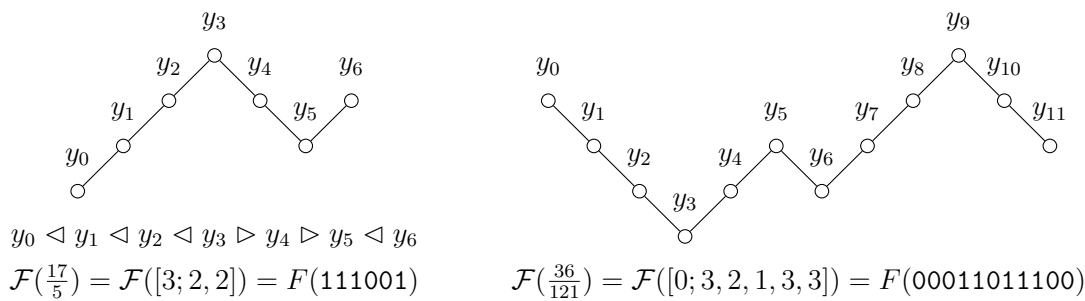
\begin{figure}[h]
    \begin{tikzpicture}[scale=.6,every node/.style={scale=0.9}]
        \foreach \x/\h in {0/0,1/1,2/2,3/3,4/2,5/1,6/2}  {
            \node[empty circled node] (v\x) at (\x,\h) {};
            \node[above=2mm]  at (\x,\h) {$y_\x$};
        }
        \draw (v0) -- (v1) ;
        \draw (v1) -- (v2) ;
        \draw (v2) -- (v3) ;
        \draw (v3) -- (v4) ;
        \draw (v4) -- (v5) ;
        \draw (v5) -- (v6) ;
        \node at (3,-1) {$y_0\vartriangleleft
                          y_1\vartriangleleft
                          y_2\vartriangleleft
                          y_3\vartriangleright
                          y_4\vartriangleright
                          y_5\vartriangleleft y_6$};
        \node at (3,-2) {$\Fcal(\frac{17}{5})=\Fcal([3;2,2])=F(\1\1\1\0\0\1)$};
    \end{tikzpicture}
    \qquad
    \begin{tikzpicture}[scale=.6,every node/.style={scale=0.9}]
        \foreach \x/\h in {0/3,1/2,2/1,3/0,4/1,5/2,6/1,7/2,8/3,9/4,10/3,11/2}  {
            \node[empty circled node] (v\x) at (\x,\h) {};
            \node[above=2mm]  at (\x,\h) {$y_{\x}$};
        }
        \draw (v0) -- (v1) ;
        \draw (v1) -- (v2) ;
        \draw (v2) -- (v3) ;
        \draw (v3) -- (v4) ;
        \draw (v4) -- (v5) ;
        \draw (v5) -- (v6) ;
        \draw (v6) -- (v7) ;
        \draw (v7) -- (v8) ;
        \draw (v8) -- (v9) ;
        \draw (v9) -- (v10) ;
        \draw (v10) -- (v11) ;
        \node at (5.5,-1) {$\Fcal(\frac{36}{121})=\Fcal([0;3,2,1,3,3])=F(\mathtt{00011011100})$};
    \end{tikzpicture}
    \caption{The fence posets %
                              $F(\1\1\1\0\0\1)$
             and %
                 $F(\mathtt{00011011100})$.}
    \label{fig:fence-posets}
\end{figure}

For every word $w=w_1\cdots w_n\in\{\0,\1\}^*$, let
$F(w)$ be the fence poset with elements $\{y_0,y_1,\dots,y_n\}$
and covering relations
\[
\begin{cases}
    y_{i-1} \vartriangleleft y_i  & \text{ if } w_i=\0,\\
    y_{i-1} \vartriangleright y_i & \text{ if } w_i=\1,
\end{cases}
\]
for every integer index $i$ with $1\leq i\leq n$,
where $\vartriangleleft$ is the order relation in $F(w)$.
Every fence poset can be expressed
as $F(w)$ for some binary word $w\in\{\0,\1\}^*$.
Thus, $F$ is a bijection from $\{\0,\1\}^*$
to the set of fence posets.
For example, the fence posets $F(\0\0\0\1\1\0)$
and $F(\mathtt{00011011100})$
are shown in Figure~\ref{fig:fence-posets}.

We let $\Fcal=F\circ W\circ\CFeven$ allowing to associate bijectively a fence poset
to every positive rational number:
\[
    \Fcal:\Q_{>0}\to\{\text{fence posets}\}.
\]
Figure~\ref{fig:fence-posets} shows a fence poset $\Fcal(\xx)$ associated with
a rational number $\xx>1$ and another associated with a rational number
$\xx<1$.

The map $\Fcal$ maps the Stern-Brocot tree
of positive rational numbers \cite[Section 4.5]{MR1397498}
to the binary tree of fence posets where the parent relation
is defined by the prefix relation; see
Figure~\ref{fig:Stern-Brocot-to-fences-tree}.
\begin{figure}
\begin{center}
\includegraphics[width=\linewidth]{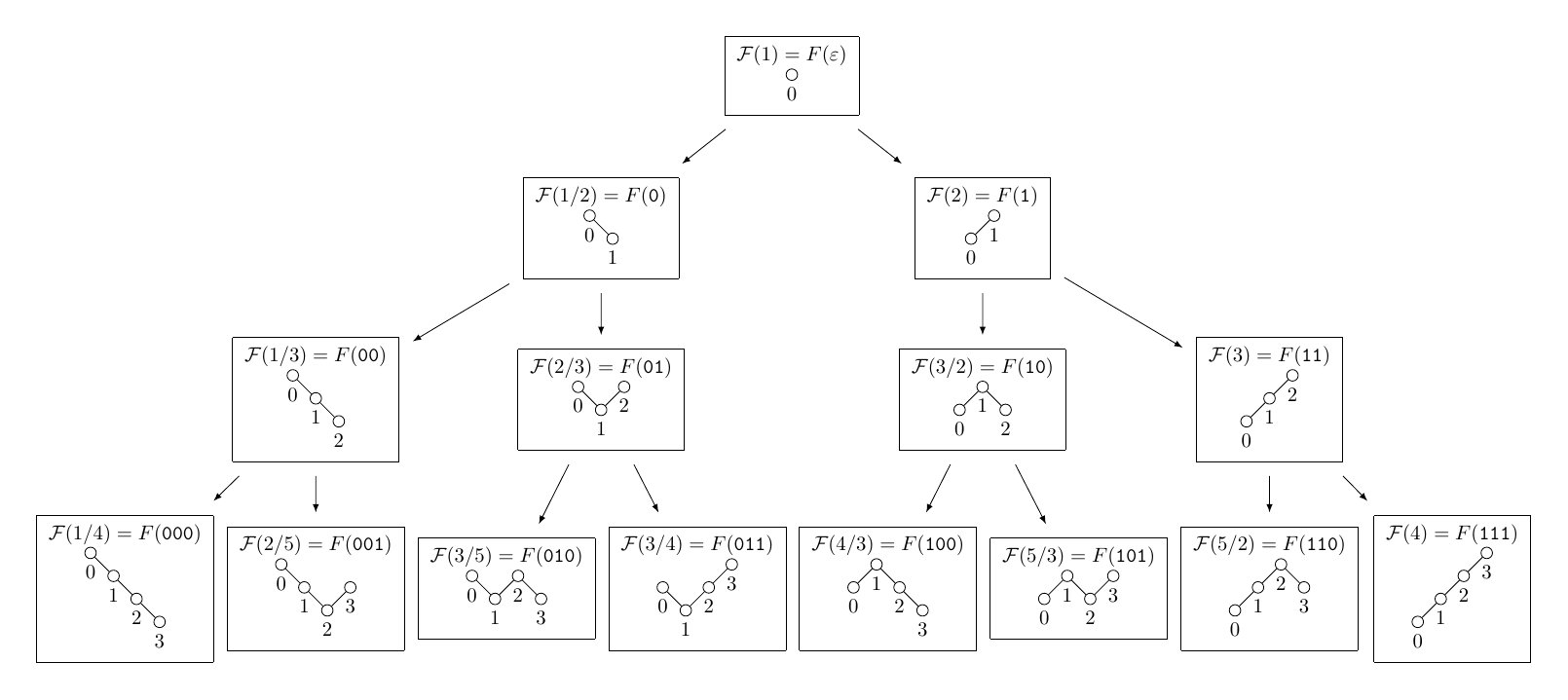}
\end{center}
    \caption{The map $\Fcal$ maps the Stern-Brocot tree of positive
    rational numbers on the binary tree of fence posets ordered by the prefix relation.}
    \label{fig:Stern-Brocot-to-fences-tree}
\end{figure}

The reader may refer to Figure~\ref{fig:the-big-picture}
to locate this bijection in the big picture.
Notice that the Hasse diagram of the fence poset $\Fcal([a_0;\dots,a_{2\ell-1}])$
has $a_0+\dots+a_{2\ell-1}-1$ edges
and $a_0+\dots+a_{2\ell-1}$ vertices.
Also, for every $x\in\Q_{>0}$, the poset $\Fcal(\xx)$ starts with a down step
if and only if $\xx<1$.

Following Lemma~\ref{lem:map-W-commutes-inverse-bar},
the fence posets $\Fcal(\xx)$ and $\Fcal(\xx^{-1})$ are mirror image
of one another under the horizontal axis.
Mirror image under the vertical axis can be understood
as another involution on the positive rational numbers;
see Lemma~\ref{lem:map-W-commutes-the-widehat}.

\begin{remark}
The construction of a fence poset from the even-length continued fraction
expansion of a rational number was already proposed in \cite{MR4588256},
but it is restricted to rational numbers larger than one.
This choice is consistent with the combinatorial interpretation of $q$-rational numbers
from $k$-vertex closures of two oriented path graphs made in \cite{MR4073883}.
On the contrary, the fence poset $\Fcal(x)$ defined here works for all
rational numbers $\xx>0$.
\end{remark}

\subsection{Fence poset order ideals}

A (lower) \emph{order ideal}
of a poset $P$ with binary relation $\vartriangleleft$
is a subset $I$ of $P$ such that if $t\in I$
and $s \vartriangleleft t$, then $s\in I$ \cite{MR2868112}.
The set of all order ideals of $P$, ordered by inclusion, forms a poset denoted
$J(P)$ which is also a \emph{distributive lattice}.
The lattice $J(P)$ is graded by
the cardinality, called \emph{rank}, of the order ideals.
The number of order ideals of each rank is encoded into a \emph{rank-generating polynomial}
\[
    \sum_{I\in J(P)} q^{|I|}
\]
where $|I|$ is the cardinality (rank) of the order ideal $I$.

In the context of fence posets,
we let $\Jcal:=J\circ\Fcal = J\circ F\circ W\circ\CFeven$.
Thus, if $\xx>0$ is a positive rational number, then
$\Jcal(\xx)$ is the set of lower order ideals of the fence poset $\Fcal(\xx)$
ordered by inclusion;
see these maps in the big picture in Figure~\ref{fig:the-big-picture}.
It was conjectured in \cite{MR4073883} and
proved in \cite{MR4499341} that the rank polynomials of the distributive
lattices of lower ideals of fence posets are unimodal.

For every $\xx\in\Q_{>0}$,
the set $\Jcal(\xx)$ of lower order ideals of the fence poset $\Fcal(\xx)$
is naturally partitioned as the disjoint union
$\Jcal(\xx)=\Jcal^\bullet(\xx)\cup\Jcal^\circ  (\xx)$
according to whether the order ideal contains the first element or not,
that is,
\begin{align*}
    \Jcal^\bullet(\xx) &= \{I\in\Jcal(\xx) \mid y_0\in    I\},\\
    \Jcal^\circ  (\xx) &= \{I\in\Jcal(\xx) \mid y_0\notin I\}.
\end{align*}
For example, an order ideal in the set $\Jcal^\bullet(399/121)$ is
shown in Figure~\ref{fig:ideal-containing-zero}.

\begin{figure}[h]
\begin{center}
    \begin{tikzpicture}[scale=.5,every node/.style={scale=0.9}]
        \foreach \x/\h in {2/2,3/3,4/2,5/1,10/2,11/3,12/4}  {
            \node[empty circled node] (v\x) at (\x,\h) {};
            \node[above=1mm]  at (\x,\h) {${\x}$};
        }
        \foreach \x/\h in {0/0,1/1,6/0,7/1,8/2,9/1,13/3,14/2}  {
            \node[plein circled node] (v\x) at (\x,\h) {};
            \node[above=1mm]  at (\x,\h) {${\x}$};
        }
        \draw (v0)  -- (v1)  ;
        \draw (v1)  -- (v2)  ;
        \draw (v2)  -- (v3)  ;
        \draw (v3)  -- (v4)  ;
        \draw (v4)  -- (v5)  ;
        \draw (v5)  -- (v6)  ;
        \draw (v6)  -- (v7)  ;
        \draw (v7)  -- (v8)  ;
        \draw (v8)  -- (v9)  ;
        \draw (v9)  -- (v10) ;
        \draw (v10) -- (v11) ;
        \draw (v11) -- (v12) ;
        \draw (v12) -- (v13) ;
        \draw (v13) -- (v14) ;
    \end{tikzpicture}
\end{center}
    \caption{The order ideal $\{0,1,6,7,8,9,13,14\}$
    of cardinality $8$
    in the distributive lattice $\Jcal^\bullet(399/121)=\Jcal^\bullet([3,3,2,1,3,3])$
    is shown with dark vertices.
    The elements $y_i$ of the poset are identified with their indices~$i$.
    }
\label{fig:ideal-containing-zero}
\end{figure}
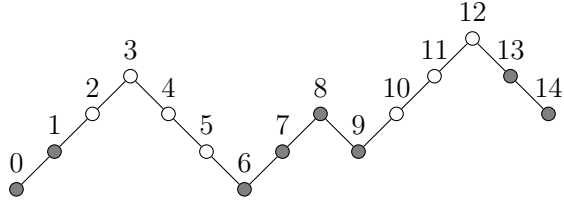

\subsection{A bijection from fence poset order ideals to admissible sequences}

The relation between order ideals of fence posets
and binary strings in $\{\0,\1\}^*$ with no two consecutive $\1$'s was noticed before
for the particular fence poset $\Fcal(\1\0\1\0\1\dots)$
whose number of order ideals is a Fibonacci number \cite{MR1948779}.
In this section, we extend this observation to all fence posets
using a partition of the fence poset into chains
as in \cite[Section 3.5]{MR2868112}.

We prove the existence of a bijection from
the set of lower order ideals of a fence poset
to the set of admissible sequences
such that the cardinality of an order ideal
correspond to the 1-norm of an admissible sequence.

Let $a=[a_0;a_1,\dots,a_{2\ell-1}]$
be the even-length continued fraction expansion of a positive
rational number $x$.
For every $i$, with $0\leq i< 2\ell$,
we define the interval of integers
\[
    C_i(a) = \left\{n\in\N
        \;\middle\vert\;
        \sum_{k=0}^{i-1}a_k
        \leq n <
        \sum_{k=0}^{i}a_k\right\}
\]
which is a chain in the fence poset $\Fcal(\xx)$.
The family $\{C_i(a)\}_{0\leq i< 2\ell}$ is made of disjoint consecutive
intervals of integers and its union is the set
\[
    \bigcup_{0\leq i< 2\ell} C_i(a)
    = \{0,1,\dots,a_0+a_1+\dots+a_{2\ell-1}-1\}.
\]
To simplify the notation,
we identify the elements $\{y_0,\dots,y_{a_0+a_1+\dots+a_{2\ell-1}-1}\}$ 
of the fence poset $\Fcal(\xx)$ with their integer indices
$\{0,1,\dots,a_0+a_1+\dots+a_{2\ell-1}-1\}$.
Counting the number of elements of an order ideal of $\Fcal(\xx)$
in each of the intervals of the family
$\{C_i(a)\}_{0\leq i< 2\ell}$
leads to a function $\Psi:\Jcal(\xx)\to\Bcal(a)$ defined as
\[
\begin{array}{rccl}
    \Psi:&\Jcal(\xx) & \to & \Bcal(a)\\
    &I & \mapsto &
    \left(\# I\cap C_i(a)\right)_{0\leq i< 2\ell}.
\end{array}
\]
An example of the bijection is shown in Figure~\ref{fig:order-ideal-to-admissible-sequence}.

\begin{figure}[h]
\begin{center}
\newcommand{\drawboxaibiUP}[6]{
        \draw[rounded corners,dashed,very thick,fill=black!4] [yshift=2mm](#2-\dx,#3)
                          -- ++ (\dx,-\dy)
                          -- ++ (#4-1+\dx,#4-1+\dy)
                          -- ++ (-\dx,\dy) -- cycle;
        \node[above] at (#2+.5*#4-.5,-2.)   {$a_{#1}=#4$};
        \node[below] at (#2+.5*#4-.5,-.5) {$b_{#1}=#5$};
}
\newcommand{\drawboxaibiDOWN}[6]{
        \draw[rounded corners,dashed,very thick,fill=black!4] [yshift=2mm](#2,#3+\dy)
                          -- ++ (-\dx,-\dy)
                          -- ++ (#4-1+\dx,-#4+1-\dy)
                          -- ++ (\dx,\dy) -- cycle;
        \node[above] at (#2+.5*#4-.5,-2)   {$a_{#1}=#4$};
        \node[below] at (#2+.5*#4-.5,-.5) {$b_{#1}=#5$};
}
        \def\dx{.8}
        \def\dy{.8}
    \begin{tikzpicture}[scale=.8,every node/.style={scale=0.9}]
        \node at (7,6) {$
        \begin{array}{c|cccccc}
            i      & 0        & 1        & 2        & 3        & 4        & 5\\
            \hline
            C_i(a) & \{0,1,2\}& \{3,4,5\}& \{6,7\}& \{8\}& \{9,10,11\}& \{12,13,14\}\\
        \end{array}$};
        \drawboxaibiUP  {0}{0} {0}{3}{2}{\{0,1,2\}}
        \drawboxaibiDOWN{1}{3} {3}{3}{0}{\{3,4,5\}}
        \drawboxaibiUP  {2}{6} {0}{2}{2}{\{6,7\}}
        \drawboxaibiDOWN{3}{8} {2}{1}{1}{\{8\}}
        \drawboxaibiUP  {4}{9} {1}{3}{1}{\{9,10,11\}}
        \drawboxaibiDOWN{5}{12}{4}{3}{2}{\{12,13,14\}}
        \foreach \x/\h in {2/2,3/3,4/2,5/1,10/2,11/3,12/4}  {
            \node[empty circled node] (v\x) at (\x,\h) {};
            \node[above=1mm]  at (\x,\h) {${\x}$};
        }
        \foreach \x/\h in {0/0,1/1,6/0,7/1,8/2,9/1,13/3,14/2}  {
            \node[plein circled node] (v\x) at (\x,\h) {};
            \node[above=1mm]  at (\x,\h) {${\x}$};
        }
        \draw (v0) -- (v1)   ;
        \draw (v1) -- (v2)   ;
        \draw (v2) -- (v3)   ;
        \draw (v3) -- (v4)   ;
        \draw (v4) -- (v5)   ;
        \draw (v5) -- (v6)   ;
        \draw (v6) -- (v7)   ;
        \draw (v7) -- (v8)   ;
        \draw (v8) -- (v9)   ;
        \draw (v9) -- (v10)  ;
        \draw (v10) -- (v11) ;
        \draw (v11) -- (v12) ;
        \draw (v12) -- (v13) ;
        \draw (v13) -- (v14) ;
    \end{tikzpicture}
\end{center}
    \caption{The order ideal $I=\{0,1,6,7,8,9,13,14\}$
         in $\Jcal^\bullet(a)$
         corresponding to an admissible sequence
         $\Psi(I)=b=(2,0,2,1,1,2)$
         for $a=[3;3,2,1,3,3]$.}
\label{fig:order-ideal-to-admissible-sequence}
\end{figure}
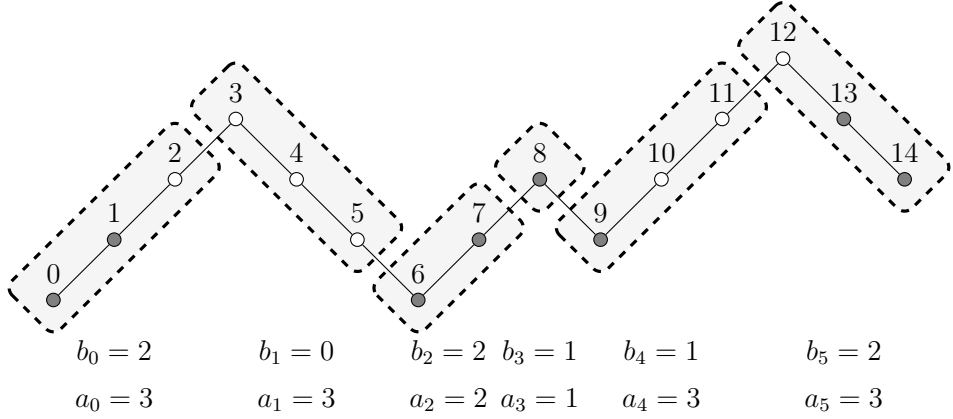

Let $a=[a_0; a_1, \dots, a_{k-1}]$ be the (even or odd-length) continued
fraction expansion of a positive rational number.
The set $\Bcal(a)$ of admissible sequences for $a$
is naturally equipped with a partial order
$(\Bcal(a), \leq)$ defined as follows:
\[
     (b_i)_{0\leq i\leq k-1}
     \leq
     (b_i')_{0\leq i\leq k-1}
    \quad
    \text{ if and only if }
    \quad
    b_i\leq b_i'
    \text{ for every $i$ with }
    0\leq i\leq k-1.
\]
The rank function of the graded poset $(\Bcal(a), \leq)$ is the 1-norm of the
admissible sequences.

\begin{theorem}\label{thm:bijection-order-ideals-to-integersequences}
    Let $x\in\Q_{>0}$ whose
even-length continued fraction expansion is $a=[a_0;a_1,\dots,a_{2\ell-1}]$.
    The map $\Psi:(\Jcal(\xx),\subseteq)\to(\Bcal(a),\leq)$ is an order-preserving bijection from
the set of order ideals of the fence poset $\Fcal(\xx)$
to the set of admissible sequences for $a$
such that:
\begin{itemize}
\item $I\in \Jcal^\bullet(\xx)$ if and only if $\Psi(I)\in \Bcal^\bullet(a)$;
\item $I\in \Jcal^\circ(\xx)$  if and only if $\Psi(I)\in \Bcal^\circ(a)$;
\item for every $I\in \Jcal(\xx)$, we have $\size{I} = \Vert\Psi(I)\Vert_1$.
\end{itemize}
\end{theorem}

\begin{proof}
    First we prove that $\Psi$ is \textbf{well-defined}.
    Let $I\in\Jcal(\xx)$ be an order ideal and $b=\Psi(I)$.
    For every integer $i$, with $0\leq i< 2\ell$, we have
    \[
        0
        \leq
    b_i=\# I\cap C_i(a)
    \leq \# C_i(a)
      = a_i.
    \]
    If $i$ is odd and $b_i = a_i$, then $C_i(a)\subset I$.
    In particular, $\sum_{k=0}^{i-1}a_k\in I$.
    This forces the previous interval to be a subset of $I$ as well, that is,
    $C_{i-1}(a)\subset I$.
    Thus, $b_{i-1} =\# I\cap C_{i-1}(a) =\# C_{i-1}(a) =a_{i-1}$.
    If $i$ is even and $b_i = 0$, $C_i(a)\cap I=\varnothing$.
    In particular, $\sum_{k=0}^{i-1}a_k\notin I$.
    This forces all of the elements of the previous interval to
    not be in $I$ as well, that is,
    $C_{i-1}(a)\cap I=\varnothing$. Thus, $b_{i-1}=\# I\cap C_{i-1}(a) = 0$.
    We conclude that $b=\Psi(I)$ is admissible for $a$
    and that the map $\Psi$ is well-defined.

    \textbf{$\Psi$ is injective.}
    Let $I,I'\in \Jcal^\bullet(\xx)$ with $I\neq I'$.
    From the definition of order ideal,
    for every integer $i$,
    the four sets
    $I\cap C_i(a)$,
    $I'\cap C_i(a)$,
    $C_i(a)\setminus I$ and
    $C_i(a)\setminus I'$ are intervals of consecutive integers.
    Since $I$ and $I'$ are distinct order ideals,
    there exists
    an integer $i$, with $0\leq i< 2\ell$, such that
    $I\cap C_i(a) \neq I'\cap C_i(a)$.
    This is only possible if the cardinalities of the four sets do not match.
    In particular, $\# I\cap C_i(a) \neq \# I'\cap C_i(a)$.
    We conclude that $\Psi(I)\neq\Psi(I')$ and $\Psi$ is injective.

    \textbf{$\Psi$ is surjective.}
    Let $(b_i)_{0\leq i< 2\ell}$ be
    an admissible sequence for $a$.
    Let
    \begin{align*}
        I = &\left(\bigcup_{0\leq i< \ell} C_{2i+1}(a)\cap (C_{2i+1}(a)+(a_{2i+1}-b_{2i+1}))\right)
            \cup\\
          &\left(\bigcup_{0\leq i< \ell} C_{2i}(a)\cap (C_{2i}(a)-(a_{2i}-b_{2i}))\right)
    \end{align*}
    be the subset made of the union of the
    $b_i$ left-most (right-most, resp.) elements of $C_i(a)$ when $i$ is even (odd, resp.).
    Since $b$ is admissible, we have that $I$ is a well-defined order ideal.
    By construction, we have $\Psi(I)=b$. Thus $\Psi$ is surjective.

    \textbf{$\Psi$ is order-preserving.}
    Let $I, I'\in \Jcal(\xx)$ be two order ideals.
    Let $(b_i)_{0\leq i< 2\ell}=\Psi(I)$
    and $(b_i')_{0\leq i< 2\ell}=\Psi(I')$.
    We have
    $I\subseteq I'$
    if and only if
    $I\cap C_i(a)\subseteq I'\cap C_i(a)$ for every index $i$ such that $0\leq i< 2\ell$
    if and only if
    $\# I\cap C_i(a)\leq \# I'\cap C_i(a)$ for every index $i$ such that $0\leq i< 2\ell$
    if and only if
    $b_i\leq b_i'$ for every index $i$ such that $0\leq i< 2\ell$
    if and only if
    $(b_i)_{0\leq i< 2\ell}\leq (b_i')_{0\leq i< 2\ell}$
    if and only if
    $\Psi(I)\subset\Psi(I')$.

    \textbf{Items 1 and 2.}
    Let $I\in \Jcal^\bullet(\xx)$, that is $0\in I$.
    We need to distinguish two cases whether $a_0=0$ or not.
    If $a_0>0$, $0\in I$ implies that $b_0>0$, so that
    $\Psi(I)\in \Bcal^\bullet(a)$.
    If $a_0=0$, $0\in I$ implies that $b_0=0$, $C_1(a)\subset I$ and $0<b_1=a_1$, so that
    $\Psi(I)\in \Bcal^\bullet(a)$.

    Let $I\in \Jcal^\circ(\xx)$, that is $0\notin I$.
    We need to distinguish two cases whether $a_0=0$ or not.
    If $a_0>0$, $0\notin I$ implies that $b_0=0$, so that
    $\Psi(I)\in \Bcal^\circ(a)$.
    If $a_0=0$, $0\notin I$ implies that $b_0=0$, $C_1(a)\not\subset I$ and $0\leq b_1<a_1$, so that
    $\Psi(I)\in \Bcal^\circ(a)$.

    This is sufficient to prove
    both Items 1 and 2 since
    $\Bcal(a)=\Bcal^\bullet(a)\cup \Bcal^\circ(a)$ is a disjoint union.

    \textbf{Item 3.}
    Let $I\in \Jcal(\xx)$ and $(b_i)_{0\leq i< 2\ell}=\Psi(I)$. We have
    \[
        \Vert\Psi(I)\Vert_1
        = \sum_{i=0}^{2\ell-1} b_i
        = \sum_{i=0}^{2\ell-1} \# I\cap C_i(a)
        = \# I\cap \left(\cup_{i=0}^{2\ell-1} C_i(a)\right)
        = \# I = \size{I},
    \]
    since the intervals $(C_i(a))_{0\leq i< 2\ell}$ are disjoint
    and $I\subset \cup_{i=0}^{2\ell-1} C_i(a)$.
\end{proof}

\begin{corollary}\label{cor:poset-isomorphism-order-ideals-to-admissible}
Let $x\in\Q_{>0}$ whose even-length continued fraction expansion is $a=[a_0;a_1,\dots,a_{2\ell-1}]$.
The posets
    $(\Jcal(\xx), \subseteq)$
    and
    $(\Bcal(a), \leq)$ are isomorphic.
\end{corollary}

\begin{proof}
    It follows from the order-preserving bijection proved in
    Theorem~\ref{thm:bijection-order-ideals-to-integersequences}.
\end{proof}

\section{Proof of
Theorem~\ref{thm:1-norm-statistics} and
Theorem~\ref{thm:nice-formula-for-order-ideals}}
\label{sec:proof-of-Theorems-B-and-D}

In this section, we introduce an equivalent condition for a pair of polynomial
functions to be counted by $q$-analog of continued fraction expansions.
We use this condition to count the number of order ideal of a fence poset
keeping track of the cardinality statistics,
thus proving Theorem~\ref{thm:nice-formula-for-order-ideals}.
Using the bijection defined in Theorem~\ref{thm:bijection-order-ideals-to-integersequences},
we deduce Theorem~\ref{thm:1-norm-statistics}.

First, we define the homomorphism
\[
    \begin{array}{rccl}
        \nu_q: &\{\0,\1\}^* &\to& \GL_2(\Z[q^{\pm1}])\\
        &\0 &\mapsto& L_q,\\
        &\1 &\mapsto& R_q.
    \end{array}
\]
Reusing the involution $w\mapsto \widehat{w}$
defined for binary words in Section~\ref{sec:binary-encoding-of-QQ},
we have the following two results.

\begin{lemma}\label{lem:identity-w-hat}
For every word $w\in\{\0,\1\}^*$,
the homomorphism $\nu_q$ satisfies
\begin{equation}\label{eq:nu_q_hat_property}
    \nu_q(w)
    \left(\begin{smallmatrix} 1 & 0 \\ 0 & q \end{smallmatrix}\right)
        =
    \left(\begin{smallmatrix} 1 & 0 \\ 0 & q \end{smallmatrix}\right)
        \nu_q(\widehat{w})^T
\end{equation}
where $\widehat{w}=\widetilde{\overline{w}}$.
\end{lemma}

\begin{proof}
    The matrix
    $\left(\begin{smallmatrix} 1 & 0 \\ 0 & q \end{smallmatrix}\right)$
    conjugates the matrix $R_q$ into the transpose of $L_q$
    and conjugates the matrix $L_q$ into the transpose of $R_q$:
\begin{align*}
    \left(\begin{smallmatrix} 1 & 0 \\ 0 & q \end{smallmatrix}\right)^{-1}
        R_q
    \left(\begin{smallmatrix} 1 & 0 \\ 0 & q \end{smallmatrix}\right)
        &= L_q^T,\\
    \left(\begin{smallmatrix} 1 & 0 \\ 0 & q \end{smallmatrix}\right)^{-1}
        L_q
    \left(\begin{smallmatrix} 1 & 0 \\ 0 & q \end{smallmatrix}\right)
        &= R_q^T.
\end{align*}
The conclusion follows.
\end{proof}

    \begin{lemma}\label{lem:sufficient-conditions-qcontinued-fraction}
        Let $a=[a_0;a_1,\dots,a_{2\ell-1}]$ be the even-length continued fraction expansion
        of a positive rational number.
        Then
        \begin{equation*}
                \nu_q(W(a))
            \left(\begin{smallmatrix} q \\ q \end{smallmatrix}\right)
            =
            R_q^{a_0}L_q^{a_1}\cdots R_q^{a_{2\ell-2}}L_q^{a_{2\ell-1}}
            \left(\begin{smallmatrix} 1 \\ 0 \end{smallmatrix}\right).
        \end{equation*}
    \end{lemma}

\begin{proof}
    We have $W(a)=\1^{a_0}\0^{a_{1}}\cdots \1^{a_{2\ell-2}}\0^{a_{2\ell-1}-1}$.
    Thus,
    \begin{align*}
        \nu_q(W(a))
        \left(\begin{smallmatrix} q \\ q \end{smallmatrix}\right)
        &=
        \nu_q(\1^{a_0}\0^{a_{1}}\cdots \1^{a_{2\ell-2}}\0^{a_{2\ell-1}-1})
           L_q
        \left(\begin{smallmatrix} 1 \\ 0 \end{smallmatrix}\right)\\
        &=
        R_q^{a_0}L_q^{a_1}\cdots R_q^{a_{2\ell-2}}L_q^{a_{2\ell-1}-1}L_q
        \left(\begin{smallmatrix} 1 \\ 0 \end{smallmatrix}\right)\\
        &=
        R_q^{a_0}L_q^{a_1}\cdots R_q^{a_{2\ell-2}}L_q^{a_{2\ell-1}}
        \left(\begin{smallmatrix} 1 \\ 0 \end{smallmatrix}\right).
    \end{align*}
\end{proof}

Then, we have the following two more important results.

    \begin{lemma}\label{lem:sufficient-conditions-to-be-described-by-Fq}
        Let
        $X:\{\0,\1\}^*\to\Z[q]$ and
        $Y:\{\0,\1\}^*\to\Z[q]$ be polynomial functions such that
        \begin{equation*}%
        \begin{aligned}
            &X(\1 w) =qX(w)+qY(w),
            &&X(\0 w) =qX(w),\\
            &Y(\1 w)  =Y(w),
            &&Y(\0 w)  =X(w)+Y(w)
        \end{aligned}
        \end{equation*}
        for every $w\in\{\0,\1\}^*$.
        Then
        \begin{equation}\label{eq:XY-in-terms-of-Fq}
            \left(\begin{array}{r}
            X(w)\\
            Y(w)
            \end{array}\right)
            =
            \nu_q(\widehat{w})^T
            \left(\begin{array}{r}
                X(\varepsilon)\\
                Y(\varepsilon)
            \end{array}\right).
        \end{equation}
    \end{lemma}

    \begin{proof}
        The proof is done by induction.
        It holds for the empty word $w=\varepsilon$.
        Suppose that \eqref{eq:XY-in-terms-of-Fq} holds
        for every $w\in\{\0,\1\}^k$ for some integer $k\geq0$.
        We have
    \begin{align*}
        \left(\begin{array}{r}
        X(\1 w)\\
        Y(\1 w)
        \end{array}\right)
        &=
        \left(\begin{array}{r}
        qX(w)+qY(w)\\
        Y(w)
        \end{array}\right)
        = \left(\begin{array}{rr} q & q\\0 & 1 \end{array}\right)
        \left(\begin{array}{r}
        X(w)\\
        Y(w)
        \end{array}\right)\\
        &= L_q^T
          \nu_q(\widehat{w})^T
          \left(\begin{array}{r} X(\varepsilon)\\ Y(\varepsilon) \end{array}\right)
        = \nu_q(\widehat{w}\0)^T
          \left(\begin{array}{r} X(\varepsilon)\\ Y(\varepsilon) \end{array}\right)
        = \nu_q(\widehat{\1 w})^T
          \left(\begin{array}{r} X(\varepsilon)\\ Y(\varepsilon) \end{array}\right).
    \end{align*}
    Also,
    \begin{align*}
        \left(\begin{array}{r}
        X(\0 w)\\
        Y(\0 w)
        \end{array}\right)
        &=
        \left(\begin{array}{r}
        qX(w)\\
        X(w)+Y(w)
        \end{array}\right)
        = \left(\begin{array}{rr} q & 0\\1 & 1 \end{array}\right)
        \left(\begin{array}{r}
        X(w)\\
        Y(w)
        \end{array}\right)\\
        &= R_q^T
          \nu_q(\widehat{w})^T
          \left(\begin{array}{r} X(\varepsilon)\\ Y(\varepsilon) \end{array}\right)
        = \nu_q(\widehat{w}\1)^T
          \left(\begin{array}{r} X(\varepsilon)\\ Y(\varepsilon) \end{array}\right)
        = \nu_q(\widehat{\0 w})^T
          \left(\begin{array}{r} X(\varepsilon)\\ Y(\varepsilon) \end{array}\right).
    \end{align*}
    \end{proof}

    \begin{proposition}\label{prop:equivalent-conditions-Fq}
        Let
        $X:\{\0,\1\}^*\to\Z[q]$ and
        $Y:\{\0,\1\}^*\to\Z[q]$ be polynomial functions.
        We have
        \begin{equation*}
            \left(\begin{array}{r}
                X(w)\\
                Y(w)
            \end{array}\right)
            =
            \left(\begin{smallmatrix} 1 & 0 \\ 0 & q \end{smallmatrix}\right)^{-1}
            \nu_q(w)
            \left(\begin{array}{cc} q \\ q \end{array}\right)
        \end{equation*}
        for every $w\in\{\0,\1\}^*$
        if and only if
        $X(\varepsilon)=q$, $Y(\varepsilon)=1$ and
        \begin{equation}\label{eq:hypothesis-on-X-Y}
        \begin{aligned}
            &X(\1 w) =qX(w)+qY(w),
            &&X(\0 w) =qX(w),\\
            &Y(\1 w)  =Y(w),
            &&Y(\0 w)  =X(w)+Y(w)\\
        \end{aligned}
        \end{equation}
        for every $w\in\{\0,\1\}^*$.
    \end{proposition}

\begin{proof}
    Let $w\in\{\0,\1\}^*$.
    First, we assume that $X(\varepsilon)=q$, $Y(\varepsilon)=1$ and
    \eqref{eq:hypothesis-on-X-Y} holds.
    Using Lemma~\ref{lem:sufficient-conditions-to-be-described-by-Fq},
    we have
    \begin{align*}
            \left(\begin{array}{r}
                X(w)\\
                Y(w)
            \end{array}\right)
        &\overset{\eqref{eq:XY-in-terms-of-Fq}}{=}
        \nu_q(\widehat{w})^T
        \left(\begin{array}{r} X(\varepsilon)\\ Y(\varepsilon) \end{array}\right)\\
            &=
        \left(\begin{smallmatrix} 1 & 0 \\ 0 & q \end{smallmatrix}\right)^{-1}
        \left(\begin{smallmatrix} 1 & 0 \\ 0 & q \end{smallmatrix}\right)
        \nu_q(\widehat{w})^T
        \left(\begin{smallmatrix} q \\ 1 \end{smallmatrix}\right)\\
        &\overset{\eqref{eq:nu_q_hat_property}}{=}
        \left(\begin{smallmatrix} 1 & 0 \\ 0 & q \end{smallmatrix}\right)^{-1}
        \nu_q(w)
        \left(\begin{smallmatrix} 1 & 0 \\ 0 & q \end{smallmatrix}\right)
        \left(\begin{smallmatrix} q \\ 1 \end{smallmatrix}\right)\\
        &=
        \left(\begin{smallmatrix} 1 & 0 \\ 0 & q \end{smallmatrix}\right)^{-1}
        \nu_q(w)
        \left(\begin{smallmatrix} q \\ q \end{smallmatrix}\right).
    \end{align*}

    Conversely,
    $
        \left(\begin{smallmatrix}
            X(\varepsilon)\\
            Y(\varepsilon)
        \end{smallmatrix}\right)
        =
        \left(\begin{smallmatrix} 1 & 0 \\ 0 & q \end{smallmatrix}\right)^{-1}
        \nu_q(\varepsilon)
        \left(\begin{smallmatrix} q \\ q \end{smallmatrix}\right)
        =
        \left(\begin{smallmatrix} q \\ 1 \end{smallmatrix}\right)
    $.
    Also,
    \begin{align*}
        \left(\begin{array}{r}
            X(\1 w)\\
            Y(\1 w)
        \end{array}\right)
        &=
        \left(\begin{smallmatrix} 1 & 0 \\ 0 & q \end{smallmatrix}\right)^{-1}
        \nu_q(\1w)
        \left(\begin{smallmatrix} q \\ q \end{smallmatrix}\right)
        =
        \left(\begin{smallmatrix} 1 & 0 \\ 0 & q \end{smallmatrix}\right)^{-1}
        R_q
        \nu_q(w)
        \left(\begin{smallmatrix} q \\ q \end{smallmatrix}\right)\\
        &=
        L_q^T
        \left(\begin{smallmatrix} 1 & 0 \\ 0 & q \end{smallmatrix}\right)^{-1}
        \nu_q(w)
        \left(\begin{smallmatrix} q \\ q \end{smallmatrix}\right)
        =
        L_q^T
        \left(\begin{array}{c}
            X(w)\\
            Y(w)
        \end{array}\right)
        =
        \left(\begin{array}{c}
            qX(w)+qY(w)\\
            Y(w)
        \end{array}\right)
    \end{align*}
    and
    \begin{align*}
        \left(\begin{array}{r}
            X(\0 w)\\
            Y(\0 w)
        \end{array}\right)
        &=
        \left(\begin{smallmatrix} 1 & 0 \\ 0 & q \end{smallmatrix}\right)^{-1}
        \nu_q(\0w)
        \left(\begin{smallmatrix} q \\ q \end{smallmatrix}\right)
        =
        \left(\begin{smallmatrix} 1 & 0 \\ 0 & q \end{smallmatrix}\right)^{-1}
        L_q
        \nu_q(w)
        \left(\begin{smallmatrix} q \\ q \end{smallmatrix}\right)\\
        &=
        R_q^T
        \left(\begin{smallmatrix} 1 & 0 \\ 0 & q \end{smallmatrix}\right)^{-1}
        \nu_q(w)
        \left(\begin{smallmatrix} q \\ q \end{smallmatrix}\right)
        =
        R_q^T
        \left(\begin{array}{c}
            X(w)\\
            Y(w)
        \end{array}\right)
        =
        \left(\begin{array}{c}
            qX(w)\\
            X(w)+Y(w)
        \end{array}\right).
    \end{align*}
\end{proof}

\subsection{A first application: counting order ideals with cardinality statistics}

We can now provide a formula computing the cardinality statistics over the set of
order ideals of the fence poset $F(w)$ for every word $w\in\{\0,\1\}^*$.

\begin{proposition}\label{prop:statistics-size-satisfies-XYhypothesis}
    For every $w\in\{\0,\1\}^*$,
    let
    \begin{align*}
        X(w) &= \sum_{I\in J^\bullet(F(w))} q^{\size{I}},\\
        Y(w) &= \sum_{I\in J^\circ(F(w))} q^{\size{I}}.
    \end{align*}
    Then, $X(\varepsilon)=q$, $Y(\varepsilon)=1$ and
    $X(w)$ and $Y(w)$ both satisfy \eqref{eq:hypothesis-on-X-Y}.
\end{proposition}

\begin{proof}
    Let
    \[
        \begin{array}{rccl}
            \shift:&2^\N&\to&2^\N \\
            &I&\mapsto &\{i-1\mid i\in I\text{ and } i\neq 0\}
        \end{array}
    \]
    be a map defined on subsets of $\N$ which subtract $1$ to every element,
    and ignoring the element 0.
    For every $w\in\{\0,\1\}^*$, we have that
    \begin{itemize}
        \item $\shift:J^\bullet(F(\1w))\to J(F(w))$ is a bijection,
        \item $\shift:J^\circ(F(\1w))\to J^\circ(F(w))$ is a bijection,
        \item $\shift:J^\bullet(F(\0w))\to J^\bullet(F(w))$ is a bijection,
        \item $\shift:J^\circ(F(\0w))\to J(F(w))$ is a bijection.
    \end{itemize}
    Therefore, for every $w\in\{\0,\1\}^*$, we have
    \begin{align*}
        X(\1w)
        &= \sum_{I\in J^\bullet(F(\1w))} q^{\size{I}}
        = \sum_{I\in J(F(w))} q^{\size{I}+1}
        = q (X(w)+Y(w)),\\
        Y(\1w)
        &= \sum_{I\in J^\circ(F(\1w))} q^{\size{I}}
        = \sum_{I\in J^\circ(F(w))} q^{\size{I}}
        = Y(w),\\
        X(\0w)
        &= \sum_{I\in J^\bullet(F(\0w))} q^{\size{I}}
        = \sum_{I\in J^\bullet(F(w))} q^{\size{I}+1}
        = q X(w),\\
        Y(\0w)
        &= \sum_{I\in J^\circ(F(\0w))} q^{\size{I}}
        = \sum_{I\in J(F(w))} q^{\size{I}}
        = X(w)+Y(w).
    \end{align*}
    Thus, $X(w)$ and $Y(w)$ both satisfy \eqref{eq:hypothesis-on-X-Y}.
    Also
    \[
        X(\varepsilon)
        = \sum_{I\in J^\bullet(F(\varepsilon))} q^{\size{I}}
        = q^1 = q.
    \]
    and
    \[
        Y(\varepsilon)
        = \sum_{I\in J^\circ(F(\varepsilon))} q^{\size{I}}
        = q^0 = 1. 
    \]
\end{proof}

\begin{proposition}\label{prop:statistics-over-order-ideals}
    Let $w\in\{\0,\1\}^*$.
The cardinality statistics over the set of order ideals
of the fence poset $F(w)$ satisfies
\begin{equation*}
\left(\begin{array}{r}
    \sum_{I\in J^\bullet(F(w))} q^{\size{I}}\\
        \sum_{I\in J^\circ (F(w))}  q^{\size{I}}
\end{array}\right)
    =
    \left(\begin{smallmatrix} 1 & 0 \\ 0 & q \end{smallmatrix}\right)^{-1}
            \nu_q(w)
        \left(\begin{array}{c} q \\ q \end{array}\right).
\end{equation*}
\end{proposition}

\begin{proof}
    For every $w\in\{\0,\1\}^*$,
    let
        $X(w) = \sum_{I\in J^\bullet(F(w))} q^{\size{I}}$ and
        $Y(w) = \sum_{I\in J^\circ(F(w))} q^{\size{I}}$.
From Proposition~\ref{prop:statistics-size-satisfies-XYhypothesis}, we
    have that $X(\varepsilon)=q$,
    and $Y(\varepsilon)=1$ and
    $X(w)$ and $Y(w)$ both satisfy \eqref{eq:hypothesis-on-X-Y}.
    Using Proposition~\ref{prop:equivalent-conditions-Fq},
    we obtain
\begin{align*}
    \left(\begin{array}{r}
            X(w)\\
            Y(w)
            \end{array}\right)
    &= \left(\begin{smallmatrix} 1 & 0 \\ 0 & q \end{smallmatrix}\right)^{-1}
            \nu_q(w)
        \left(\begin{smallmatrix} q \\ q \end{smallmatrix}\right).
\end{align*}
\end{proof}

\begin{THEOREMC}
    \MainTheoremC
\end{THEOREMC}

\begin{proof}
    Let $w=W\circ\CFeven(\xx)=W(a)
           =\1^{a_0}\0^{a_{1}}\cdots \1^{a_{2\ell-2}}\0^{a_{2\ell-1}-1}$.
    Using Proposition~\ref{prop:statistics-over-order-ideals}
    and Lemma~\ref{lem:sufficient-conditions-qcontinued-fraction},
    we obtain
    \begin{align*}
    \left(\begin{array}{r}
            \sum_{I\in \Jcal^\bullet(\xx)} q^{\size{I}}\\
            \sum_{I\in \Jcal^\circ  (\xx)} q^{\size{I}}
    \end{array}\right)
        &=
    \left(\begin{array}{r}
            \sum_{I\in J^\bullet(F(W(a)))} q^{\size{I}}\\
            \sum_{I\in J^\circ  (F(W(a)))} q^{\size{I}}
    \end{array}\right)\\
        &= \left(\begin{smallmatrix} 1 & 0 \\ 0 & q \end{smallmatrix}\right)^{-1}
                \nu_q(W(a))
            \left(\begin{smallmatrix} q \\ q \end{smallmatrix}\right)\\
        &=
        \left(\begin{smallmatrix} 1 & 0 \\ 0 & q \end{smallmatrix}\right)^{-1}
        R_q^{a_0}L_q^{a_1}\cdots R_q^{a_{2\ell-2}}L_q^{a_{2\ell-1}}
        \left(\begin{smallmatrix} 1 \\ 0 \end{smallmatrix}\right).
    \end{align*}
\end{proof}

\subsection{A second application: counting admissible sequences with 1-norm statistics}

This result, together with Theorem~\ref{thm:nice-formula-for-order-ideals}
implies the following statement.

\begin{THEOREMB}
    \MainTheoremB
\end{THEOREMB}

\begin{proof}
Let $\Psi:\Jcal(\xx)\to\Bcal(a)$ be the bijection from
the set of order ideals of the fence poset $\Fcal(\xx)$
to the set of admissible sequences $\Bcal(a)$
defined in Theorem~\ref{thm:bijection-order-ideals-to-integersequences}.
We have
\begin{align*}
    \left(\begin{array}{r}
          \sum_{b\in \Bcal^\bullet(a)} q^{\Vert b\Vert_1}\\
          \sum_{b\in \Bcal^\circ  (a)} q^{\Vert b\Vert_1}
    \end{array}\right)
    &=
    \left(\begin{array}{r}
          \sum_{I\in \Jcal^\bullet(\xx)} q^{\Vert \Psi(I)\Vert_1}\\
          \sum_{I\in \Jcal^\circ  (\xx)} q^{\Vert \Psi(I)\Vert_1}
    \end{array}\right)\\
    &=
    \left(\begin{array}{r}
          \sum_{I\in \Jcal^\bullet(\xx)} q^{\size{I}}\\
          \sum_{I\in \Jcal^\circ  (\xx)} q^{\size{I}}
    \end{array}\right).
\end{align*}
The conclusion follows from
Theorem~\ref{thm:nice-formula-for-order-ideals}.
\end{proof}

\section{Snake graphs and their perfect matchings}
\label{sec:snake-graphs-matchings}

Snake graphs have been introduced to provide combinatorial formulas
for elements in cluster algebras of surface type
\cite{MR2661414,MR2807089,MR3034481,zbMATH06144657},
see also \cite{zbMATH07181526}.
The formulas are obtained from the set of perfect matchings
of a snake graph. The enumeration of perfect matchings in a graph,
or equivalently, of domino tilings \cite{MR2074946}
is part of a longer history including the study of
alternating-sign matrices and the Aztec diamonds \cite{zbMATH00130434}.

Perfect matchings in certain snake graphs
also give an interpretation of Markoff numbers
\cite{MR3098784} and of the numerator and denominator of rational numbers from
their continued fraction expansion \cite{MR3778183,MR4058266}.
However, these results are limited to rational numbers $>1$
and need different snake graphs for the numerator and denominators.
In what follows, we introduce snake graphs in a way that allows us to associate
them bijectively to every positive rational number $>0$.
The representation of snake graphs by strings was proposed in
\cite{zbMATH07336893} using a ternary alphabet $\{\rightarrow,\leftarrow,\bullet\}$.
Here, we simply use the binary alphabet $\{\0,\1\}$.

A binary word $w\in\{\0,\1\}^*$ describes a path in the North East quadrant
starting from the origin with
letter $\0$ associated with a unitary horizontal step in the East direction
and
letter $\1$ associated with a unitary vertical step in the North direction
as in Figure~\ref{fig:legend01}.
The end point of the path is
$\overrightarrow{w}=(|w|_\0,|w|_\1)$
where $|w|_a$ denotes the number of occurrences of the letter $a$ in the word $w$.
\begin{figure}[h]
\begin{center}
    \begin{tikzpicture}[scale=.5,>=latex]
    \draw[->] (0,0) -- (1,0) node[right] {$\0$};
    \draw[->] (0,0) -- (0,1) node[above] {$\1$};
    \end{tikzpicture}
    \qquad
    \qquad
    \begin{tikzpicture}[scale=.5,>=latex]
        \draw[->] (0,0) -- ++ (1,0) -- ++ (0,1) -- ++ (1,0) -- ++ (1,0) -- ++
        (0,1) -- ++ (1,0) -- ++ (1,0) ;
        \node at (2.5,-1) {$w=\mathtt{0100100}$};
        \node[left] at (0,0) {(0,0)};
        \node[right] at (5,2) {$(5,2)=\overrightarrow{w}$};
    \end{tikzpicture}
\end{center}
    \caption{A path in the plane using vertical (North) and horizontal (East) unitary steps
    is encoded as a word over the alphabet $\{\0,\1\}$.}
    \label{fig:legend01}
\end{figure}

From a word $w\in\{\0,\1\}^*$ describing a path, we define the snake graph $G(w)$ as a thickening of
this path. More formally, the snake graph $G(w)=(V_w,E_w)$ is defined by the
following set of vertices and (undirected) edges:
\begin{align*}
    V_w&=\{\overrightarrow{p}\mid \text{$p$ prefix of $w$}\}
    + \{0, e_1,e_2, e_1+e_2\} \subset\Z^2\\
    E_w&=
    \bigcup_{\text{$p$ prefix of $w$}}
       S_{\overrightarrow{p}}
\end{align*}
where
$
 S_{g}
 =    \{(g,    g+e_1),
        (g,    g+e_2),
        (g+e_1,g+e_1+e_2),
        (g+e_2,g+e_1+e_2)\}
$
is the set of four edges of the unit square at point $g\in\Z^2$, see
Figure~\ref{fig:snake-graph}.
Notice that a word $w\in\{0,1\}^*$ of length $|w|=n$ has $n+1$ prefixes including
itself and the empty word. Therefore, the snake graph $G(w)$ is the
concatenation of $|w|+1$ unit squares.

\begin{figure}[h]
\begin{center}
    \begin{tikzpicture}[scale=.5,>=latex]
        \draw[dotted,very thick] (0,0) -- ++ (1,0) -- ++ (0,1) -- ++ (-1,0) -- ++ (0,-1);
        \node[left]       at (0,0) {$g$};
        \node[above left] at (0,1) {$g+e_2$};
        \node[right]        at (1,0) {$g+e_1$};
        \node[above right] at (1,1) {$g+e_1+e_2$};
        \node[circled node] at (0,0)  {};
        \node[circled node] at (0,1)  {};
        \node[circled node] at (1,0) {};
        \node[circled node] at (1,1) {};
        \draw (-3,-2) node[circled node] {}
            -- ++ (1,0) node[circled node] {}
            -- ++ (0,1) node[circled node] {}
            -- ++ (1,0) node[circled node] {}
            -- ++ (1,0) node[circled node] {}
            -- ++ (0,1) node[circled node] {};
    \end{tikzpicture}
    \qquad
    \qquad
    \begin{tikzpicture}[scale=.5,>=latex]
        \draw[dotted,very thick] (0,0) node[circled node] {}
            -- ++ (1,0) node[circled node] {}
            -- ++ (1,0) node[circled node] {}
            -- ++ (0,1) node[circled node] {}
            -- ++ (1,0) node[circled node] {}
            -- ++ (1,0) node[circled node] {}
            -- ++ (0,1) node[circled node] {}
            -- ++ (1,0) node[circled node] {}
            -- ++ (1,0) node[circled node] {}
            -- ++ (0,1) node[circled node] {}
            -- ++ (-1,0) node[circled node] {}
            -- ++ (-1,0) node[circled node] {}
            -- ++ (-1,0) node[circled node] {}
            -- ++ (0,-1) node[circled node] {}
            -- ++ (-1,0) node[circled node] {}
            -- ++ (-1,0) node[circled node] {}
            -- ++ (0,-1) node[circled node] {}
            -- ++ (-1,0) node[circled node] {}
            -- ++ (0,-1) node[circled node] {};
            \draw[dotted,very thick] (1,0) -- ++ (0,1);
            \draw[dotted,very thick] (2,1) -- ++ (0,1);
            \draw[dotted,very thick] (3,1) -- ++ (0,1);
            \draw[dotted,very thick] (4,2) -- ++ (0,1);
            \draw[dotted,very thick] (5,2) -- ++ (0,1);
            \draw[dotted,very thick] (1,1) -- ++ (1,0);
            \draw[dotted,very thick] (3,2) -- ++ (1,0);
            \node[left] at  (0,0) {(0,0)};
            \node[left] at  (0,1) {(0,1)};
            \node[right] at (6,3) {(6,3)};
            \node[right] at (6,2) {(6,2)};
    \end{tikzpicture}
\end{center}
\caption{Right: the snake graph $G(\mathtt{0100100})$.
         Left: the vertices of the unit square
         at $g=\protect\overrightarrow{p}$
         for the prefix
         $p=\0\1\0\0\1$.
         }
    \label{fig:snake-graph}
\end{figure}

\subsection{From rational numbers to snake graphs}

We denote the non-identity involution on the alphabet $\{\0,\1\}$ as
$\overline{\0}=\1$ and $\overline{\1}=\0$ which extends
to an involution on $\{\0,\1\}^*$. This involution
induces a symmetry on the set of snake graphs.

\begin{lemma}\label{lem:G-wbar-is-symmetric-to-Gw}
    For every $w\in\{\0,\1\}^*$,
    the snake graph $G(\overline{w})$
    is the mirror image of the snake graph $G(w)$
    under the map $\sigma:\R^2\to\R^2:(a,b)\mapsto (b,a)$.
\end{lemma}

\begin{proof}
    Let $\sigma:\R^2\to\R^2:(a,b)\mapsto(b,a)$. The vertices
    of the snake graph $\sigma(G(w))$ satisfy
\begin{align*}
    \text{vertices of } \sigma(G(w))
    &= \sigma\left(\{\overrightarrow{p}\mid \text{$p$ prefix of $w$}\}
        + \{0, e_1,e_2, e_1+e_2\}\right)\\
    &= \{\sigma(\overrightarrow{p})\mid \text{$p$ prefix of $w$}\}
        + \{0, e_1,e_2, e_1+e_2\}\\
    &= \{\overrightarrow{p}\mid \text{$p$ prefix of $\overline{w}$}\}
        + \{0, e_1,e_2, e_1+e_2\}\\
    &= \text{vertices of } G(\overline{w}).\qedhere
\end{align*}
\end{proof}

We define a length-preserving map
$\theta:\{\0,\1\}^*\to\{\0,\1\}^*$
recursively by $\theta(\varepsilon)=\varepsilon$ and
\[
    \theta(\alpha w)=
    \begin{cases}
        \overline{\alpha}\theta(w) & \text{ if $|w|$ is even,}\\
        \alpha\theta(w)            & \text{ if $|w|$ is odd,}
    \end{cases}
\]
for every word $w\in\{\0,\1\}^*$ and letter $\alpha\in\{\0,\1\}$.

\begin{lemma}
$\theta:\{\0,\1\}^*\to\{\0,\1\}^*$ is an involution.
\end{lemma}

\begin{proof}
The map $\theta$ flips the letters at even distance
from the right end of the word.
\end{proof}

\begin{lemma} \label{lem:map-theta-commutes-bar-bar}
    For every $w\in\{\0,\1\}^*$, $\overline{\theta(w)}=\theta(\overline{w})$.
\end{lemma}

\begin{proof}
    We proceed by induction on the length of $w$.
    The equality holds if $w=\varepsilon$ is the empty word.
    Let $\alpha\in\{\0,\1\}$ and $w\in\{\0,\1\}^*$.
    If $|w|$ is even, we have
    \[
        \overline{\theta(\alpha w)}
        = \overline{\overline{\alpha} \theta(w)}
        = \alpha \overline{\theta(w)}
        = \alpha \theta(\overline{w})
        =\theta(\overline{\alpha w}).
    \]
    If $|w|$ is odd, we have
    \[
        \overline{\theta(\alpha w)}
        = \overline{\alpha \theta(w)}
        = \overline{\alpha} \overline{\theta(w)}
        = \overline{\alpha} \theta(\overline{w})
        =\theta(\overline{\alpha w}).\qedhere
    \]
\end{proof}

    In other words, it follows from
    Lemma~\ref{lem:map-W-commutes-inverse-bar},
    Lemma~\ref{lem:map-theta-commutes-bar-bar} and
    Lemma~\ref{lem:G-wbar-is-symmetric-to-Gw}
    that the diagram shown in Figure~\ref{fig:commutative-Gcal}
    is commutative.
Using the maps $W$ and $\theta$, we associate
a snake graph $\Gcal(\xx)$
to every rational number $\xx\in\Q_{>0}$
as follows:
\begin{equation}
\Gcal = G\circ \theta\circ W.
\end{equation}

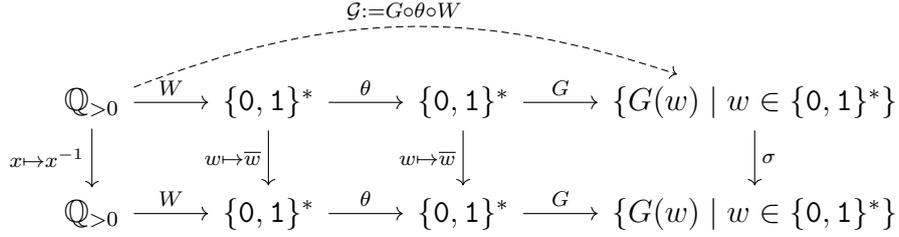
\begin{figure}[h]
\begin{center}
\begin{tikzcd}[ampersand replacement=\&]
    \Q_{>0}
    \arrow{r}{W}
    \arrow[swap]{d}{x\mapsto x^{-1}}
    \arrow[rrr,bend left=20,pos=.5,dashed,"\Gcal:=G\circ\theta\circ W"]
	\& \{\0,\1\}^* \arrow{r}{\theta}
                   \arrow[swap]{d}{w\mapsto \overline{w}}
    \& \{\0,\1\}^* \arrow[swap]{d}{w\mapsto \overline{w}}
                   \arrow{r}{G}
    \& \{G(w)\mid w\in\{\0,\1\}^*\}
                   \arrow{d}{\sigma}
    \\
    \Q_{>0} \arrow{r}{W}
	\& \{\0,\1\}^* \arrow{r}{\theta}
    \& \{\0,\1\}^* \arrow{r}{G}
    \& \{G(w)\mid w\in\{\0,\1\}^*\}
\end{tikzcd}
\end{center}
    \caption{A commutative diagram involving $\Gcal = G\circ \theta\circ W$.}
    \label{fig:commutative-Gcal}
\end{figure}

\begin{lemma}
    The map $\Gcal$ is a bijection from the set $\Q_{>0}$
    of positive rational numbers to the set $\{G(w)\mid w\in\{\0,\1\}^*\}$ of
    snake graphs.
\end{lemma}

\begin{proof}
It follows from the fact that intermediate maps $W$, $\theta$ and $G$ are bijections.
\end{proof}

Commutativity of the diagram shown in Figure~\ref{fig:commutative-Gcal} implies
the following result.

\begin{lemma}\label{lem:mirror-image-snake-graph-inverse-rational-number}
    For every $\xx\in\Q_{>0}$,
    the snake graph $\Gcal(\xx^{-1})$
    is the mirror image of the snake graph $\Gcal(\xx)$
    under the map $\sigma:\R^2\to\R^2:(a,b)\mapsto (b,a)$.
\end{lemma}

\begin{proof}
Lemma~\ref{lem:map-W-commutes-inverse-bar}
and
Lemma~\ref{lem:map-theta-commutes-bar-bar}
imply that
    $\theta\circ W(\xx^{-1})=\overline{\theta\circ W(\xx)}$.
Using Lemma~\ref{lem:G-wbar-is-symmetric-to-Gw}, we have
    \[
        \Gcal(\xx^{-1})
        =G(\theta\circ W(\xx^{-1}))
        =G(\overline{\theta\circ W(\xx)})
        =\sigma(G(\theta\circ W(\xx)))
        =\sigma(\Gcal(\xx)).\qedhere
    \]
\end{proof}

An example illustrating
\begin{align*}
    \Gcal(10/27)&=\Gcal([0;2,1,2,2,1])=G(\theta(\0\0\1\0\0\1\1))=G(\1\0\0\0\1\1\0) \text{ and }\\
    \Gcal(27/10)&=\Gcal([2;1,2,3])    =G(\theta(\1\1\0\1\1\0\0))=G(\0\1\1\1\0\0\1)
\end{align*}
is shown in Figure~\ref{fig:Lee_Schiffler_2017_Figure2}.
Both snake graphs contain 8 cells
which is the sum of the partial quotients
$8=2+1+2+3=0+2+1+2+2+1$.

\begin{figure}[h]
\[
    \def\arraycolsep{10mm}
\begin{array}{cc}
    \includegraphics{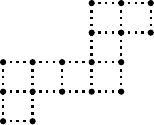} &
    \includegraphics{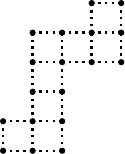}\\
    \Gcal(10/27)
    &
    \Gcal(27/10)\\
\end{array}
\]
    \caption{
    The snake graphs
    $\Gcal(10/27)$
    and $\Gcal(27/10))$.
    }
    \label{fig:Lee_Schiffler_2017_Figure2}
\end{figure}

The map $\Gcal$ maps the Raney tree
of positive rational numbers \cite{MR1453849},
also known as the Calkin-Wilf tree \cite{MR1763062},
to the binary tree of snake graphs where the parent relation
is defined by the suffix relation; see
Figure~\ref{fig:CalkinWilf-tree-of-snakes}.
\begin{figure}
\begin{center}
\includegraphics[width=\linewidth]{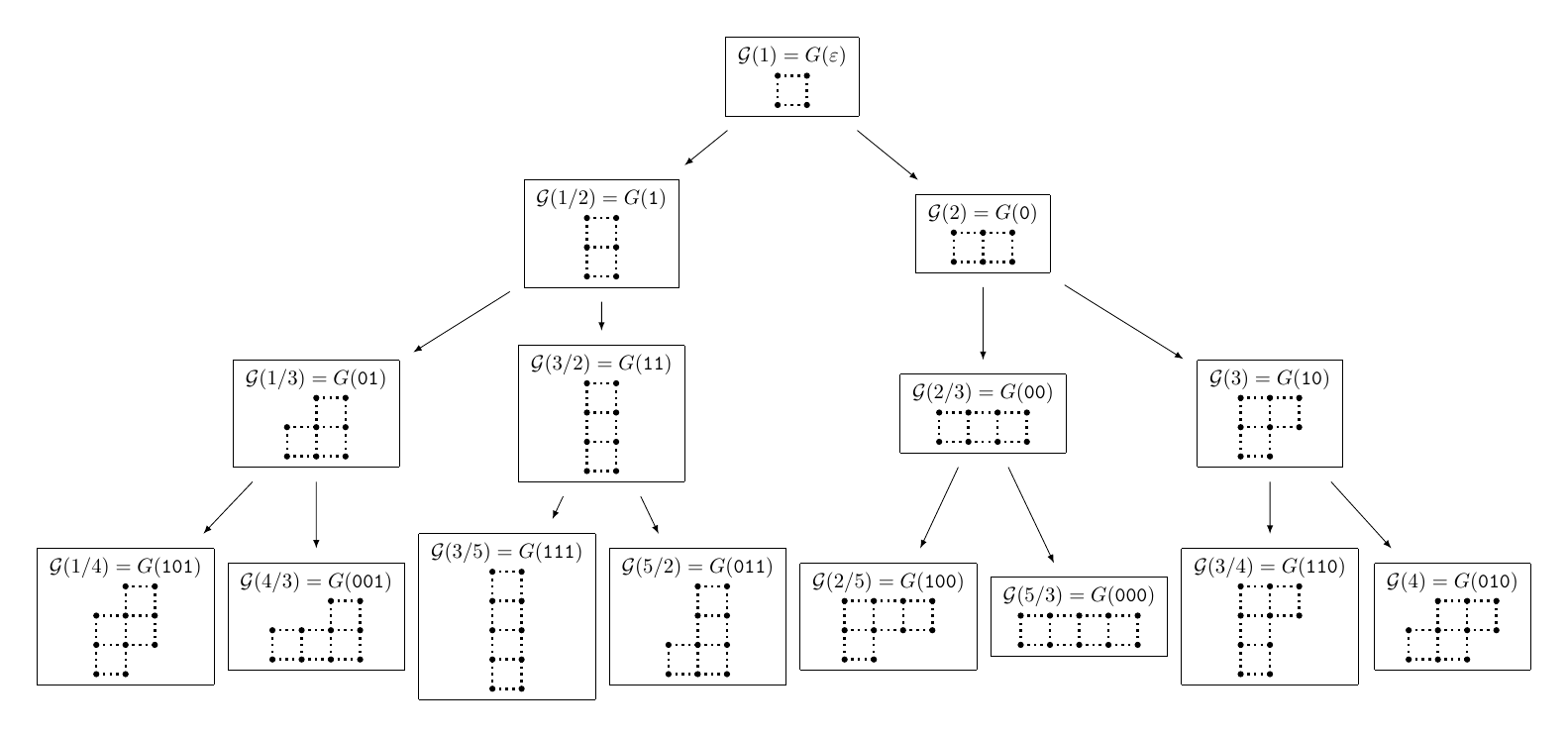}
\end{center}
    \caption{The map $\Gcal$ maps the Calkin-Wilf tree of positive
    rational numbers onto the binary tree of snake graphs defined by the suffix relation.}
    \label{fig:CalkinWilf-tree-of-snakes}
\end{figure}
In general, the snake graph
\[
\Gcal[a_0;\dots,a_{2\ell-1}]
= G(\theta(\1^{a_0}\0^{a_{1}}\cdots \1^{a_{{2\ell}-2}}\0^{a_{2\ell-1}-1}))
\]
is made of $a_0+\dots+a_{2\ell-1}$ unit squares.
By comparison, the snake graphs defined in \cite{MR4016518} contain
one less unit square and are defined only for rational numbers that are
larger than 1.
Figure~\ref{fig:Lee_Schiffler_2017_Figure2}
can be compared with Figure~2 in \cite{MR4016518}
where the snake graph associated with the rational number $27/10$ contains seven
unit squares instead of eight.
See Corollary~\ref{cor:the-usual-combinatorial-interpretation}
and Remark~\ref{rem:relation-to-Thm-3.4-Schiffler}
for the precise relation between these definitions of snake graphs.

\subsection{The dual map defined by Claussen}
In \cite{claussen_expansion_2020} (also implicitly defined in
\cite{zbMATH07181526}), another length-preserving involution
$\eta:\{\0,\1\}^*\to\{\0,\1\}^*$
called \emph{dual map} that resembles $\theta$ was considered
in the context of snake graphs.
It is defined recursively by $\eta(\varepsilon)=\varepsilon$ and
\begin{equation}\label{eq:eta}
    \eta(w\alpha)=
    \begin{cases}
        \eta(w) \overline{\alpha} & \text{ if $|w|$ is even,} \\
        \eta(w) \alpha            & \text{ if $|w|$ is odd,}
    \end{cases}
\end{equation}
for every word $w\in\{\0,\1\}^*$ and letter $\alpha\in\{\0,\1\}$.

It turns out that the map $\theta$ is conjugate to the dual map $\eta$ through the involution
$w \mapsto \widehat{w}$ on $\{\0,\1\}^*$ defined in
Section~\ref{sec:proof-of-Theorems-B-and-D}.
Recall that $\widehat{w} = \overline{w_k} \cdots \overline{w_1}$ if $w = w_1
\cdots w_k$, where $\overline{\0}=\1$ and $\overline{\1}=\0$.

\begin{lemma}\label{lem:the-dual-map-eta}
    For every word $w\in\{\0,\1\}^*$,
    $\widehat{\theta(w)}=\eta(\widehat{w})$.
    In other words, the following diagram is commutative:
\begin{center}
\begin{tikzcd}[ampersand replacement=\&]
	\{\0,\1\}^* \arrow{r}{\theta}
                   \arrow[swap]{d}{w\,\mapsto \widehat{w}}
    \& \{\0,\1\}^* \arrow[swap]{d}{w\,\mapsto \widehat{w}}
    \\
	\{\0,\1\}^* \arrow{r}{\eta}
    \& \{\0,\1\}^*
\end{tikzcd}
\end{center}
\end{lemma}

\begin{proof}
    Let $w\in\{\0,\1\}^*$ and $\alpha\in\{\0,\1\}$.
    If $|w|$ is even, then
    \[
      \widehat{\theta(\alpha w})
     =\widehat{\overline{\alpha}\theta(w})
     =\widehat{\theta(w})\alpha
     =\eta(\widehat{w})\alpha
     =\eta(\widehat{w}\overline{\alpha})
     =\eta(\widehat{\alpha w}).
    \]
    If $|w|$ is odd, then
    \[
      \widehat{\theta(\alpha w})
     =\widehat{\alpha\theta(w})
     =\widehat{\theta(w})\overline{\alpha}
     =\eta(\widehat{w})\overline{\alpha}
     =\eta(\widehat{w}\overline{\alpha})
     =\eta(\widehat{\alpha w}).
    \]
\end{proof}

Recall, from Lemma~\ref{lem:map-W-commutes-the-widehat}, that the map $w
\mapsto \widehat{w}$ is conjugate to the involution
$\tau:\Acaleven\to\Acaleven$.
However, in this work, it is preferable to use $\theta$ instead of $\eta$ in
order to
obtain Theorem~\ref{thm:nice-formula-for-qmatchings}
and
have the start of the snake graph correspond to the first partial
quotients of the continued fraction expansion.

\subsection{Perfect matchings}

In graph theory, a \emph{matching} in an undirected graph is a subset of edges
without common vertices and a \emph{perfect matching} is a matching that
matches all vertices of the graph.
We denote by $M(g)$ the set of perfect matchings of a graph $g$.
By abuse of notation, for every $w\in\{\0,\1\}^*$,
we also denote by $M(w)$ the set of perfect matchings of the snake graph $G(w)$.

A snake graph admits perfect matchings.
An example is shown in Figure~\ref{fig:a-perfect-matching}.
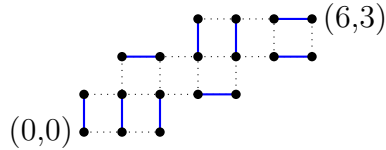
\begin{figure}[h]
\begin{center}
    \begin{tikzpicture}[scale=.5,>=latex]
        \draw[dotted] (0,0) -- ++ (1,0) -- ++ (1,0) -- ++ (0,1) -- ++ (1,0) --
        ++ (1,0) -- ++ (0,1) -- ++ (1,0) -- ++ (1,0) -- ++ (0,1) -- ++ (-1,0)
        -- ++ (-1,0) -- ++ (-1,0) -- ++ (0,-1) -- ++ (-1,0) -- ++ (-1,0) -- ++
        (0,-1) -- ++ (-1,0) -- ++ (0,-1);
        \draw[dotted] (1,0) -- ++ (0,1);
        \draw[dotted] (2,1) -- ++ (0,1);
        \draw[dotted] (3,1) -- ++ (0,1);
        \draw[dotted] (4,2) -- ++ (0,1);
        \draw[dotted] (5,2) -- ++ (0,1);
        \draw[dotted] (1,1) -- ++ (1,0);
        \draw[dotted] (3,2) -- ++ (1,0);
        \draw[somematching] (0,0) -- ++ (0,1);
        \draw[somematching] (2,0) -- ++ (0,1);
        \draw[somematching] (1,0) -- ++ (0,1);
        \draw[somematching] (5,2) -- ++ (1,0);
        \draw[somematching] (5,3) -- ++ (1,0);
        \draw[somematching] (3,2) -- ++ (0,1);
        \draw[somematching] (4,2) -- ++ (0,1);
        \draw[somematching] (3,1) -- ++ (1,0);
        \draw[somematching] (1,2) -- ++ (1,0);
        \foreach \p in {(0,0), (1,0), (2,0), (2,1), (3,1), (4,1), (4,2), (5,2),
        (6,2), (6,3), (5,3), (4,3), (3,3), (3,2), (2,2), (1,2), (1,1), (0,1)}
            \node[circled node] at \p {};
        \node[left] at (0,0) {(0,0)};
        \node[right] at (6,3) {(6,3)};
    \end{tikzpicture}
\end{center}
    \caption{A perfect matching of the snake graph $G(w)$ with $w=\mathtt{0100100}$.}
\label{fig:a-perfect-matching}
\end{figure}

For every positive rational number $\xx\in\Q_{>0}$, we define
$\Mcal(\xx)$
to be the set of perfect matchings of the snake graph
$\Gcal(\xx)$.
In other words, $\Mcal:=M\circ G\circ\theta\circ W$;
see Figure~\ref{fig:the-big-picture}.

A combinatorial interpretation of every rational number $\frac{r}{s}>1$ was
proposed by {\c{C}}anak\c{c}i and Schiffler
in terms of the quotient of cardinalities of sets of perfect matchings from
two distinct snake graphs \cite{MR3778183}.
Their result can be stated using $\Gcal$ and $\Mcal$ as follows.

\begin{theorem}[Theorem 3.4, \cite{MR3778183}]\label{thm:schiffler}
Let $[a_0;\dots,a_{n-1}]$ be the continued fraction expansion
of a rational number greater than 1, that is, such that $a_0\geq 1$.
Then,
\[
[a_0;\dots,a_{n-1}]
=
\frac{\#\Mcal [a_0-1,\dots,a_{n-1}]}
     {\#\Mcal [a_1-1,\dots,a_{n-1}]}
\]
and the fraction on the right-hand side is reduced.
\end{theorem}

Below, we show that the above theorem can be extended to
consider the case $a_0=0$ while also giving a combinatorial
interpretation of both the numerator and denominator of a rational
number from the same snake graph using a natural dichotomy on its
set of perfect matchings.

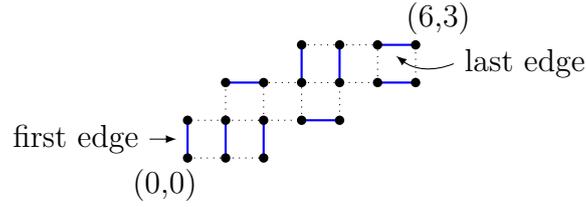
\begin{figure}[h]
\begin{center}
    \begin{tikzpicture}[scale=.5,>=latex]
        \draw[dotted] (0,0) -- ++ (1,0) -- ++ (1,0) -- ++ (0,1) -- ++ (1,0) --
        ++ (1,0) -- ++ (0,1) -- ++ (1,0) -- ++ (1,0) -- ++ (0,1) -- ++ (-1,0)
        -- ++ (-1,0) -- ++ (-1,0) -- ++ (0,-1) -- ++ (-1,0) -- ++ (-1,0) -- ++
        (0,-1) -- ++ (-1,0) -- ++ (0,-1);
        \draw[dotted] (1,0) -- ++ (0,1);
        \draw[dotted] (2,1) -- ++ (0,1);
        \draw[dotted] (3,1) -- ++ (0,1);
        \draw[dotted] (4,2) -- ++ (0,1);
        \draw[dotted] (5,2) -- ++ (0,1);
        \draw[dotted] (1,1) -- ++ (1,0);
        \draw[dotted] (3,2) -- ++ (1,0);
        \draw[somematching] (0,0) -- ++ (0,1);
        \draw[somematching] (2,0) -- ++ (0,1);
        \draw[somematching] (1,0) -- ++ (0,1);
        \draw[somematching] (5,2) -- ++ (1,0);
        \draw[somematching] (5,3) -- ++ (1,0);
        \draw[somematching] (3,2) -- ++ (0,1);
        \draw[somematching] (4,2) -- ++ (0,1);
        \draw[somematching] (3,1) -- ++ (1,0);
        \draw[somematching] (1,2) -- ++ (1,0);
        \foreach \p in {(0,0), (1,0), (2,0), (2,1), (3,1), (4,1), (4,2), (5,2),
        (6,2), (6,3), (5,3), (4,3), (3,3), (3,2), (2,2), (1,2), (1,1), (0,1)}
            \node[circled node] at \p {};
        \node[right] (lastedge) at (7,2.5) {last edge};
        \node (A) at (5.2,3.0) {};
        \draw[<-,bend right] (A) to (lastedge.west);
        \node[left] (firstedge) at (-1,.5) {first edge};
        \node (B) at (0,.5) {};
        \draw[->] (firstedge) to (B);
        \node[below,xshift=-3mm] at (0,0) {(0,0)};
        \node[above,xshift=3mm] at (6,3) {(6,3)};
    \end{tikzpicture}
\end{center}
    \caption{A perfect matching of the snake graph $G(w)$ with $w=\mathtt{0100100}$.
             The first and last edge of the matching are indicated.}
\label{fig:a-perfect-matching-first-last-edge}
\end{figure}

\subsection*{The first and last edge of a matching}
A notion which is important is the first and last edge of a perfect matching.
The \emph{first edge} of a perfect matching is the edge
containing the vertex $(0,0)$, that is, the bottom-most left-most vertex of the
snake graph $G(w)$.
The \emph{last edge} of a perfect matching is the edge
containing the vertex $\overrightarrow{w}+(1,1)$, which is the top-most right-most vertex of the snake graph $G(w)$,
see Figure~\ref{fig:a-perfect-matching-first-last-edge}.

\subsection*{A dichotomy on the set of perfect matchings}

Let $w\in\{\0,\1\}^*$.
Let $M(w)$ be the set of perfect matchings of the snake graph $G(w)$.
The dichotomy on the set of perfect matchings is defined as follows.
\begin{definition}\label{def:matching-dichothomoy}
We say that a perfect matching $m\in M(w)$ is a \emph{$\carreplein$-matching} if
\begin{itemize}
\item $|w|$ is even and its first edge is horizontal, or
\item $|w|$ is odd and its first edge is vertical.
\end{itemize}
Otherwise, we say that $m$ is a \emph{$\carrevide$-matching}, that is, if
\begin{itemize}
\item $|w|$ is even and its first edge is vertical, or
\item $|w|$ is odd and its first edge is horizontal.
\end{itemize}
\end{definition}
The terminology $\carrevide$- and $\carreplein$-matching was chosen because
the first edge of a $\carrevide$-matching is parallel to the first edge of the
basic perfect matching and the first edge of a $\carreplein$-matching is orthogonal to the
first edge of the basic perfect matching; see Remark~\ref{rem:parallel-perp-interpretation}.
Let
\begin{align*}
    M^\carreplein(w)&=\{m\in M(w)\mid\text{$m$ is a $\carreplein$-matching}\},\\
    M^\carrevide(w) &=\{m\in M(w)\mid\text{$m$ is a $\carrevide$-matching}\}.
\end{align*}
These two sets form a partition of $M(w)=M^\carrevide(w)\cup M^\carreplein(w)$.

\begin{example}
The snake graph associated with $\xx=\frac{2}{7}$ is
$\Gcal(\frac{2}{7})=G(\0\1\0\0)$
since
\[
\theta\circ W(\textstyle\frac{2}{7})
= \theta(W([0;3,1,1]))
= \theta(\1^0\0^3\1^1\0^{1-1})
= \theta(\0\0\0\1)
= \0\1\0\0.
\]
The perfect matchings
of the snake graph $G(\0\1\0\0)$ whose first edge is horizontal are
\[
M^\carreplein(\0\1\0\0) = \left\{\raisebox{-5mm}{\includegraphics{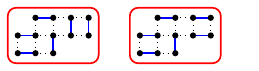}}\right\}.
\]
The perfect matchings of the snake graph $G(\0\1\0\0)$ whose first edge is vertical are
\[
    M^\carrevide(\0\1\0\0) = \left\{\raisebox{-12mm}{\includegraphics{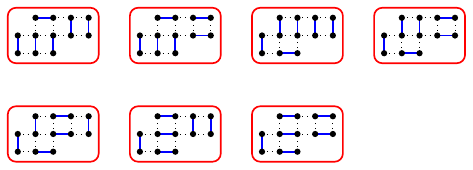}}\right\}.
\]
We observe that the cardinalities
$\# M^\carreplein(\0\1\0\0)$ and $\# M^\carrevide (\0\1\0\0)$
are the numerator and denominator of the rational number $\xx$:
\[
\frac{\# \Mcal^\carreplein(\textstyle\frac{2}{7})}
     {\# \Mcal^\carrevide (\textstyle\frac{2}{7})}
     =
\frac{\# M^\carreplein\circ \Gcal(\textstyle\frac{2}{7})}
     {\# M^\carrevide \circ \Gcal(\textstyle\frac{2}{7})}
     =
\frac{\# M^\carreplein(\0\1\0\0)}
     {\# M^\carrevide (\0\1\0\0)}
    = \frac{2}{7} 
    = x.
\]
\end{example}

We show that this holds in general.
The following result extends Theorem~\ref{thm:schiffler} to every
positive rational number. It gives an interpretation of both the numerator
and denominator of a rational number in terms of the perfect matchings
of the same snake graph.

\begin{theorem}\label{thm:our-combinatorial-interpretation-of-rational}
Let $\xx>0$ be a positive rational number.
Then,
\[
x
=
    \frac{\#\Mcal^\carreplein (\xx)}
     {\#\Mcal^\carrevide  (\xx)}
\]
and the fraction on the right-hand side is reduced.
\end{theorem}

Theorem~\ref{thm:our-combinatorial-interpretation-of-rational}
is proved as a consequence of
Theorem~\ref{thm:nice-formula-for-qmatchings}
which is a $q$-analog of it
thanks to the area statistics defined below.
We prove Theorem~\ref{thm:nice-formula-for-qmatchings} in
Section~\ref{sec:proof-of-thm:nice-formula-for-qmatchings}.
An alternative statement deduced from
Theorem~\ref{thm:our-combinatorial-interpretation-of-rational}
is the following.

\begin{corollary}
    Let $r,s\geq1$ be coprime integers.
    The snake graph $\Gcal(\frac{r}{s})$
    has $r$ $\carreplein$-matchings and
    $s$ $\carrevide$-matchings.
    Thus,
    \[
        \frac{r}{s}
        =
        \frac{\# \Mcal^\carreplein(\frac{r}{s})}
             {\# \Mcal^\carrevide(\frac{r}{s})}.
    \]
\end{corollary}

\begin{figure}[h]
    \includegraphics{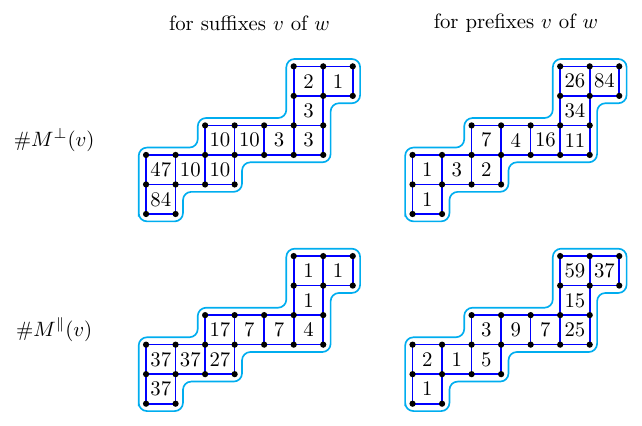}
    \caption{
    The snake graph $\Gcal(\frac{84}{37})=G(w)$
    with $w=\1\0\0\1\0\0\0\1\1\0$.
    The values indicate the number of
    $\carreplein$-matchings $\carrevide$-matchings
    in the snake graph $G(v)$ for every prefix and suffixes of $w$.
    This reproduces Example 2.3 from \cite{MR4058266}
    using a single snake graph for both the numerators and denominators.
    }
    \label{fig:prefixes_suffixes_MR4058266_example23}
\end{figure}

\begin{example}
    This example is a reproduction of Example 2.3 from \cite{MR4058266}
    with our definition of snake graph.
    The continued fraction
    $[2, 3, 1, 2, 3]=\frac{84}{37}$ has convergents
    \[
        [2]=\frac{2}{1},\quad
        [2, 3]=\frac{7}{3},\quad
        [2, 3, 1]=\frac{9}{4},\quad
        [2, 3, 1, 2]=\frac{25}{11},\quad
        [2, 3, 1, 2, 3]=\frac{84}{37}.
    \]
    These rational numbers define a path in the Stern-Brocot tree
    \cite{MR1397498}
    whose intermediate nodes are
    \[
        \frac{1}{1},\quad
        \frac{2}{1},\quad
        \frac{3}{1},\quad
        \frac{5}{2},\quad
        \frac{7}{3},\quad
        \frac{9}{4},\quad
        \frac{16}{7},\quad
        \frac{25}{11},\quad
        \frac{34}{15},\quad
        \frac{59}{26},\quad
        \frac{84}{37}.
    \]
    All these values can be interpreted in the
    snake graph $\Gcal(\frac{84}{37})=G(w)$
    with $w=\1\0\0\1\0\0\0\1\1\0$.
    The numbers $\#M^\carreplein(v)$
    and $\#M^\carrevide(v)$ are numerators and denominators of
    a convergent $\frac{84}{37}$ when $v$ is a prefix of $w$.
    When $v$ is a suffix of $w$,
    the numbers $\#M^\carreplein(v)$
    and $\#M^\carrevide(v)$ are consecutive values in
    the execution of the Euclid algorithm when computing the gcd of 84 and 37;
    see Figure~\ref{fig:prefixes_suffixes_MR4058266_example23}.
\end{example}

\subsection*{Perfect matchings with no interior edges}

An edge $(u,v)$ of a snake graph $G(w)=(V_w,E_w)$ is \emph{interior}
if the graph induced by the subset $V_w\setminus\{u,v\}$ is not connected
\cite{MR3034481}.
Remaining edges are called \emph{boundary} edges \cite{MR2661414,zbMATH06144657}.
Two of the perfect matchings use only boundary edges of a snake graph;
see Figure~\ref{fig:boundary-matchings}.
\begin{figure}[h]
\begin{center}
    \begin{tikzpicture}[scale=.5,>=latex]
        \draw[dotted] (0,0) -- ++ (1,0) -- ++ (1,0) -- ++ (0,1) -- ++ (1,0) --
        ++ (1,0) -- ++ (0,1) -- ++ (1,0) -- ++ (1,0) -- ++ (0,1) -- ++ (-1,0)
        -- ++ (-1,0) -- ++ (-1,0) -- ++ (0,-1) -- ++ (-1,0) -- ++ (-1,0) -- ++
        (0,-1) -- ++ (-1,0) -- ++ (0,-1);
        \draw[dotted] (1,0) -- ++ (0,1);
        \draw[dotted] (2,1) -- ++ (0,1);
        \draw[dotted] (3,1) -- ++ (0,1);
        \draw[dotted] (4,2) -- ++ (0,1);
        \draw[dotted] (5,2) -- ++ (0,1);
        \draw[dotted] (1,1) -- ++ (1,0);
        \draw[dotted] (3,2) -- ++ (1,0);
        \draw[somematching] (0,0)
            to ++ (1,0) ++ (1,0)
            to ++ (0,1) ++ (1,0)
            to ++ (1,0) ++ (0,1)
            to ++ (1,0) ++ (1,0)
            to ++ (0,1) ++ (-1,0)
            to ++ (-1,0) ++ (-1,0)
            to ++ (0,-1) ++ (-1,0)
            to ++ (-1,0) ++ (0,-1)
            to ++ (-1,0) ++ (0,-1) ;
        \foreach \p in {(0,0), (1,0), (2,0), (2,1), (3,1), (4,1), (4,2), (5,2),
        (6,2), (6,3), (5,3), (4,3), (3,3), (3,2), (2,2), (1,2), (1,1), (0,1)}
            \node[circled node] at \p {};
    \end{tikzpicture}
    \qquad
    \qquad
    \begin{tikzpicture}[scale=.5,>=latex]
        \draw[dotted] (0,0) -- ++ (1,0) -- ++ (1,0) -- ++ (0,1) -- ++ (1,0) --
        ++ (1,0) -- ++ (0,1) -- ++ (1,0) -- ++ (1,0) -- ++ (0,1) -- ++ (-1,0)
        -- ++ (-1,0) -- ++ (-1,0) -- ++ (0,-1) -- ++ (-1,0) -- ++ (-1,0) -- ++
        (0,-1) -- ++ (-1,0) -- ++ (0,-1);
        \draw[dotted] (1,0) -- ++ (0,1);
        \draw[dotted] (2,1) -- ++ (0,1);
        \draw[dotted] (3,1) -- ++ (0,1);
        \draw[dotted] (4,2) -- ++ (0,1);
        \draw[dotted] (5,2) -- ++ (0,1);
        \draw[dotted] (1,1) -- ++ (1,0);
        \draw[dotted] (3,2) -- ++ (1,0);
        \draw[somematching] (0,0)
            ++ (1,0)  to ++ (1,0)
            ++ (0,1)  to ++ (1,0)
            ++ (1,0)  to ++ (0,1)
            ++ (1,0)  to ++ (1,0)
            ++ (0,1)  to ++ (-1,0)
            ++ (-1,0) to ++ (-1,0)
            ++ (0,-1) to ++ (-1,0)
            ++ (-1,0) to ++ (0,-1)
            ++ (-1,0) to ++ (0,-1) ;
        \foreach \p in {(0,0), (1,0), (2,0), (2,1), (3,1), (4,1), (4,2), (5,2),
        (6,2), (6,3), (5,3), (4,3), (3,3), (3,2), (2,2), (1,2), (1,1), (0,1)}
            \node[circled node] at \p {};
    \end{tikzpicture}
\end{center}
    \caption{Two perfect matchings of $G(w)$ use only boundary edges.}
    \label{fig:boundary-matchings}
\end{figure}
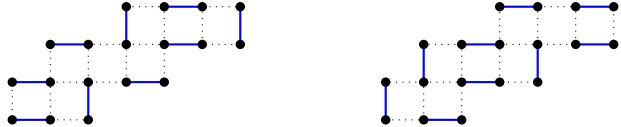

These two perfect matchings using only boundary edges are called \emph{minimal}
and \emph{maximal} perfect matching in
\cite{MR3034481,MR2661414,zbMATH06144657}, see also \cite{zbMATH08015631}.
The standard convention for the minimal perfect matching is that
the first edge is horizontal.
In this article, we need another convention.
The perfect matching with only boundary edges and such that its last edge is
vertical is called the \emph{basic matching} of the snake graph.
Depending on the parity of the length of the snake graph, the basic matching may
correspond to the minimal or to the maximal perfect matching.
We denote the basic matching as $\bfrak_w$ or as $\bfrak$ when the context is clear
and we illustrate it with thick red dashed edges in illustrations;
see Figure~\ref{fig:the-basic-matching}.

\begin{figure}[h]
\begin{center}
    \begin{tikzpicture}[scale=.5,>=latex]
        \draw[dotted] (0,0) -- ++ (1,0) -- ++ (1,0) -- ++ (0,1) -- ++ (1,0) --
        ++ (1,0) -- ++ (0,1) -- ++ (1,0) -- ++ (1,0) -- ++ (0,1) -- ++ (-1,0)
        -- ++ (-1,0) -- ++ (-1,0) -- ++ (0,-1) -- ++ (-1,0) -- ++ (-1,0) -- ++
        (0,-1) -- ++ (-1,0) -- ++ (0,-1);
        \draw[dotted] (1,0) -- ++ (0,1);
        \draw[dotted] (2,1) -- ++ (0,1);
        \draw[dotted] (3,1) -- ++ (0,1);
        \draw[dotted] (4,2) -- ++ (0,1);
        \draw[dotted] (5,2) -- ++ (0,1);
        \draw[dotted] (1,1) -- ++ (1,0);
        \draw[dotted] (3,2) -- ++ (1,0);
        \draw[thebasicmatching] (0,0)
            to ++ (1,0) ++ (1,0)
            to ++ (0,1) ++ (1,0)
            to ++ (1,0) ++ (0,1)
            to ++ (1,0) ++ (1,0)
            to ++ (0,1) ++ (-1,0)
            to ++ (-1,0) ++ (-1,0)
            to ++ (0,-1) ++ (-1,0)
            to ++ (-1,0) ++ (0,-1)
            to ++ (-1,0) ++ (0,-1) ;
        \foreach \p in {(0,0), (1,0), (2,0), (2,1), (3,1), (4,1), (4,2), (5,2),
        (6,2), (6,3), (5,3), (4,3), (3,3), (3,2), (2,2), (1,2), (1,1), (0,1)}
            \node[circled node] at \p {};
            \node at (-2,1.5) {$\bfrak_w=$};
    \end{tikzpicture}
\end{center}
    \caption{The basic matching $\bfrak$ has only boundary edges
    and its last edge is vertical.}
    \label{fig:the-basic-matching}
\end{figure}
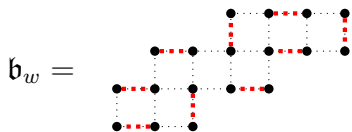

Notice that the first edge of the basic perfect matching is vertical
or horizontal depending on the parity of $|w|$, where $|w|$ denotes
the length of the word $w$; see Figure~\ref{fig:parity-of-w}.
In particular, the basic matching of $G(w)$ always belong to the set $M^\carrevide(w)$.

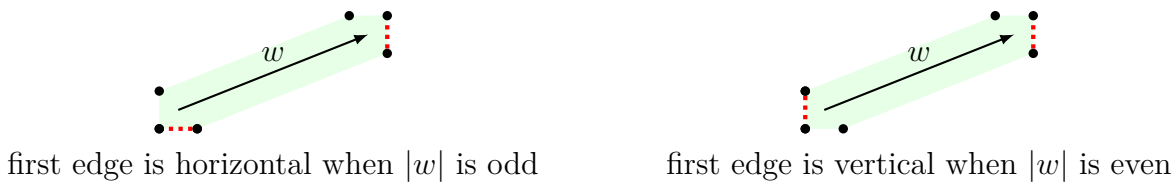
\begin{figure}[h]
\begin{center}
    \def\endX{6}
    \def\endY{3}
    \def\thesetup{
        \fill[\enclosedAreaColor] (0,0) -- (1,0) -- (\endX,\endY-1) -- (\endX,\endY) -- (\endX-1,\endY) -- (0,1) -- cycle;
        \draw[thebasicmatching] (\endX,\endY) to ++ (0,-1);
        \draw[thick,->,>=latex] (.5,.5) -- node[above] {$w$} (\endX-.5,\endY-.5);
        \foreach \p in {(0,0), (1,0), (0,1), (\endX,\endY-1), (\endX,\endY), (\endX-1,\endY)}
            \node[circled node] at \p {};
    }
    \begin{tikzpicture}[scale=.5,>=latex,auto]
        \begin{scope}[xshift=0cm,yshift=0cm]
            \thesetup
            \draw[thebasicmatching] (0,0) to (1,0);
            \node at (3,-1) {first edge is horizontal when $|w|$ is odd};
            \foreach \p in {(0,0), (1,0)} \node[circled node] at \p {};
        \end{scope}
        \begin{scope}[xshift=17cm,yshift=0cm]
            \thesetup
            \draw[thebasicmatching] (0,0) to (0,1);
            \node at (3,-1) {first edge is vertical when $|w|$ is even};
            \foreach \p in {(0,0), (0,1)} \node[circled node] at \p {};
        \end{scope}
    \end{tikzpicture}
\end{center}
    \caption{The orientation of the first edge of the basic matching of $G(w)$.}
    \label{fig:parity-of-w}
\end{figure}

\begin{remark}\label{rem:parallel-perp-interpretation}
A matching $m\in G(w)$ is a \emph{$\carrevide$-matching} if and only if its
    first edge belongs to the basic perfect matching
and a matching is a \emph{$\carreplein$-matching} if and only if its first edge does not
belong to the basic perfect matching.
\end{remark}

\subsection*{The area statistics}
Given $w\in\{\0,\1\}^*$ and two perfect matchings $m,m'\in M(w)$,
let $m\Delta m'=(m\setminus m')\cup(m'\setminus m)$ denote their symmetric difference.
We recall that the symmetric difference of sets $A$ and $B$
is defined as: $A\Delta B=(A\setminus B)\cup(B\setminus A)$.
The symmetric difference $m\Delta m'$ is the union of disjoint cycles
in the snake graph $G(w)$.
Since $G(w)$ is embedded in the plane $\R^2$,
the cycles in the symmetric difference defines a bounded region which
we denote by
$
    \region(m\Delta m')\subset\R^2.
$
A statistics which takes an important role is the
\emph{area} of this region which we denote by
\[
    \area(m\Delta m').
\]
Equivalently, it is equal to the area of the region enclosed
by the union $m\cup m'$.

Note that the region enclosed by $m\Delta m_-$
where $m_-$ is the minimal perfect matching was used
in the description of variables in cluster algebras from surfaces
\cite{MR2661414,MR3034481,zbMATH06144657}.
In this work, for every matching $m\in M(w)$, we rather consider the region
enclosed by $m\Delta\bfrak_w$ where $\bfrak_w$ is the basic perfect matching of
the snake graph $G(w)$; see Figure~\ref{fig:area}.
It is more convenient to use the basic perfect matching $\bfrak_w$
in the symmetric difference in order for the start of the snake graph
to correspond to the initial partial quotients in the continued fraction expansion
while having the first edge of the perfect matching used to obtain a dichotomy
in the set of perfect matchings contributing to the numerator or denominator
of the associated rational number.

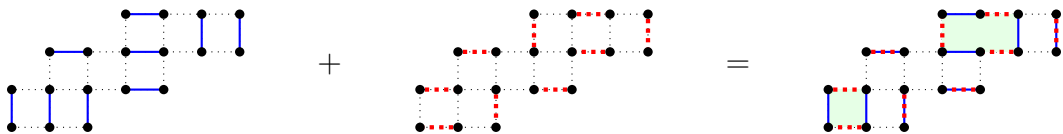
\begin{figure}[h]
\begin{center}
    \def\enclosedarea{
    \fill[\enclosedAreaColor] (0,0) rectangle ++ (1,1);
    \fill[\enclosedAreaColor] (3,2) rectangle ++ (2,1);}
    \def\thevertices{
    \foreach \p in {(0,0), (1,0), (2,0), (2,1), (3,1), (4,1), (4,2), (5,2),
    (6,2), (6,3), (5,3), (4,3), (3,3), (3,2), (2,2), (1,2), (1,1), (0,1)}
        \node[circled node] at \p {};}

    \def\thebasicmatching{
    \draw[thebasicmatching] (0,0) to ++ (1,0) ++ (1,0)
                                  to ++ (0,1) ++ (1,0)
                                  to ++ (1,0) ++ (0,1)
                                  to ++ (1,0) ++ (1,0)
                                  to ++ (0,1) ++ (-1,0)
                                  to ++ (-1,0) ++ (-1,0)
                                  to ++ (0,-1) ++ (-1,0)
                                  to ++ (-1,0) ++ (0,-1)
                                  to ++ (-1,0) ++ (0,-1) ;
                                  }
    \def\somematching{
        \draw[somematching] (0,0) -- ++ (0,1);
        \draw[somematching] (2,0) -- ++ (0,1);
        \draw[somematching] (1,0) -- ++ (0,1);
        \draw[somematching] (5,2) -- ++ (0,1);
        \draw[somematching] (6,2) -- ++ (0,1);
        \draw[somematching] (3,2) -- ++ (1,0);
        \draw[somematching] (3,3) -- ++ (1,0);
        \draw[somematching] (3,1) -- ++ (1,0);
        \draw[somematching] (1,2) -- ++ (1,0);}
    \def\thesnakegraph{
        \draw[dotted] (0,0) -- ++ (1,0) -- ++ (1,0) -- ++ (0,1) -- ++ (1,0) --
        ++ (1,0) -- ++ (0,1) -- ++ (1,0) -- ++ (1,0) -- ++ (0,1) -- ++ (-1,0)
        -- ++ (-1,0) -- ++ (-1,0) -- ++ (0,-1) -- ++ (-1,0) -- ++ (-1,0) -- ++
        (0,-1) -- ++ (-1,0) -- ++ (0,-1);
        \draw[dotted] (1,0) -- ++ (0,1);
        \draw[dotted] (2,1) -- ++ (0,1);
        \draw[dotted] (3,1) -- ++ (0,1);
        \draw[dotted] (4,2) -- ++ (0,1);
        \draw[dotted] (5,2) -- ++ (0,1);
        \draw[dotted] (1,1) -- ++ (1,0);
        \draw[dotted] (3,2) -- ++ (1,0);
    }

    \begin{tikzpicture}[scale=.5,>=latex]
        \thesnakegraph
        \somematching
        \thevertices
    \end{tikzpicture}
    \qquad
    \raisebox{8mm}{+}
    \qquad
    \begin{tikzpicture}[scale=.5,>=latex]
        \thesnakegraph
        \thebasicmatching
        \thevertices
    \end{tikzpicture}
    \qquad
    \raisebox{8mm}{=}
    \qquad
    \begin{tikzpicture}[scale=.5,>=latex]
        \thesnakegraph
        \enclosedarea
        \somematching
        \thebasicmatching
        \thevertices
    \end{tikzpicture}
\end{center}
\caption{A perfect matching on top of the basic matching
encloses a region of area $1+2=3$.}
\label{fig:area}
\end{figure}

The area of the basic perfect matching of a snake graph is zero.
The area of the other perfect matching with no interior edges is maximal,
see Figure~\ref{fig:boundary-matching-area}.
Note that the definition of area corresponds to the degree of the height
monomials of a perfect matching
\cite{MR2661414,MR3034481,zbMATH06144657,zbMATH08015631},
except that the height monomials are defined using the minimal perfect matching
$m_-$ instead of the basic perfect matching $\bfrak_w$.

\begin{figure}[h]
\begin{center}
    \begin{tikzpicture}[scale=.5,>=latex]
        \draw[somematching] (0,0)
            -- ++ (1,0)  ++ (1,0)
            -- ++ (0,1)  ++ (1,0)
            -- ++ (1,0)  ++ (0,1)
            -- ++ (1,0)  ++ (1,0)
            -- ++ (0,1)  ++ (-1,0)
            -- ++ (-1,0) ++ (-1,0)
            -- ++ (0,-1) ++ (-1,0)
            -- ++ (-1,0) ++ (0,-1)
            -- ++ (-1,0) ++ (0,-1) ;
        \def\thebasicmatching{
        \draw[thebasicmatching] (0,0) to ++ (1,0) ++ (1,0)
                                    to ++ (0,1) ++ (1,0)
                                    to ++ (1,0) ++ (0,1)
                                    to ++ (1,0) ++ (1,0)
                                    to ++ (0,1) ++ (-1,0)
                                    to ++ (-1,0) ++ (-1,0)
                                    to ++ (0,-1) ++ (-1,0)
                                    to ++ (-1,0) ++ (0,-1)
                                    to ++ (-1,0) ++ (0,-1) ;}
        \thebasicmatching
        \foreach \p in {(0,0), (1,0), (2,0), (2,1), (3,1), (4,1), (4,2), (5,2),
        (6,2), (6,3), (5,3), (4,3), (3,3), (3,2), (2,2), (1,2), (1,1), (0,1)}
            \node[circled node] at \p {};
    \end{tikzpicture}
    \qquad
    \qquad
    \begin{tikzpicture}[scale=.5,>=latex]
        \fill[\enclosedAreaColor] (0,0) -- ++ (2,0)
                        -- ++ (0,1)
                        -- ++ (2,0)
                        -- ++ (0,1)
                        -- ++ (2,0)
                        -- ++ (0,1)
                        -- ++ (-3,0)
                        -- ++ (0,-1)
                        -- ++ (-2,0)
                        -- ++ (0,-1)
                        -- ++ (-1,0)
                        -- ++ (0,-1) -- cycle;
        \draw[somematching] (0,0)
            ++ (1,0)  -- ++ (1,0)
            ++ (0,1)  -- ++ (1,0)
            ++ (1,0)  -- ++ (0,1)
            ++ (1,0)  -- ++ (1,0)
            ++ (0,1)  -- ++ (-1,0)
            ++ (-1,0) -- ++ (-1,0)
            ++ (0,-1) -- ++ (-1,0)
            ++ (-1,0) -- ++ (0,-1)
            ++ (-1,0) -- ++ (0,-1) ;
        \def\thebasicmatching{
        \draw[thebasicmatching] (0,0) to ++ (1,0) ++ (1,0)
                                    to ++ (0,1) ++ (1,0)
                                    to ++ (1,0) ++ (0,1)
                                    to ++ (1,0) ++ (1,0)
                                    to ++ (0,1) ++ (-1,0)
                                    to ++ (-1,0) ++ (-1,0)
                                    to ++ (0,-1) ++ (-1,0)
                                    to ++ (-1,0) ++ (0,-1)
                                    to ++ (-1,0) ++ (0,-1) ;}
        \thebasicmatching
        \foreach \p in {(0,0), (1,0), (2,0), (2,1), (3,1), (4,1), (4,2), (5,2),
        (6,2), (6,3), (5,3), (4,3), (3,3), (3,2), (2,2), (1,2), (1,1), (0,1)}
            \node[circled node] at \p {};
    \end{tikzpicture}
\end{center}
\caption{A perfect matching of minimal area 0 and maximal area.}
\label{fig:boundary-matching-area}
\end{figure}
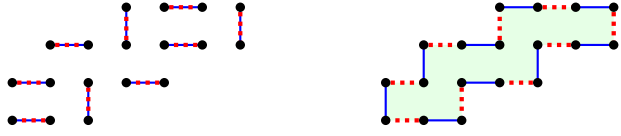

\begin{remark}\label{rem:carreplein-interpretation}
Let $m\in M(w)$.
Observe that $m\in M^\carreplein(w)$ if and only if
the unit square at the origin is inside the enclosed area.
Equivalently, $m\in M^\carrevide(w)$ if and only if
the unit square at the origin is not in the enclosed area.
\end{remark}
In particular,
$\area(m\Delta\bfrak_w) \geq 1$
for every matching $m\in M^\carreplein(w)$.
Note that this implies that
$q$ divides the polynomial $\sum_{m\in M^\carreplein(w)} q^{\area(m\Delta\bfrak_w)}$.

\subsection*{A graded poset}
For every $w\in\{\0,\1\}^*$,
the set of perfect matchings $M(w)$ is naturally equipped with a partial order
$(M(w), \prec)$ defined as follows:
\[
    m\prec m'
    \quad
    \text{ if and only if }
    \quad
    \region(m \Delta \bfrak_w)
    \subseteq
    \region(m' \Delta \bfrak_w)
\]
for every perfect matchings $m, m'\in M(w)$.
The rank function of the graded poset $(M(w), \prec)$ is the area of the perfect matchings.
The basic perfect matching $\bfrak_w$ is its minimal element.

\begin{example}
For example, let $w=\0\1\0\0$.
The basic perfect matching of $G(w)$ is
\[
\bfrak = \bfrak_w = \raisebox{-5mm}{
            \includegraphics{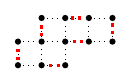}
            }.
\]
The region enclosed by $m\Delta\bfrak$
for every perfect matchings $m\in M^\carreplein(\0\1\0\0)$ is:
\[
\left\{\raisebox{-12mm}{\includegraphics{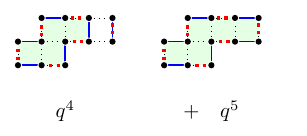}}\right\}
\]
Thus, $\sum_{m\in M^\carreplein(\0\1\0\0)} q^{\area(m\Delta\bfrak)}=q^4+q^5$.
The region enclosed by $m\Delta\bfrak$
for every perfect matchings $m\in M^\carrevide(\0\1\0\0)$ is:
\[
\left\{\raisebox{-20mm}{\includegraphics{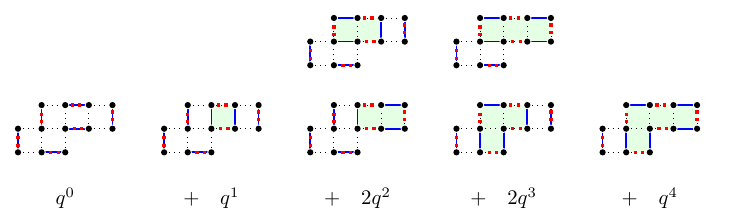}}\right\}
\]
Thus, $\sum_{m\in M^\carrevide(\0\1\0\0)} q^{\area(m\Delta\bfrak)}=1+q+2q^2+2q^3+q^4$.
We observe that
\[
    q^{-1}\frac{\sum_{m\in M^\carreplein(\0\1\0\0)} q^{\area(m\Delta\bfrak)}}
         {\sum_{m\in M^\carrevide (\0\1\0\0)} q^{\area(m\Delta\bfrak)} }
= q^{-1}\frac{q^4+q^5}{1+q+2q^2+2q^3+q^4}
= [0;3,1,1]_q .
\]
\end{example}

\section{A bijection from perfect matchings to fence poset order ideals}
\label{sec:bijection-perfect-matchings-to-order-ideals}

The main result of this section is the existence of a bijection
from the set of perfect matchings of a snake graph
to the set of order ideals of a fence poset
such that the area statistics of a perfect matching
correspond to the size (rank) of the order ideal.

The existence of such a bijection is not new; see for instance
Theorem 5.4 in \cite{zbMATH06144657},
Section 4.1 in \cite{zbMATH08015631},
Section 6 in \cite{claussen_expansion_2020},
Section 4 in \cite{burcroff_higher_2024},
Section 4 in \cite{zbMATH07181526}.
For completeness, but also because our definitions slightly differ from the
literature, we provide a proof of the bijection which works for all snake
graphs and for all fence posets and which involves the area statistics
defined from our specific choice of basic perfect matching.

For every word $w\in\{\0,\1\}^*$, the map is defined as:
\[
\begin{array}{rccl}
    \Phi:&M(w) & \to & J(F(\theta(w)))\\
    &m & \mapsto &
    \left\{|p| \colon w = ps \text{ and }
            \region(S_{\overrightarrow{p}})\subseteq \region(m\Delta\bfrak_w)
        \right\}.
\end{array}
\]
where $\bfrak_w$ is the basic matching of the snake graph $G(w)$,
that is, the matching using only boundary edges and whose last edge is vertical.
The subset inclusion
$\region(S_{\overrightarrow{p}})\subseteq \region(m\Delta\bfrak_w)$
means that the unit square $S_{\overrightarrow{p}}$ at $\overrightarrow{p}$
is inside the region delimited by a cycle of $m\Delta\bfrak_w$.
The map $\Phi$ is illustrated in Figure~\ref{fig:Phi-ex-with-fence-poset}.

\newcommand{\arrowR}{\tikz \draw[->] (-2pt,0) -- (2pt,0);}
\newcommand{\arrowL}{\tikz \draw[<-] (-2pt,0) -- (2pt,0);}
\tikzset{milieu/.style={sloped, pos=0.5, allow upside down}}

\begin{figure}[h]
\begin{center}
    \def\enclosedarea{
    \fill[\enclosedAreaColor]
                    (0,0) -- ++ (1,0)
                    -- ++ (0,2)
                    -- ++ (-1,0)
                    -- ++ (0,-2) -- cycle;
    \fill[\enclosedAreaColor] (2,3) -- ++ (1,0)
                    -- ++ (0,2)
                    -- ++ (2,0)
                    -- ++ (0,1)
                    -- ++ (-3,0)
                    -- ++ (0,-3) -- cycle;
    \fill[\enclosedAreaColor] (7,6) -- ++ (1,0)
                    -- ++ (0,2)
                    -- ++ (-1,0)
                    -- ++ (0,-2) -- cycle;
	}
    \def\thevertices{
    \foreach \p in {
    (0,0), (1,0),
    (0,1), (0,2),
    (1,1), (1,2), (1,3), (1,4),
    (2,1), (2,2), (2,3), (2,4), (2,5), (2,6),
    (3,3), (3,4), (3,5), (3,6),
    (4,5), (4,6),
    (5,5), (5,6), (5,7),
    (6,5), (6,6), (6,7),
    (7,6), (7,7), (7,8),
    (8,6), (8,7), (8,8)
    }
        \node[circled node] at \p {};}

    \def\thebasicmatching{
    \draw[thebasicmatching] (0,0) -- ++ (0,1)
        ++ (0,1)  -- ++ (1,0)
        ++ (0,1)  -- ++ (0,1)
        ++ (1,0)  -- ++ (0,1)
        ++ (0,1)  -- ++ (1,0)
        ++ (1,0)  -- ++ (1,0)
        ++ (0,1)  -- ++ (1,0)
        ++ (1,0)  -- ++ (0,1)
        ++ (1,0)  -- ++ (0,-1)
        ++ (0,-1)  -- ++ (-1,0)
        ++ (-1,0)  -- ++ (0,-1)
        ++ (-1,0)  -- ++ (-1,0)
        ++ (-1,0)  -- ++ (0,-1)
    	++ (0,-1)  -- ++ (-1,0)
	    ++ (0,-1)  -- ++ (0,-1)
	    ++ (-1,0)  -- ++ (0,-1)
       ;}
    \def\somematching{
        \draw[somematching] (0,0) -- (1,0);
        \draw[somematching] (0,1) -- (0,2);
        \draw[somematching] (1,1) -- (1,2);
        \draw[somematching] (2,1) -- (2,2);
        \draw[somematching] (1,3) -- (1,4);
        \draw[somematching] (2,3) -- (2,4);
        \draw[somematching] (3,3) -- (3,4);
        \draw[somematching] (3,5) -- (4,5);
        \draw[somematching] (5,5) -- (5,6);
        \draw[somematching] (2,5) -- (2,6);
        \draw[somematching] (3,6) -- (4,6);
        \draw[somematching] (6,5) -- (6,6);
        \draw[somematching] (5,7) -- (6,7);
        \draw[somematching] (7,6) -- (7,7);
        \draw[somematching] (8,6) -- (8,7);
        \draw[somematching] (7,8) -- (8,8);
      }
    \def\dottedlines{
        \draw[dotted] (0,1) -- + (1,0);
        \draw[dotted] (1,1) -- + (1,0);
        \draw[dotted] (1,2) -- + (1,0);
        \draw[dotted] (1,3) -- + (1,0);
        \draw[dotted] (1,4) -- + (1,0);
        \draw[dotted] (2,4) -- + (1,0);
        \draw[dotted] (2,5) -- + (1,0);
        \draw[dotted] (5,5) -- + (1,0);
        \draw[dotted] (5,6) -- + (1,0);
        \draw[dotted] (6,6) -- + (1,0);
        \draw[dotted] (6,7) -- + (1,0);
        \draw[dotted] (7,7) -- + (1,0);
        \draw[dotted] (1,2) -- + (0,1);
        \draw[dotted] (2,2) -- + (0,1);
        \draw[dotted] (3,5) -- + (0,1);
        \draw[dotted] (4,5) -- + (0,1);
        \draw[dotted] (5,6) -- + (0,1);
        \draw[dotted] (6,6) -- + (0,1);
      }

    \begin{tikzpicture}[scale=.5,>=latex]
        \enclosedarea
        \thevertices
        \thebasicmatching
        \somematching
        \dottedlines

        \begin{scope}[xshift=2cm,yshift=-1cm,
                      every node/.style={scale=0.9}]
            \foreach \x in {2,3,4,10,11,12}  {
                \node[empty circled node] (v\x) at (\x,0) {};
                \node[below=1mm]  at (\x,0) {$\x$};
            }
            \foreach \x in {0,1,5,6,7,8,9,13,14}  {
                \node[plein circled node] (v\x) at (\x,0) {};
                \node[below=1mm]  at (\x,0) {$\x$};
            }
            \draw (v0) -- (v1) node[milieu]{\arrowL};
            \draw (v1) -- (v2) node[milieu]{\arrowL};
            \draw (v2) -- (v3) node[milieu]{\arrowL};
            \draw (v3) -- (v4) node[milieu]{\arrowR};
            \draw (v4) -- (v5) node[milieu]{\arrowR};
            \draw (v5) -- (v6) node[milieu]{\arrowR};
            \draw (v6) -- (v7) node[milieu]{\arrowL};
            \draw (v7) -- (v8) node[milieu]{\arrowL};
            \draw (v8) -- (v9) node[milieu]{\arrowR};
            \draw (v9) -- (v10) node[milieu]{\arrowL};
            \draw (v10) -- (v11) node[milieu]{\arrowL};
            \draw (v11) -- (v12) node[milieu]{\arrowL};
            \draw (v12) -- (v13) node[milieu]{\arrowR};
            \draw (v13) -- (v14) node[milieu]{\arrowR};
        \end{scope}

        \draw[dashed,thick,shorten <= 3mm,shorten >= 2mm] (v0) -- (1,0);
        \draw[dashed,thick,shorten <= 3mm,shorten >= 2mm] (v1) -- (1,1);
        \draw[dashed,thick,shorten <= 3mm,shorten >= 2mm] (v5) -- (3,3);
        \draw[dashed,thick,shorten <= 3mm,shorten >= 2mm] (v6) -- (3,4);
        \draw[dashed,thick,shorten <= 3mm,shorten >= 2mm] (v7) -- (3,5);
        \draw[dashed,thick,shorten <= 3mm,shorten >= 2mm] (v8) -- (4,5);
        \draw[dashed,thick,shorten <= 3mm,shorten >= 2mm] (v9) -- (5,5);
        \draw[dashed,thick,shorten <= 3mm,shorten >= 2mm] (v13) -- (8,6);
        \draw[dashed,thick,shorten <= 3mm,shorten >= 2mm] (v14) -- (8,7);
        \draw[dotted,thick,shorten <= 3mm,shorten >= 2mm] (v2) -- (2,1);
        \draw[dotted,thick,shorten <= 3mm,shorten >= 2mm] (v3) -- (2,2);
        \draw[dotted,thick,shorten <= 3mm,shorten >= 2mm] (v4) -- (2,3);
        \draw[dotted,thick,shorten <= 3mm,shorten >= 2mm] (v10) -- (6,5);
        \draw[dotted,thick,shorten <= 3mm,shorten >= 2mm] (v11) -- (6,6);
        \draw[dotted,thick,shorten <= 3mm,shorten >= 2mm] (v12) -- (7,6);

        \begin{scope}[xshift=2cm,yshift=-8cm,
                      every node/.style={scale=0.9}]
            \foreach \x/\h in {2/2,3/3,4/2,10/2,11/3,12/4}  {
                \node[empty circled node] (f\x) at (\x,\h) {};
            }
            \foreach \x/\h in {0/0,1/1,5/1,6/0,7/1,8/2,9/1,13/3,14/2}  {
                \node[plein circled node] (f\x) at (\x,\h) {};
            }
            \draw (f0)  -- (f1)  ;
            \draw (f1)  -- (f2)  ;
            \draw (f2)  -- (f3)  ;
            \draw (f3)  -- (f4)  ;
            \draw (f4)  -- (f5)  ;
            \draw (f5)  -- (f6)  ;
            \draw (f6)  -- (f7)  ;
            \draw (f7)  -- (f8)  ;
            \draw (f8)  -- (f9)  ;
            \draw (f9)  -- (f10) ;
            \draw (f10) -- (f11) ;
            \draw (f11) -- (f12) ;
            \draw (f12) -- (f13) ;
            \draw (f13) -- (f14) ;
        \end{scope}

        \draw[dashed,thick,shorten <= 8mm,shorten >= 3mm] (v0)  -- (f0) ;
        \draw[dashed,thick,shorten <= 8mm,shorten >= 3mm] (v1)  -- (f1) ;
        \draw[dashed,thick,shorten <= 8mm,shorten >= 3mm] (v5)  -- (f5) ;
        \draw[dashed,thick,shorten <= 8mm,shorten >= 3mm] (v6)  -- (f6) ;
        \draw[dashed,thick,shorten <= 8mm,shorten >= 3mm] (v7)  -- (f7) ;
        \draw[dashed,thick,shorten <= 8mm,shorten >= 3mm] (v8)  -- (f8) ;
        \draw[dashed,thick,shorten <= 8mm,shorten >= 3mm] (v9)  -- (f9) ;
        \draw[dashed,thick,shorten <= 8mm,shorten >= 3mm] (v13) -- (f13);
        \draw[dashed,thick,shorten <= 8mm,shorten >= 3mm] (v14) -- (f14);
        \draw[dotted,thick,shorten <= 8mm,shorten >= 3mm] (v2)  -- (f2) ;
        \draw[dotted,thick,shorten <= 8mm,shorten >= 3mm] (v3)  -- (f3) ;
        \draw[dotted,thick,shorten <= 8mm,shorten >= 3mm] (v4)  -- (f4) ;
        \draw[dotted,thick,shorten <= 8mm,shorten >= 3mm] (v10) -- (f10);
        \draw[dotted,thick,shorten <= 8mm,shorten >= 3mm] (v11) -- (f11);
        \draw[dotted,thick,shorten <= 8mm,shorten >= 3mm] (v12) -- (f12);

    \end{tikzpicture}
\end{center}
\caption{The image of a perfect matching of the snake graph
    $G(\1\0\1\1\0\1\1\0\0\0\1\0\0\1)$
    under the map $\Phi$ is the order ideal
    $\{0,1,5,6,7,8,9,13,14\}$
    of the fence poset
    $F(\theta(\1\0\1\1\0\1\1\0\0\0\1\0\0\1))
    =F(       \1\1\1\0\0\0\1\1\0\1\1\1\0\0)$.}
\label{fig:Phi-ex-with-fence-poset}
\end{figure}
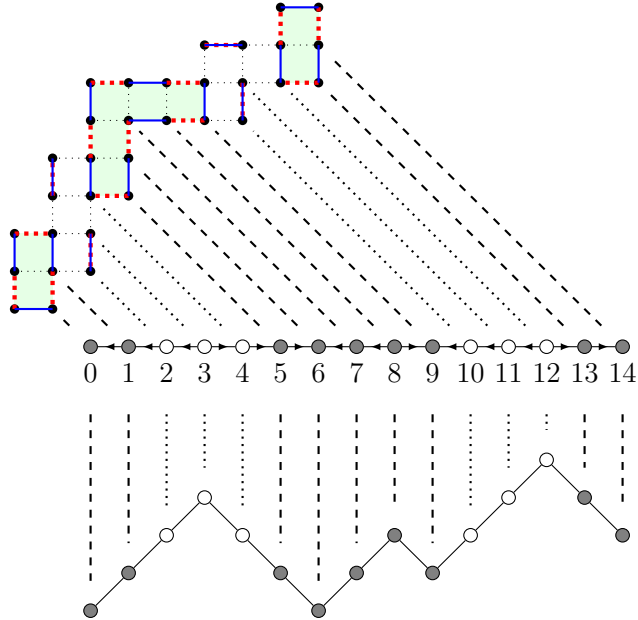

We show that the map $\Phi$ is a well-defined bijection
from the set of perfect matchings $M(w)$ to the set of fence poset order ideals $J(F(\theta(w)))$.

\begin{theorem}\label{theo:bij}
Let $w\in\{\0,\1\}^*$.
The map $\Phi:(M(w),\prec)\to (J(F(\theta(w))), \subseteq)$
is an order-preserving bijection from
the set $M(w)$ of perfect matchings of the snake graph $G(w)$
to
    the set $J(F(\theta(w)))$ of order ideals of the fence poset $F(\theta(w))$
such that:
\begin{itemize}
    \item $m\in M^\carreplein(w)$ if and only if $\Phi(m)\in J^\bullet(F(\theta(w)))$;
    \item $m\in M^\carrevide(w)$ if and only if  $\Phi(m)\in J^\circ(F(\theta(w)))$;
    \item for every $m\in M(w)$, we have $\area(m\Delta\bfrak_w) = \size{\Phi(m)}.$
\end{itemize}
\end{theorem}

\def\firstVertEdge{u^|}
\def\firstHorizEdge{u^-}

\begin{proof}
    Let $\firstHorizEdge=(0,e_1)$ denote the horizontal bottom edge and
    $\firstVertEdge=(0,e_2)$ denote the vertical left edge of the first unit square
    of the snake graph.
    Let $S$ be the four edges $\{\firstHorizEdge, \firstVertEdge, (e_1,e_1+e_2), (e_2,e_1+e_2)\}$
    of the first unit square at the origin.
    Using $S$, we define a function from $M(\alpha w)$ to $M(w)$
    which erases the information in the first unit square of the snake graph $G(w)$:
    \[
    \begin{array}{rccl}
        \pop:& M(\alpha w)&\to& M(w)\\
        &m&\mapsto &
        \begin{cases}
            m\setminus\{\firstVertEdge\} - e_1
            & \text{ if } \alpha = \0 \text{ and } \firstVertEdge\in m,\\
            (m\Delta S)\setminus\{\firstVertEdge\} - e_1
            & \text{ if } \alpha = \0 \text{ and } \firstHorizEdge\in m,\\
            m\setminus\{\firstHorizEdge\} - e_2
            & \text{ if } \alpha = \1 \text{ and } \firstHorizEdge\in m,\\
            (m\Delta S)\setminus\{\firstHorizEdge\} - e_2
            & \text{ if } \alpha = \1 \text{ and } \firstVertEdge\in m.
        \end{cases}
    \end{array}
    \]
    Using $\pop$, we can equivalently define the map $\Phi$ by induction
    as follows:
    \[
    \begin{array}{rccl}
        \Phi:&M(w) & \to & J(F(\theta(w)))\\
        &m & \mapsto &
        \begin{cases}
            \varnothing
            & \text{ if } w=\varepsilon \text{ and } m \in M^\carrevide(\varepsilon),\\
            \{0\}
            & \text{ if } w=\varepsilon \text{ and } m \in M^\carreplein(\varepsilon),\\
            \Phi(\pop(m))+1
            & \text{ if } w\neq\varepsilon \text{ and } m \in M^\carrevide(w),\\
            \{0\} \cup (\Phi(\pop(m))+1)
            & \text{ if } w\neq\varepsilon \text{ and } m \in M^\carreplein(w).
        \end{cases}
    \end{array}
    \]

    The proof that $\Phi$ is a bijection is done by induction on the length of the word $w$.
    Suppose that $\Phi:M(w)\to J(F(\theta(w)))$
    is a bijection for every word $w\in\{\0,\1\}^*$ of length $n$.
    Let $\alpha\in\{\0,\1\}$ be some letter.
    We want to show that $\Phi:M(\alpha w)\to J(F(\theta(\alpha w)))$
    is also a bijection.

    First we prove that $\Phi$ is \textbf{well-defined}.
    Let $m\in M(\alpha w)$
    and $\beta\in\{\0,\1\}$ be the first letter of
    $\theta(\alpha w)=\beta\theta(w)$.
    From the induction hypothesis,
    we have that $\Phi(\pop(m))$ is an order ideal in $J(F(\theta(w)))$.
    Thus, $\Phi(\pop(m))+1$ is an order ideal in $J(F(\theta(\alpha w)))$.
    It remains to show that $\Phi(m)$ is also an order ideal.
    This depends on $\beta$ and the subset $\{0,1\}\cap\Phi(m)$.
    Suppose that $\beta=\1$.
    We want to show that $1\in\Phi(m)$ implies that $0\in\Phi(m)$.
    Suppose by contradiction that $1\in\Phi(m)$ and $0\notin\Phi(m)$.
    There are two cases to consider:
    \begin{itemize}
        \item Case $|\alpha w|$ is even.
            From the definition of $\theta$, we must have $\alpha=\1$.
            Since $1\in\Phi(m)$ and $0\notin\Phi(m)$, we have the interior edge
            $(e_2,e_1+e_2)\in m$.
            Since $|\alpha w|$ is even, we have $\firstVertEdge\in\bfrak_{\alpha w}$.
            Since $0\notin\Phi(m)$, we have $m\in M^\carrevide(\alpha w)$.
            Thus, $\firstVertEdge\in m$ as well.
            This is a contradiction, because vertex $e_2$ may not be paired
            with two distinct vertices in the perfect matching $m$.
        \item Case $|\alpha w|$ is odd.
            From the definition of $\theta$, we must have $\alpha=\0$.
            Since $1\in\Phi(m)$ and $0\notin\Phi(m)$, we have the interior edge
            $(e_1,e_1+e_2)\in m$.
            Since $|\alpha w|$ is odd, we have $\firstHorizEdge\in\bfrak_{\alpha w}$.
            Since $0\notin\Phi(m)$, we have $m\in M^\carrevide(\alpha w)$.
            Thus, $\firstHorizEdge\in m$ as well.
            This is a contradiction, because vertex $e_1$ may not be paired
            with two distinct vertices in the perfect matching $m$.
    \end{itemize}
    Suppose that $\beta=\0$.
    We need to show that $0\in\Phi(m)$ implies that $1\in\Phi(m)$.
    As above, we may show that assuming $0\in\Phi(m)$ and $1\notin\Phi(m)$
    leads to a contradiction.
    We conclude that the map $\Phi$ is well-defined.

    \textbf{$\Phi$ is surjective.}
    Let $I\in J(F(\theta(\alpha w)))$ be an order ideal.
    We have $(I\setminus\{0\})-1 \in J(F(\theta(w)))$.
    By induction, there exists a perfect matching $m\in M(w)$ such that
    $\Phi(m)=(I\setminus\{0\})-1$.
    Let $m'\in M(\alpha w)$ be a perfect matching of $G(\alpha w)$ defined as follows:
    \[
        m'=
        \begin{cases}
            (m+e_1)\cup \{\firstVertEdge\}
                             & \text{ if } \alpha=\0, 0\in I \text{ and $|w|$ is even}\\
                             & \text{ or } \alpha=\0, 0\notin I \text{ and $|w|$ is odd},\\
            ((m+e_1)\cup \{\firstVertEdge\})\Delta S
                             & \text{ if } \alpha=\0, 0\in I \text{ and $|w|$ is odd}\\
                             & \text{ or } \alpha=\0, 0\notin I \text{ and $|w|$ is even},\\
            ((m+e_2)\cup \{\firstHorizEdge\})\Delta S
                             & \text{ if } \alpha=\1, 0\in I \text{ and $|w|$ is even}\\
                             & \text{ or } \alpha=\1, 0\notin I \text{ and $|w|$ is odd},\\
            (m+e_2)\cup \{\firstHorizEdge\}
                             & \text{ if } \alpha=\1, 0\in I \text{ and $|w|$ is odd}\\
                             & \text{ or } \alpha=\1, 0\notin I \text{ and $|w|$ is even}.
        \end{cases}
    \]
    This matching satisfies that $\pop(m')=m$.
    If $0\in I$, we have that $m'\in M^\carreplein(\alpha w)$ and
    \begin{align*}
        \Phi(m')
        &= \{0\} \cup (\Phi(\pop(m'))+1)
         = \{0\} \cup (\Phi(m)+1)\\
        &= \{0\} \cup ((I\setminus\{0\})-1+1)
         = \{0\} \cup (I\setminus\{0\})
        = I.
    \end{align*}
    If $0\notin I$, we have that $m'\in M^\carrevide(\alpha w)$ and
    \begin{align*}
        \Phi(m')
        &= \Phi(\pop(m'))+1
         = \Phi(m)+1\\
        &= (I\setminus\{0\})-1+1
         = I\setminus\{0\}
         = I.
    \end{align*}
    We conclude that $\Phi$ is surjective.

    \textbf{$\Phi$ is injective.}
    Let $m_1,m_2\in M(\alpha w)$ be two perfect matchings such that $\Phi(m_1)=\Phi(m_2)$.
    We have
    \begin{align*}
        \Phi(\pop(m_1))
        &=\Phi(\pop(m_1))+1-1
        =\Phi(m_1)\setminus\{0\}-1\\
        &=\Phi(m_2)\setminus\{0\}-1
        =\Phi(\pop(m_2))+1-1
        =\Phi(\pop(m_2)).
    \end{align*}
    We may now use the induction hypothesis.
    Since $\Phi:M(w)\to J(F(\theta(w)))$ is injective
    and $\pop(m_1),\pop(m_2)\in M(w)$,
    we must have that $\pop(m_1)=\pop(m_2)$.
    If $0\in \Phi(m_1) =\Phi(m_2)$,
    then $m_1,m_2\in M^\carreplein(w)$.
    If $0\notin \Phi(m_1) =\Phi(m_2)$,
    then $m_1,m_2\in M^\carrevide(w)$.
    In particular, the first edge of $m_1$ is the same as
    the first edge of $m_2$.
    Assume $w$ starts with $\0$.
    If the first edge of $m_1$ and $m_2$ is $\firstVertEdge$, then
    $m_1\setminus\{\firstVertEdge\}-e_1
    =\pop(m_1)
    =\pop(m_2)
    =m_2\setminus\{\firstVertEdge\}-e_1$.
    Thus, $m_1=m_2$.
    If the first edge of $m_1$ and $m_2$ is $\firstHorizEdge$, then
    $(m_1\Delta S)\setminus\{\firstVertEdge\}-e_1
    =\pop(m_1)
    =\pop(m_2)
    =(m_2\Delta S)\setminus\{\firstVertEdge\}-e_1$.
    Thus, $m_1=m_2$.
    We reach the same conclusion if we assume $w$ that starts with $\1$.
    We conclude that $\Phi$ is injective.

    \textbf{$\Phi$ is order-preserving.}
    Let $m, m'\in M(w)$ be two perfect matchings.
    We have
    $m\prec m'$
    if and only if
    $\region(m \Delta \bfrak_w)
    \subseteq
    \region(m' \Delta \bfrak_w)$
    if and only if
    $\Phi(m)\subset\Phi(m')$.

    \textbf{Items 1 and 2.}
    By definition of $\Phi$, we have
    $m\in M^\carreplein(w)$
    if and only if $0\in\Phi(m)$
    if and only if $\Phi(m)\in J^\bullet(F(\theta(w)))$.
    Since $M(w)=M^\carreplein(w)\cup M^\carrevide(w)$ is a disjoint union,
    this also implies that
    $m\in M^\carrevide(w)$
    if and only if $0\notin\Phi(m)$
    if and only if $\Phi(m)\in J^\circ(F(\theta(w)))$.

    \textbf{Item 3.}
    It follows from the definition of $\Phi$
    that $\area(m\Delta\bfrak_w) = \size{\Phi(m)}$
    for every perfect matching $m\in M(w)$.
\end{proof}

Since $\theta$ is an involution on $\{\0,\1\}^*$,
Theorem~\ref{theo:bij}
also means that the map $\Phi$
is a bijection $M(\theta(w))\to J(F(w))$
from
the set of perfect matchings of the snake graph $G(\theta(w))$
to the set of order ideals of the fence poset $F(w)$.

\begin{corollary}\label{cor:poset-isomorphism-matchings-to-order-ideals}
For every $w\in\{\0,\1\}^*$,
the posets $(M(\theta(w)),\prec)$ and $(J(F(w)), \subseteq)$
are isomorphic.
\end{corollary}

\begin{proof}
    It follows from the order-preserving bijection proved in
    Theorem~\ref{theo:bij}.
\end{proof}

\section{Proof of Theorem~\ref{thm:nice-formula-for-qmatchings}
and Corollary~\ref{cor:three-combinatorial-interpretation-of-q-rational}}
\label{sec:proof-of-thm:nice-formula-for-qmatchings}

Similarly to Proposition~\ref{prop:statistics-over-order-ideals},
we have the following result involving the map $\theta$
as a consequence of Theorem~\ref{theo:bij}.

\begin{proposition}\label{prop:statistics-over-snake-graph-Gw}
    Let $w\in\{\0,\1\}^*$.
The area statistics over the set of perfect matchings
of the snake graph $G(w)$ satisfies
\begin{equation*}
\left(\begin{array}{r}
    \sum_{m\in M^\carreplein(w)} q^{\area(m\Delta\bfrak)}\\
        \sum_{m\in M^\carrevide (w)} q^{\area(m\Delta\bfrak)}
\end{array}\right)
    =
    \left(\begin{smallmatrix} 1 & 0 \\ 0 & q \end{smallmatrix}\right)^{-1}
            \nu_q(\theta(w))
        \left(\begin{array}{c} q \\ q \end{array}\right).
\end{equation*}
where $\bfrak$ is the basic perfect matching of the snake graph $G(w)$.
\end{proposition}

\begin{proof}
    For every $w\in\{\0,\1\}^*$,
    let
    $\Phi:M(w)\to J(F(\theta w))$ be the bijection from
    the set of perfect matchings of the snake graph $G(w)$
    to the set of order ideals of the fence poset $F(\theta w)$
defined in Theorem~\ref{theo:bij}.
Using Proposition~\ref{prop:statistics-over-order-ideals},
we have
\begin{align*}
    \left(\begin{array}{r}
        \sum_{m\in M^\carreplein(w)} q^{\area(m\Delta\bfrak)}\\
        \sum_{m\in M^\carrevide(w)}  q^{\area(m\Delta\bfrak)}
    \end{array}\right)
    &=
    \left(\begin{array}{r}
            \sum_{I\in J^\bullet(F(\theta w))} q^{\area(\Phi^{-1}(I)\Delta\bfrak)}\\
            \sum_{I\in J^\circ  (F(\theta w))} q^{\area(\Phi^{-1}(I)\Delta\bfrak)}
    \end{array}\right)\\
    &=
    \left(\begin{array}{r}
          \sum_{I\in J^\bullet(F(\theta w))} q^{\size{I}}\\
          \sum_{I\in J^\circ  (F(\theta w))} q^{\size{I}}
    \end{array}\right)%
    =
    \left(\begin{smallmatrix} 1 & 0 \\ 0 & q \end{smallmatrix}\right)^{-1}
            \nu_q(\theta w)
        \left(\begin{array}{c} q \\ q \end{array}\right).
\end{align*}
\end{proof}

    The following result says that the area statistics of the perfect
    $\carrevide$-matchings
    and $\carreplein$-matchings
    of a snake graph is related
    to the $q$-analog of a continued fraction expansion of a rational number.

\begin{THEOREMD}
    \MainTheoremD
\end{THEOREMD}

\begin{proof}
    Let $a=[a_0;a_1,\dots,a_{2\ell-1}]$ be the even-length continued fraction expansion
    of a positive rational number $\xx>0$.
    Let $w=W(a)=\1^{a_0}\0^{a_{1}}\cdots \1^{a_{2\ell-2}}\0^{a_{2\ell-1}-1}$.
    Using Proposition~\ref{prop:statistics-over-snake-graph-Gw}
    and Lemma~\ref{lem:sufficient-conditions-qcontinued-fraction},
    we obtain
    \begin{align*}
    \left(\begin{array}{r}
        \sum_{m\in \Mcal^\carreplein(\xx)} q^{\area(m\Delta\bfrak)}\\
            \sum_{m\in \Mcal^\carrevide (\xx)} q^{\area(m\Delta\bfrak)}
    \end{array}\right)
        &=
    \left(\begin{array}{r}
        \sum_{m\in M^\carreplein(\theta(W(a)))} q^{\area(m\Delta\bfrak)}\\
            \sum_{m\in M^\carrevide (\theta(W(a)))} q^{\area(m\Delta\bfrak)}
    \end{array}\right)\\
        &= \left(\begin{smallmatrix} 1 & 0 \\ 0 & q \end{smallmatrix}\right)^{-1}
                \nu_q(\theta^2(W(a)))
            \left(\begin{smallmatrix} q \\ q \end{smallmatrix}\right)\\
        &= \left(\begin{smallmatrix} 1 & 0 \\ 0 & q \end{smallmatrix}\right)^{-1}
                \nu_q(W(a))
            \left(\begin{smallmatrix} q \\ q \end{smallmatrix}\right)\\
        &=
        \left(\begin{smallmatrix} 1 & 0 \\ 0 & q \end{smallmatrix}\right)^{-1}
        R_q^{a_0}L_q^{a_1}\cdots R_q^{a_{2\ell-2}}L_q^{a_{2\ell-1}}
        \left(\begin{smallmatrix} 1 \\ 0 \end{smallmatrix}\right).
    \end{align*}
\end{proof}

Also, in terms of the Morier-Genoud--Ovsienko $q$-analog of rational numbers,
the following result holds in general.

\begin{COROLLARYE}
    \MainCorollaryE
\end{COROLLARYE}

\begin{proof}
    Let $a=[a_0;\dots,a_{2\ell-1}]$ be the even-length continued fraction expansion
    of a positive rational number $\frac{r}{s}>0$.
    Let $R(q)$ and $S(q)$ be the polynomials such that
    the Morier-Genoud--Ovsienko $q$-analog of the rational number $\frac{r}{s}$ is
    \[
        \left[\frac{r}{s}\right]_q
        =
        \frac{R(q)}{S(q)}.
    \]
    Using Theorem~\ref{thm:nice-formula-for-qmatchings}, we have
    \begin{align*}
        \left(\begin{array}{r}
            \sum_{m\in \Mcal^\carreplein(\xx)} q^{\area(m\Delta\bfrak)}\\
              \sum_{m\in \Mcal^\carrevide (\xx)} q^{\area(m\Delta\bfrak)}
        \end{array}\right)
            &=
        \left(\begin{smallmatrix} 1 & 0 \\ 0 & q \end{smallmatrix}\right)^{-1}
        R_q^{a_0}L_q^{a_1}\cdots R_q^{a_{2\ell-2}}L_q^{a_{2\ell-1}}
        \left(\begin{smallmatrix} 1 \\ 0 \end{smallmatrix}\right)\\
        &=
        \left(\begin{smallmatrix} 1 & 0 \\ 0 & q \end{smallmatrix}\right)^{-1}
        q\cdot q^{-1} R_q^{a_0}L_q^{a_1}\cdots R_q^{a_{2\ell-2}}L_q^{a_{2\ell-1}}
            \left(\begin{array}{c} 1 \\ 0 \end{array}\right)\\
        &\overset{\eqref{eq:CF-expansion-LqRq}}{=}
        \left(\begin{smallmatrix} q & 0 \\ 0 & 1 \end{smallmatrix}\right)
        \left(\begin{array}{c} R(q) \\ S(q) \end{array}\right)
        =
        \left(\begin{array}{c} q\cdot R(q) \\ S(q) \end{array}\right).
    \end{align*}
    We obtain
    \[
        \left[\frac{r}{s}\right]_q
        = \frac{R(q)}{S(q)}
        = \frac{q^{-1}\sum_{m\in M^\carreplein(\xx)} q^{\area(m\Delta\bfrak)}}
                     {\sum_{m\in M^\carrevide (\xx)} q^{\area(m\Delta\bfrak)}}.
    \]
    The two other equalities follow from
    Theorem~\ref{thm:1-norm-statistics} and
    Theorem~\ref{thm:nice-formula-for-order-ideals}.
\end{proof}

A combinatorial interpretation of rational numbers
\cite[Theorem 3.4]{MR3778183},
stated here as Theorem~\ref{thm:schiffler},
and of their $q$-analogs
\cite[Theorem 4]{MR4073883}
was already provided
in terms of order ideals of fence posets and snake graphs.
All of these formulas are involving different combinatorial
objects for the numerator and denominator
and are restricted to rational numbers larger than one.
We can deduce their results as a consequence of ours.

\begin{corollary}\label{cor:the-usual-combinatorial-interpretation}
    Let $\xx>1$ be a rational number
    whose even-length continued fraction expansion
    is $a=[a_0;a_1,\dots,a_{2\ell-1}]$.
    Then,
    \begin{equation} \label{eq:usual-combinatorial-interpretations}
\left[\xx\right]_q
=
    \frac{\sum_{b\in \Bcal(\xx')} q^{\Vert b\Vert_1}}
     {\sum_{b\in \Bcal(\xx'')  } q^{\Vert b\Vert_1}}
=
    \frac{\sum_{I\in \Jcal(\xx')} q^{\size{I}}}
     {\sum_{I\in \Jcal(\xx'')  } q^{\size{I}}}
=
    \frac{\sum_{m\in \Mcal(\xx')} q^{\area(m\Delta\bfrak')}}
     {\sum_{m\in \Mcal(\xx'')  } q^{\area(m\Delta\bfrak'')}}
\end{equation}
    where
    \begin{align*}
        \xx' &=\xx-1&&=[a_0-1;a_1,a_2,\dots,a_{2\ell-1}],\\
        \xx''&=((\xx-\lfloor \xx\rfloor)^{-1}-1)^{-1}&&=[0;a_1-1,a_2,\dots,a_{2\ell-1}],
    \end{align*}
          and $\bfrak'$ (resp. $\bfrak''$)
          is the basic perfect matching of the snake graph
          $\Gcal(\xx')$ (resp. $\Gcal(\xx'')$).
\end{corollary}

\begin{proof}
    Let $\xx>1$ be a rational number
    whose even-length continued fraction expansion
    is $a=[a_0;a_1,\dots,a_{2\ell-1}]$.
    Since $\xx>1$, we have $a_0>0$.
    Below, we use the following identity
    \begin{equation}\label{eq:prefix-identity}
        q^{-1}
        \left(\begin{smallmatrix} 1 & 0 \end{smallmatrix}\right)
        \left(\begin{smallmatrix} 1 & 0 \\ 0 & q \end{smallmatrix}\right)^{-1}
            R_q
            = \left(\begin{smallmatrix} 1 & 1 \end{smallmatrix}\right)
              \left(\begin{smallmatrix} 1 & 0 \\ 0 & q \end{smallmatrix}\right)^{-1}
            =
        \left(\begin{smallmatrix} 0 & 1 \end{smallmatrix}\right)
        \left(\begin{smallmatrix} 1 & 0 \\ 0 & q \end{smallmatrix}\right)^{-1}
        R_q^{a_0}L_q.
    \end{equation}
    For the numerator of $[\xx]_q$
    expressed in Corollary~\ref{cor:three-combinatorial-interpretation-of-q-rational},
    we can rewrite it as follows
    using Theorem~\ref{thm:nice-formula-for-order-ideals} and \eqref{eq:prefix-identity}:
    \begingroup
    \allowdisplaybreaks
    \begin{align*}
        q^{-1}\sum_{I\in \Jcal^\bullet(\xx)} q^{\size{I}}
        &=
        q^{-1}
        \left(\begin{smallmatrix} 1 & 0 \end{smallmatrix}\right)
        \left(\begin{smallmatrix} 1 & 0 \\ 0 & q \end{smallmatrix}\right)^{-1}
            R_q\cdot
        R_q^{a_0-1}L_q^{a_1}\cdots R_q^{a_{2\ell-2}}L_q^{a_{2\ell-1}}
        \left(\begin{smallmatrix} 1 \\ 0 \end{smallmatrix}\right)\\
        &=
        \left(\begin{smallmatrix} 1 & 1 \end{smallmatrix}\right)
        \left(\begin{smallmatrix} 1 & 0 \\ 0 & q \end{smallmatrix}\right)^{-1}
            \cdot
        R_q^{a_0-1}L_q^{a_1}\cdots R_q^{a_{2\ell-2}}L_q^{a_{2\ell-1}}
        \left(\begin{smallmatrix} 1 \\ 0 \end{smallmatrix}\right)\\
        &=
          \sum_{I\in \Jcal^\bullet[a_0-1;a_1,\dots,a_{2\ell-1}]} q^{\size{I}}
        + \sum_{I\in \Jcal^\circ[a_0-1;a_1,\dots,a_{2\ell-1}]} q^{\size{I}}\\
        &=
        \sum_{I\in \Jcal[a_0-1;a_1,\dots,a_{2\ell-1}]} q^{\size{I}}
        = \sum_{I\in \Jcal(\xx')} q^{\size{I}}.
    \end{align*}
    \endgroup
    For the denominator of $[\xx]_q$
    expressed in Corollary~\ref{cor:three-combinatorial-interpretation-of-q-rational},
    we can rewrite it as follows
    using Theorem~\ref{thm:nice-formula-for-order-ideals} and \eqref{eq:prefix-identity}:
    \begingroup
    \allowdisplaybreaks
    \begin{align*}
        \sum_{I\in \Jcal^\circ  (\xx)} q^{\size{I}}
        &=
        \left(\begin{smallmatrix} 0 & 1 \end{smallmatrix}\right)
        \left(\begin{smallmatrix} 1 & 0 \\ 0 & q \end{smallmatrix}\right)^{-1}
        R_q^{a_0}L_q\cdot
        L_q^{a_1-1}\cdots R_q^{a_{2\ell-2}}L_q^{a_{2\ell-1}}
        \left(\begin{smallmatrix} 1 \\ 0 \end{smallmatrix}\right)\\
        &=
        \left(\begin{smallmatrix} 1 & 1 \end{smallmatrix}\right)
        \left(\begin{smallmatrix} 1 & 0 \\ 0 & q \end{smallmatrix}\right)^{-1}
            \cdot
        L_q^{a_1-1}\cdots R_q^{a_{2\ell-2}}L_q^{a_{2\ell-1}}
        \left(\begin{smallmatrix} 1 \\ 0 \end{smallmatrix}\right)\\
        &=
          \sum_{I\in \Jcal^\bullet[0;a_1-1,\dots,a_{2\ell-1}]} q^{\size{I}}
        + \sum_{I\in \Jcal^\circ  [0;a_1-1,\dots,a_{2\ell-1}]} q^{\size{I}}\\
        &=
        \sum_{I\in \Jcal[0;a_1-1,\dots,a_{2\ell-1}]} q^{\size{I}}
        = \sum_{I\in \Jcal(\xx'')} q^{\size{I}}.
    \end{align*}
    \endgroup
    Notice that if $a_1=1$, then $\xx''=[a_2;a_3,\dots,a_{2\ell-1}]$.
    The other equalities are deduced from the order-preserving bijections
    proved in Theorem~\ref{thm:bijection-order-ideals-to-integersequences}
    and Theorem~\ref{theo:bij}.
\end{proof}

One may observe that the fence posets $\Fcal(\xx')$ and $\Fcal(\xx'')$ are
the two oriented path graphs used to define the numerator and denominator
formulas for $[\xx]_q$ in Theorem~4 of  \cite{MR4073883}.
Here is a remark why \eqref{eq:usual-combinatorial-interpretations}
is a $q$-analog of Theorem 3.4 from \cite{MR3778183}.

\begin{remark}\label{rem:relation-to-Thm-3.4-Schiffler}
    When $q=1$, the perfect matchings part of Equation
    \eqref{eq:usual-combinatorial-interpretations} becomes
    \[
    \xx
    = \frac{\#\Mcal(\xx')}
           {\#\Mcal(\xx'')}
    = \frac{\#\Mcal(\xx')}
           {\#\Mcal((\xx'')^{-1})}
    = \frac{\#\Mcal([a_0-1;a_1,\dots,a_{2\ell-1}])}
           {\#\Mcal([a_1-1,a_2,\dots,a_{2\ell-1}])}
    \]
    where we used the fact that the number of perfect matchings
    of the snake graph $\Gcal(\xx'')$
    is equal to the number of perfect matchings of
    its mirror image snake graph $\Gcal((\xx'')^{-1})$;
    see Lemma~\ref{lem:mirror-image-snake-graph-inverse-rational-number}.
    Thus, we recover Theorem 3.4 from \cite{MR3778183}.
    However, replacing $\xx''$ by its inverse does not work
    in general when $q\neq 1$.
    Thus, \eqref{eq:usual-combinatorial-interpretations}
    is a legitimate $q$-analog of Theorem 3.4 from \cite{MR3778183}.
\end{remark}

It remains open whether Theorem 1.3 from \cite{musiker_higher_2023} can also be
obtained from the statistics of the $m$-dimer cover over a single snake graph.

\section{Markoff numbers and their \texorpdfstring{$q$}{q}-analogs}
\label{sec:Markoff}

A Markoff triple is a positive solution of the Diophantine equation
$x^2+y^2+z^2=3xyz$ \cite{MR1510073,M1879}.
Markoff triples can be defined recursively as follows:
$(1,1,1)$, $(1,2,1)$ and $(1,5,2)$ are Markoff triples and if
$(x,y,z)$ is a Markoff triple with $y\geq x$ and $y\geq z$, then
$(x,3xy-z,y)$ and $(y,3yz-x,z)$ are Markoff triples.
A list of small Markoff numbers (elements of a Markoff triple) is
\[
    1,2,5,13,29,34,89,169,194,233,433,610,985,1325,1597,2897,4181,
    \ldots
\]
referenced as sequence \href{https://oeis.org/A002559}{A002559} in OEIS \cite{OEISA005132}.

Christoffel words are words over the alphabet $\{\0,\1\}$ that can be defined
recursively as follows: $\0$, $\1$ and $\0\1$ are Christoffel words and if
$u,v,uv\in\{\0,\1\}^*$ are Christoffel words then $uuv$ and $uvv$ are Christoffel
words \cite{MR2464862}. The shortest Christoffel words are:
\[
    \0,\1,
    \0\1,
    \0\0\1,\0\1\1,
    \0\0\0\1,\0\0\1\0\1,\0\1\0\1\1,\0\1\1\1,
    \0\0\0\0\1,\0\0\0\1\0\0\1,\0\0\1\0\0\1\0\1,\0\0\1\0\1\0\1,
    \ldots
\]
Note that these are usually named \emph{lower} Christoffel words.
A Christoffel word other than the letters $\0$ or $\1$ is called \emph{proper}.

It is known that each Markoff number can be expressed in terms of a Christoffel
word.
More precisely,
let $\mu$ be the monoid homomorphism $\{\0,\1\}^* \rightarrow \GL_2(\mathbb{Z})$ defined by
\[
    \mu(\0) = \begin{pmatrix} 2 & 1 \\ 1 & 1 \end{pmatrix}
\quad
\text{ and }
\quad
    \mu(\1) = \begin{pmatrix} 5 & 2 \\ 2 & 1 \end{pmatrix}.
\]
Each Markoff number is equal to $\mu(w)_{12}$ for some Christoffel word $w$
\cite{MR2534916},
where
$M_{12}$ denotes
the element above the diagonal in a matrix
$M=\left(\begin{smallmatrix}M_{11}&M_{12}\\M_{21}&M_{22}\end{smallmatrix}\right)\in\GL_2(\mathbb{Z})$.

For example, the Markoff number 194 is associated with the Christoffel word $\0\0\1\0\1$
as it is the entry at position $(1,2)$ in the matrix
\begin{align*}
\mu(\0\0\1\0\1)
&=
\begin{pmatrix} 2 & 1 \\ 1 & 1 \end{pmatrix}
\begin{pmatrix} 2 & 1 \\ 1 & 1 \end{pmatrix}
\begin{pmatrix} 5 & 2 \\ 2 & 1 \end{pmatrix}
\begin{pmatrix} 2 & 1 \\ 1 & 1 \end{pmatrix}
\begin{pmatrix} 5 & 2 \\ 2 & 1 \end{pmatrix}
=
\begin{pmatrix} 463 & 194 \\ 284 & 119 \end{pmatrix}.
\end{align*}
Whether the map $w\mapsto\mu(w)_{12}$ provides a bijection between Christoffel
words and Markoff numbers is a question (stated differently in \cite{F1913})
that has remained open for more than 100 years \cite{MR3098784}.
The conjecture, also known as the \emph{uniqueness conjecture}, can be
expressed in terms of the injectivity of the map $w\mapsto\mu(w)_{12}$
over the set of Christoffel words \cite[\S 3.3]{MR3887697}.

A $q$-analog of the Markoff numbers was proposed in \cite{kogiso_q-deformations_2020}
which can be expressed using the following morphism of monoids
$\mu_q:\{\0,\1\}^*\to\GL_2(\Z[q^{\pm 1}])$ \cite{MR4265544}:
\begin{align*}
    \mu_q(\0)
    &=\nu_q(\1\0)
    =R_qL_q
    =
\begin{pmatrix}
    q + q^{2} & 1 \\
    q & 1
\end{pmatrix},\\
    \mu_q(\1)
    &=\nu_q(\1\1\0\0)
    =R_qR_qL_qL_q
    =
\begin{pmatrix}
    q + 2q^2+q^3+q^4 & 1 + q \\
    q + q^{2} & 1
\end{pmatrix}.
\end{align*}
It was proved that
the map $w\mapsto\mu_q(w)_{12}$ is injective
over the language of a balanced sequence \cite{MR4405998}
and over the set of Christoffel words \cite{MR4686132}.
Note that different $q$-analogs of Markoff numbers were proposed in
\cite{evans_q-deformed_2025} based on other pairs of Cohn matrices.

Consider the following morphism
    \[
    \begin{array}{rccl}
        \gamma:&\{\0,\1\}^*&\to& \{\0,\1\}^*\\
        &\0&\mapsto &\0\0\\
        &\1&\mapsto &\0\1\1\0
    \end{array}.
    \]
The following result provides a
combinatorial interpretation of the $q$-analog of Markoff numbers
based on the map $\mu_q$ in terms of the area statistics of the perfect
matching of a snake graph. Its proof is short and is based on the matrices
$L_q$ and $R_q$.

\begin{theorem}\label{thm:we-recover-the-q-analog-of-Markoff-number}
    The $q$-analog of the Markoff number associated with a proper Christoffel
    word $\0 m\1\in\{\0,\1\}^*$ is
    the area polynomial of the set of perfect matchings
    of the snake graph $G(\0\gamma(m)\0)$.
    More precisely,
    \[
        \mu_q(\0m\1)_{12}
        =
        \sum_{m\in M(\0\gamma(m)\0)} q^{\area(m\Delta\bfrak)}.
    \]
\end{theorem}

\begin{proof}
    Note that the morphism
    \[
    \begin{array}{rccl}
        \gamma':&\{\0,\1\}^*&\to& \{\0,\1\}^*\\
        &\0&\mapsto &\1\0\\
        &\1&\mapsto &\1\1\0\0
    \end{array}
    \]
    satisfies $\mu_q=\nu_q\circ \gamma'$ and
    \[
        \theta(\0\gamma(w)\0)=\0\gamma'(w)\1
    \]
    for every even-length word $w\in\{\0,\1\}^*$.
    Using Proposition~\ref{prop:statistics-over-snake-graph-Gw},
    we have
    \begingroup
    \allowdisplaybreaks
    \begin{align*}
        \sum_{m\in M(\0\gamma(m)\0)} q^{\area(m\Delta\bfrak)}
        &=\sum_{m\in M^\carreplein(\0\gamma(m)\0)} q^{\area(m\Delta\bfrak)}
         +\sum_{m\in M^\carrevide (\0\gamma(m)\0)} q^{\area(m\Delta\bfrak)}\\
        &=
        \left(\begin{array}{cc} 1 & 1 \end{array}\right)
        \left(\begin{array}{cc}
            \sum_{m\in M^\carreplein(\0\gamma(m)\0)} q^{\area(m\Delta\bfrak)}\\
            \sum_{m\in M^\carrevide (\0\gamma(m)\0)} q^{\area(m\Delta\bfrak)}
        \end{array}\right)
            \\
        &=
        \left(\begin{array}{cc} 1 & 1 \end{array}\right)
        \left(\begin{smallmatrix} 1 & 0 \\ 0 & q \end{smallmatrix}\right)^{-1}
            \nu_q(\theta(\0\gamma(m)\0))
        \left(\begin{smallmatrix} q \\ q \end{smallmatrix}\right)\\
        &=
        \left(\begin{array}{cc} 1 & 1 \end{array}\right)
        \left(\begin{smallmatrix} 1 & 0 \\ 0 & q \end{smallmatrix}\right)^{-1}
            \nu_q(\0\gamma'(m)\1)
        \left(\begin{smallmatrix} q \\ q \end{smallmatrix}\right)\\
        &=
        \left(\begin{array}{cc} q & 1 \end{array}\right)
            \cdot \nu_q(\0\gamma'(m)\1)\cdot
        \left(\begin{smallmatrix} 1 \\ 1 \end{smallmatrix}\right)\\
        &=
        \left(\begin{array}{cc} 1 & 0 \end{array}\right)
            R_q
            \cdot \nu_q(\0\gamma'(m)\1)\cdot
            R_qL_qL_q
        \left(\begin{smallmatrix} 0 \\ 1 \end{smallmatrix}\right)\\
        &=
        \left(\begin{array}{cc} 1 & 0 \end{array}\right)
            \cdot \nu_q(\1\0\gamma'(m)\1\1\0\0)\cdot
        \left(\begin{smallmatrix} 0 \\ 1 \end{smallmatrix}\right)\\
        &=
        \left(\begin{array}{cc} 1 & 0 \end{array}\right)
            \cdot \nu_q(\gamma'(\0m\1))\cdot
        \left(\begin{smallmatrix} 0 \\ 1 \end{smallmatrix}\right)\\
        &=
        \left(\begin{array}{cc} 1 & 0 \end{array}\right)
            \cdot \mu_q(\0m\1)\cdot
        \left(\begin{smallmatrix} 0 \\ 1 \end{smallmatrix}\right)\\
        &= \mu_q(\0m\1)_{12}.
    \end{align*}
    \endgroup
\end{proof}

As a consequence, when evaluating at $q=1$, we recover
Theorem 4 from \cite{zbMATH07181526},
see also Theorem 7.12 from \cite{MR3098784}
and Theorem C from \cite{MR4058266}.

\begin{corollary}\label{cor:we-recover-the-Markoff-number}
    The Markoff number associated with a proper Christoffel
    word $\0 m\1\in\{\0,\1\}^*$ is
    equal to the number of perfect matchings
    of the snake graph $G(\0\gamma(m)\0)$, that is,
    \[
        \mu(\0m\1)_{12}
        =
        \# M(\0\gamma(m)\0).
    \]
\end{corollary}

\begin{proof}
    It follows from Theorem~\ref{thm:we-recover-the-q-analog-of-Markoff-number},
    after evaluating at $q=1$.
\end{proof}

Our proof is different as it is based on
Proposition~\ref{prop:statistics-over-snake-graph-Gw}
which is more general as it works for all snake graphs
using a product over the elementary matrices $L_q$ and $R_q$.
The proof proposed in \cite{MR3098784}
uses a decomposition of certain snake graphs (those obtained from Christoffel words)
into blocks of size 2 or 4
each block being associated with one of the Cohn matrices
$\left(\begin{smallmatrix} 1 & 1 \\ 1 & 2\end{smallmatrix}\right)$
    and
$\left(\begin{smallmatrix} 3 & 2 \\ 4 & 3\end{smallmatrix}\right)$.

The bijections provided in
Theorem~\ref{thm:bijection-order-ideals-to-integersequences} and
Theorem~\ref{theo:bij} mean
that
Theorem~\ref{thm:we-recover-the-q-analog-of-Markoff-number}
and Corollary~\ref{cor:we-recover-the-Markoff-number}
also admit equivalent interpretations
of Markoff numbers and their $q$-analogs
in terms of admissible sequences and order ideals of fence posets.

\begin{example}
    This reproduces Example 7.11 in \cite{MR3098784}.
Consider the Christoffel word $\0\0\1\0\1=\0 m\1$
with $m=\0\1\0$.
The associated Markoff number is
    \[
        \mu(\0m\1)_{12}
        =\mu(\0\0\1\0\1)_{12}
        =433
        = \# M(\0\gamma(m)\0)
        = \# M(\0\0\1\1\0\0\0\0\1\1\0\0)
    \]
    which is equal to the number of perfect matchings
    of the snake graph $G(\0\0\1\1\0\0\0\0\1\1\0\0)=\Gcal(179/254)$
    shown in Figure~\ref{fig:aigner-example711}.
\end{example}

\begin{figure}[h]
\includegraphics{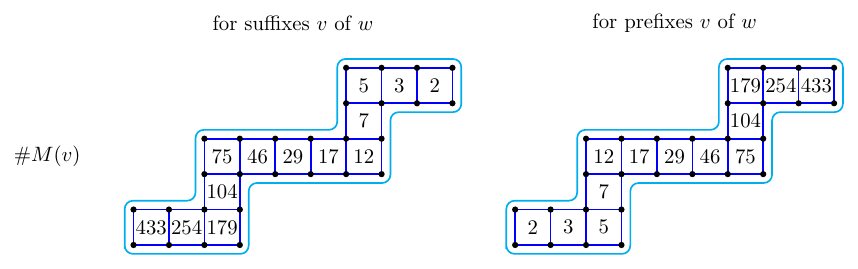}
    \caption{With $a=179/254 = [0; 1, 2, 2, 1, 1, 2, 2, 2]$,
    we have  $W(a)=011001011001$ and
$w=\theta\circ  W(a)=001100001100$.
The number of perfect matchings in the snake graph $G(v)$
    for all prefixes and suffixes of $w$ are indicated in the boxes.
    This reproduces Example 7.11 in Aigner's book \cite{MR3098784}
    for the Markov number 433.
    }
    \label{fig:aigner-example711}
\end{figure}

\section{Convexity of the set of admissible sequences}
\label{sec:convexity}

In this section, we prove additional results about the set
of admissible sequences for the even or odd-length continued
fraction expansion of a positive rational number,
namely it is a convex subset of $\Z^k$
whose convex hull is a polytope \cite{MR1311028}.

\begin{definition}
    Let $k\geq 1$ be an integer.
    For every even or odd-length continued fraction expansion
    $a=[a_0; a_1, \dots, a_{k-1}]$ of a positive rational number,
    we define the polytope
    \[
        \Pcal(a):=\conv(\Bcal(a))\subset\R^{k}
    \]
    where $\conv(X)$ denotes the convex hull of $X$ for every subset
    $X\subset\R^k$.
Also, for every positive rational number $\xx\in\Q_{>0}$,
let
\[
    \Pcaleven(x)=\Pcal(\CFeven(x))
    \qquad
    \text{ and }
    \qquad
    \Pcalodd(x)=\Pcal(\CFodd(x))
\]
be the two polytopes associated with $x\in\Q_{>0}$.
\end{definition}

Notice that the dimension of the polytope $\Pcal(a)$
is much smaller than the dimension ($=a_0+a_1+\dots+a_{k-1}$)
of the order polytope of the fence poset $\Fcal(a)$
\cite{zbMATH03958127}.

\begin{theorem}\label{thm:convex-hull}
    Let $k\geq 1$ be an integer
    and $a=[a_0; a_1, \dots, a_{k-1}]$ be the even or odd-length continued
    fraction expansion of a positive rational number.
    The set of admissible sequences for $a$
    are the integer points of the polytope $\Pcal(a)$, that is,
    \[
        \Pcal(a)\cap\Z^k = \Bcal(a).
    \]
\end{theorem}

\begin{proof}
    ($\supseteq$)
    Admissible sequences are sequences of nonnegative integers, so we have
    $\Bcal(a)\subset\Z^k$.
    From the definition of convex hull, we have
    $\Bcal(a)\subset\conv(\Bcal(a))=\Pcal(a)$.
    The conclusion follows.

    ($\subseteq$)
    Let $(c_i)_{0\leq i\leq k-1}\in\Pcal(a)\cap\Z^k$.
    There exist $m\geq1$ admissible sequences
    $(b_i^{(1)})_{0\leq i\leq k-1},\dots,
     (b_i^{(m)})_{0\leq i\leq k-1}\in\Bcal(a)$
    such that
    $(c_i)_{0\leq i\leq k-1}$ is in their convex hull.
    In other words,
    there exists $(\gamma_1,\dots,\gamma_m)\in[0,1]^m$,
    with $\gamma_1+\dots+\gamma_m=1$,
    such that
    \[
        c_i= \sum_{j=1}^m \gamma_j b_i^{(j)}
    \]
    for every integer $i$ with $0\leq i\leq k-1$.
    We may assume that the number $m$ of admissible sequences allowing to
    express $(c_i)_i$ is minimal. In other words, we can assume that $\gamma_j\neq
    0$ for every $j$ with $1\leq j\leq m$.
    If there exists $j$ such that $\gamma_j=1$, then
    $(c_i)_{0\leq i\leq k-1} = (b_i^{(j)})_{0\leq i\leq k-1}$ is an admissible
    sequence.
    Suppose that $0<\gamma_j<1$ for every $j$ with $1\leq j\leq m$.
    Since every sequence $(b_i^{(j)})_{0\leq i\leq k-1}$ is admissible for $a$,
    for every integer $i$, we have
    \[
        0\leq \min_{1\leq j\leq m}(b_i^{(j)}) \leq c_i
         \leq \max_{1\leq j\leq m}(b_i^{(j)}) \leq a_i.
    \]
    Assume $c_i=a_i$ for some odd integer $i>0$. Then
    $c_i=b_i^{(j)}=a_i$ for every $j$ with $1\leq j\leq m$.
    Since these $m$ sequences are admissible,
    we deduce that $b_{i-1}^{(j)}=a_{i-1}$ for every $j$ with $1\leq j\leq m$.
    Thus, $c_{i-1}=a_{i-1}$.
    Assume $c_i=0$ for some even integer $i>0$. Then
    $c_i=b_i^{(j)}=0$ for every $j$ with $1\leq j\leq m$.
    Since these $m$ sequences are admissible,
    we deduce that $b_{i-1}^{(j)}=0$ for every $j$ with $1\leq j\leq m$.
    Thus, $c_{i-1}=0$.
    Therefore, the sequence $(c_i)_{0\leq i\leq k-1}$ is admissible for $a$.
\end{proof}

Convexity also holds for the two subsets
forming the partition
$\Bcal(a)= \Bcal^\bullet(a) \cup \Bcal^\circ(a)$.
In fact, the two subsets are in two complementary half-spaces.
For every $x\in\R$, let
$\delta_{x}=0$ if $x=0$, and
$\delta_{x}=1$ if $x\neq 0$.
For every integer $k\geq 2$ and every continued fraction expansion
$a=[a_0;a_1,\dots,a_{k-1}]$ of a positive real number, consider
the open half-space
\[
    H_a^+ = \{(x_0,x_1,\dots,x_{k-1})\in\R^k\mid
                \delta_{a_0}(1-x_0)+(1-\delta_{a_0})(a_1-x_1)>0\}
        \subset\R^k.
\]

\begin{lemma}\label{lem:half-space-splits-the-polytope}
    We have
    \begin{align*}
        \Bcal^\bullet(a)
        &= \Bcal(a)\cap (\R^k\setminus H_a^+),\\
        \Bcal^\circ(a)
        &=\Bcal(a)\cap H_a^+.
    \end{align*}
\end{lemma}

\begin{proof}
    If $a_0=0$, we have $\delta_{a_0}=0$ and
    \begin{align*}
        \Bcal^\bullet(a)
        &=
        \{(b_i)_{0\leq i\leq k-1} \in\Bcal(a)
            \mid b_1=a_1
        \}
        =\Bcal(a)\cap (\R^k\setminus H_a^+)
        ,\\
        \Bcal^\circ(a)
        &=
        \{(b_i)_{0\leq i\leq k-1} \in\Bcal(a)
            \mid b_1<a_1
        \}
        =\Bcal(a)\cap H_a^+.
    \end{align*}
    If $a_0\neq0$, we have $\delta_{a_0}=1$ and
    \begin{align*}
        \Bcal^\bullet(a)
        &=
        \{(b_i)_{0\leq i\leq k-1}
        \in\Bcal(a)
            \mid
            b_0>0
        \}
        \\
        &=
        \{(b_i)_{0\leq i\leq k-1}
        \in\Bcal(a)
            \mid
            b_0\geq1
        \}
        =\Bcal(a)\cap (\R^k\setminus H_a^+)
        ,\\
        \Bcal^\circ(a)
        &=
        \{(b_i)_{0\leq i\leq k-1}
        \in\Bcal(a)
            \mid
            b_0=0
        \}
        \\
        &=
        \{(b_i)_{0\leq i\leq k-1}
        \in\Bcal(a)
            \mid
            b_0<1
        \}
        =\Bcal(a)\cap H_a^+.
    \end{align*}
\end{proof}

If $\xx\in\Q_{>0}$,
we deduce
from
Corollary~\ref{cor:three-combinatorial-interpretation-of-q-rational} (when $q=1$)
and Lemma~\ref{lem:half-space-splits-the-polytope},
that the number of integer points
in the polytope $\Pcaleven(x)$ that are inside (resp., outside) of the open half-space $H_a^+$
is equal to the denominator (resp., numerator) of the
rational number $\xx$.

Unimodality of the rank polynomials proved in \cite{MR4499341}
means that the sequence of number of integers points of $\Pcal(a)$ on each affine hyperplane
orthogonal to the vector $(1,1,\dots,1)$ is also unimodal.

Given the continued fraction expansion $a=[a_0; a_1, \dots, a_{k-1}]$
of a positive rational number,
the properties of the polytope $\Pcal(a)$ remain to be explored.
A first step would be to
describe its extremal points and facets.

\section*{Acknowledgements}
This work was partly funded from France's Agence Nationale de la Recherche
(ANR) projects IZES (ANR-22-CE40-0011) and COMBINÉ (ANR-19-CE48-0011).
It was also supported by grants from the
\emph{Symbolic Dynamics and Arithmetic Expansions}
(SymDynAr) Project, co-funded by ANR (ANR-23-CE40-0024)
and FWF (\href{https://dx.doi.org/10.55776/I6750}{I 6750}),
the Austrian Science Fund.

The second author acknowledges Université de Bordeaux's program
``\textit{Mobilité internationale des personnels de recherche}''
partially supporting a one-year stay at CRM-CNRS in Montréal (2025-2026).

We are thankful to Alejandro Morales for helpful discussions about fence posets
and their order ideals during the finalization of this article.

\bibliographystyle{alpha-first-name-initials-with-doi} %

\bibliography{biblio}

\newcommand{\etalchar}[1]{$^{#1}$}
\begin{thebibliography}{EJMGO25}

\bibitem[Aig13]{MR3098784}
M. Aigner.
\newblock {\em Markov's theorem and 100 years of the uniqueness conjecture. {A}
  mathematical journey from irrational numbers to perfect matchings}.
\newblock Springer, Cham, 2013.
\newblock \doi{10.1007/978-3-319-00888-2}.

\bibitem[Arn02]{MR1970391}
P. Arnoux.
\newblock Sturmian sequences.
\newblock In {\em Substitutions in dynamics, arithmetics and combinatorics},
  volume 1794 of {\em Lecture Notes in Math.}, pages 143--198. Springer,
  Berlin, 2002.
\newblock \doi{10.1007/3-540-45714-3\_6}.

\bibitem[AS03]{MR1997038}
J.-P. Allouche and J. Shallit.
\newblock {\em Automatic sequences}.
\newblock Cambridge University Press, Cambridge, 2003.
\newblock Theory, applications, generalizations.
\newblock \doi{10.1017/CBO9780511546563}.

\bibitem[BdL97]{MR1453849}
J. Berstel and A. de~Luca.
\newblock Sturmian words, {L}yndon words and trees.
\newblock {\em Theoret. Comput. Sci.}, 178(1-2):171--203, 1997.
\newblock \doi{10.1016/S0304-3975(96)00101-6}.

\bibitem[Ber01]{berthe_autour_2001}
V. Berthé.
\newblock Autour du système de numération d'{Ostrowski}.
\newblock {\em Bulletin of the Belgian Mathematical Society - Simon Stevin},
  8(2):209--239, 2001.

\bibitem[BLRS09]{MR2464862}
J. Berstel, A. Lauve, C. Reutenauer, and F.~V. Saliola.
\newblock {\em Combinatorics on words. {C}hristoffel words and repetitions in
  words}, volume~27 of {\em CRM Monograph Series}.
\newblock American Mathematical Society, Providence, RI, 2009.

\bibitem[BOSZ24]{burcroff_higher_2024}
A. Burcroff, N. Ovenhouse, R. Schiffler, and S.~W. Zhang.
\newblock Higher $q$-{Continued} {Fractions}.
\newblock August 2024.
\newblock \arxiv{2408.06902}.

\bibitem[Bou16]{bourla_ostrowski_2016}
A. Bourla.
\newblock The {Ostrowski} {Expansions} {Revealed}.
\newblock June 2016.
\newblock \arxiv{1605.07992}.

\bibitem[Cla20]{claussen_expansion_2020}
A. Claussen.
\newblock {\em Expansion {Posets} for {Polygon} {Cluster} {Algebras}}.
\newblock PhD thesis, May 2020.
\newblock \arxiv{2005.02083}.

\bibitem[CS13]{MR3034481}
I. Canakci and R. Schiffler.
\newblock Snake graph calculus and cluster algebras from surfaces.
\newblock {\em J. Algebra}, 382:240--281, 2013.
\newblock \doi{10.1016/j.jalgebra.2013.02.018}.

\bibitem[{\c{C}}S18]{MR3778183}
{\.{I}}. {\c{C}}anak\c{c}i and R. Schiffler.
\newblock Cluster algebras and continued fractions.
\newblock {\em Compos. Math.}, 154(3):565--593, 2018.
\newblock \doi{10.1112/S0010437X17007631}.

\bibitem[{\c{C}}S20]{MR4058266}
{\.{I}}. {\c{C}}anak\c{c}i and R. Schiffler.
\newblock Snake graphs and continued fractions.
\newblock {\em European J. Combin.}, 86:103081, 19, 2020.
\newblock \doi{10.1016/j.ejc.2020.103081}.

\bibitem[{\c{C}}S21]{zbMATH07336893}
{\.I}. {\c{C}}anak{\c{c}}{\i} and S. Schroll.
\newblock Lattice bijections for string modules, snake graphs and the weak
  {Bruhat} order.
\newblock {\em Adv. Appl. Math.}, 126:22, 2021.
\newblock \doi{10.1016/j.aam.2020.102094}.
\newblock Id/No 102094.

\bibitem[CW00]{MR1763062}
N. Calkin and H.~S. Wilf.
\newblock Recounting the rationals.
\newblock {\em Amer. Math. Monthly}, 107(4):360--363, 2000.
\newblock \doi{10.2307/2589182}.

\bibitem[DT89]{MR1020484}
J.-M. Dumont and A. Thomas.
\newblock Syst\`emes de num\'eration et fonctions fractales relatifs aux
  substitutions.
\newblock {\em Theoret. Comput. Sci.}, 65(2):153--169, 1989.
\newblock \doi{10.1016/0304-3975(89)90041-8}.

\bibitem[EFG{\etalchar{+}}12]{epifanio_sturmian_2012}
C. Epifanio, C. Frougny, A. Gabriele, F. Mignosi, and J. Shallit.
\newblock Sturmian graphs and integer representations over numeration systems.
\newblock {\em Discrete Applied Mathematics}, 160(4):536--547, March 2012.
\newblock \doi{10.1016/j.dam.2011.10.029}.

\bibitem[EJMGO25]{evans_q-deformed_2025}
S. Evans, P. Jouteur, S. Morier-Genoud, and V. Ovsienko.
\newblock On $q$-deformed {Markov} numbers. {Cohn} matrices and perfect
  matchings with weighted edges.
\newblock July 2025.
\newblock \arxiv{2507.19080}.

\bibitem[EKLP92]{zbMATH00130434}
N. Elkies, G. Kuperberg, M. Larsen, and J. Propp.
\newblock Alternating-sign matrices and domino tilings. {I}.
\newblock {\em J. Algebr. Comb.}, 1(2):111--132, 1992.
\newblock \doi{10.1023/A:1022420103267}.

\bibitem[Fog02]{MR1970385}
N.~P. Fogg.
\newblock {\em Substitutions in Dynamics, Arithmetics and Combinatorics},
  volume 1794 of {\em Lecture Notes in Mathematics}.
\newblock Springer-Verlag, Berlin, 2002.
\newblock Edited by V. Berth\'{e}, S. Ferenczi, C. Mauduit and A. Siegel.
\newblock \doi{10.1007/b13861}.

\bibitem[Fro13]{F1913}
G. Frobenius.
\newblock \"{U}ber die {M}arkoffschen {Z}ahlen.
\newblock {\em Sitzungsberichte der K\"{o}niglich Preussischen Akademie der
  Wissenschaften zu Berlin}, 26:458--487, 1913.

\bibitem[GKP94]{MR1397498}
R.~L. Graham, D.~E. Knuth, and O. Patashnik.
\newblock {\em Concrete mathematics. {A} foundation for computer science}.
\newblock Addison-Wesley Publishing Company, Reading, MA, second edition, 1994.

\bibitem[HW08]{MR2445243}
G.~H. Hardy and E.~M. Wright.
\newblock {\em An introduction to the theory of numbers}.
\newblock Oxford University Press, Oxford, sixth edition, 2008.
\newblock Revised by D. R. Heath-Brown and J. H. Silverman, With a foreword by
  Andrew Wiles.

\bibitem[Knu09]{zbMATH05597000}
D.~E. Knuth.
\newblock {\em The art of computer programming. {Vol}. 4, {Fasc}. 0--4. {Fasc}.
  0: {Introduction} to combinatorial algorithms and {Boolean} functions.
  {Fasc}. 1: {Bitwise} tricks \& techniques, binary decision diagrams. {Fasc}.
  2: {Generating} all tuples and permutations. {Fasc}. 3: {Generating} all
  combinations and partitions. {Fasc}. 4: {Generating} all trees. {History} of
  combinatorial generation.}
\newblock Upper Saddle River, NJ: Addison-Wesley, fascicle 0: 3rd printing
  2009; {Fascicle} 1: 1st printing 2009; {Fascicle} 2: 2nd printing 2009;
  {Fascicle} 3: 2nd printing 2009; {Fascicle} 4: 3rd printing 2009 edition,
  2009.

\bibitem[Kog20]{kogiso_q-deformations_2020}
T. Kogiso.
\newblock $q$-{Deformations} and $t$-deformations of {Markov} triples.
\newblock October 2020.
\newblock \arxiv{2008.12913}.

\bibitem[KOR23]{MR4499341}
E. Kantarc\i~O\u{g}uz and M. Ravichandran.
\newblock Rank polynomials of fence posets are unimodal.
\newblock {\em Discrete Math.}, 346(2):Paper No. 113218, 20, 2023.
\newblock \doi{10.1016/j.disc.2022.113218}.

\bibitem[KOY25]{zbMATH08015631}
E. Kantarc{\i}~O{\u{g}}uz and E. Y{\i}ld{\i}r{\i}m.
\newblock Cluster expansions: $t$-walks, labeled posets and matrix
  calculations.
\newblock {\em J. Algebra}, 669:183--219, 2025.
\newblock \doi{10.1016/j.jalgebra.2025.01.024}.

\bibitem[Kuo04]{MR2074946}
E.~H. Kuo.
\newblock Applications of graphical condensation for enumerating matchings and
  tilings.
\newblock {\em Theoret. Comput. Sci.}, 319(1-3):29--57, 2004.
\newblock \doi{10.1016/j.tcs.2004.02.022}.

\bibitem[LL22]{MR4405998}
S. Labb\'{e} and M. Lapointe.
\newblock The {$q$}-analog of the {M}arkoff injectivity conjecture over the
  language of a balanced sequence.
\newblock {\em Comb. Theory}, 2(1):Paper No. 9, 25, 2022.
\newblock \doi{10.5070/c62156881}.

\bibitem[LL24]{MR4836876}
S. Labb\'e and J. Lep\v{s}ov\'a.
\newblock Dumont-{T}homas complement numeration systems for {$\Bbb Z$}.
\newblock {\em Integers}, 24:Paper No. A112, 27, 2024.
\newblock \doi{10.5281/zenodo.14340125}.

\bibitem[LLS23]{MR4686132}
S. Labb\'e, M. Lapointe, and W. Steiner.
\newblock A {$q$}-analog of the {M}arkoff injectivity conjecture holds.
\newblock {\em Algebr. Comb.}, 6(6):1677--1685, 2023.
\newblock \doi{10.5802/alco.322}.

\bibitem[LMG21]{MR4265544}
L. Leclere and S. Morier-Genoud.
\newblock {$q$}-deformations in the modular group and of the real quadratic
  irrational numbers.
\newblock {\em Adv. in Appl. Math.}, 130:Paper No. 102223, 28, 2021.
\newblock \doi{10.1016/j.aam.2021.102223}.

\bibitem[LS19]{MR4016518}
K. Lee and R. Schiffler.
\newblock Cluster algebras and {J}ones polynomials.
\newblock {\em Selecta Math. (N.S.)}, 25(4):Paper No. 58, 41, 2019.
\newblock \doi{10.1007/s00029-019-0503-x}.

\bibitem[Mar79]{M1879}
A. Markoff.
\newblock Sur les formes quadratiques binaires ind\'{e}finies.
\newblock {\em Math. Ann.}, 15(3):381--496, 1879.

\bibitem[Mar80]{MR1510073}
A. Markoff.
\newblock Sur les formes quadratiques binaires ind\'{e}finies (second
  m\'{e}moire).
\newblock {\em Math. Ann.}, 17(3):379--399, 1880.
\newblock \doi{10.1007/BF01446234}.

\bibitem[MGO19]{MorierGenoud2019}
S. Morier-Genoud and V. Ovsienko.
\newblock On $q$-deformed real numbers.
\newblock {\em Experimental Mathematics}, pages 1--9, October 2019.
\newblock \doi{10.1080/10586458.2019.1671922}.

\bibitem[MGO20]{MR4073883}
S. Morier-Genoud and V. Ovsienko.
\newblock {$q$}-deformed rationals and {$q$}-continued fractions.
\newblock {\em Forum Math. Sigma}, 8:Paper No. e13, 55, 2020.
\newblock \doi{10.1017/fms.2020.9}.

\bibitem[MOSZ23]{musiker_higher_2023}
G. Musiker, N. Ovenhouse, R. Schiffler, and S.~W. Zhang.
\newblock Higher {Dimer} {Covers} on {Snake} {Graphs}.
\newblock June 2023.
\newblock \arxiv{2306.14389}.

\bibitem[MPS25]{mcconville_hyperbinary_2025}
T. McConville, J. Propp, and B.~E. Sagan.
\newblock Hyperbinary partitions and $q$-deformed rationals.
\newblock August 2025.
\newblock \arxiv{2508.20026}.

\bibitem[MS10]{MR2661414}
G. Musiker and R. Schiffler.
\newblock Cluster expansion formulas and perfect matchings.
\newblock {\em J. Algebraic Combin.}, 32(2):187--209, 2010.
\newblock \doi{10.1007/s10801-009-0210-3}.

\bibitem[MSS21]{MR4266256}
T. McConville, B.~E. Sagan, and C. Smyth.
\newblock On a rank-unimodality conjecture of {M}orier-{G}enoud and {O}vsienko.
\newblock {\em Discrete Math.}, 344(8):Paper No. 112483, 13, 2021.
\newblock \doi{10.1016/j.disc.2021.112483}.

\bibitem[MSW11]{MR2807089}
G. Musiker, R. Schiffler, and L. Williams.
\newblock Positivity for cluster algebras from surfaces.
\newblock {\em Adv. Math.}, 227(6):2241--2308, 2011.
\newblock \doi{10.1016/j.aim.2011.04.018}.

\bibitem[MSW13]{zbMATH06144657}
G. Musiker, R. Schiffler, and L. Williams.
\newblock Bases for cluster algebras from surfaces.
\newblock {\em Compos. Math.}, 149(2):217--263, 2013.
\newblock \doi{10.1112/S0010437X12000450}.

\bibitem[Mus25]{musiker_super_2025}
G. Musiker.
\newblock Super {Markov} {Numbers} and {Signed} {Double} {Dimer} {Covers}.
\newblock March 2025.
\newblock \arxiv{2503.21872}.

\bibitem[MZS02]{MR1948779}
E. Munarini and N. Zagaglia~Salvi.
\newblock On the rank polynomial of the lattice of order ideals of fences and
  crowns.
\newblock {\em Discrete Math.}, 259(1-3):163--177, 2002.
\newblock \doi{10.1016/S0012-365X(02)00378-3}.

\bibitem[NS16]{zbMATH06559878}
T. Nakanishi and S. Stella.
\newblock Wonder of sine-{Gordon} $y$-systems.
\newblock {\em Trans. Am. Math. Soc.}, 368(10):6835--6886, 2016.
\newblock \doi{10.1090/tran/6505}.

\bibitem[{OEI}22]{OEISA005132}
{OEIS Foundation Inc.}
\newblock Entry {A}005132 in the on-line encyclopedia of integer sequences,
  2022.
\newblock \url{http://oeis.org/A005132}.

\bibitem[Ost21]{zbMATH02600866}
A. Ostrowski.
\newblock Bemerkungen zur {Theorie} der diophantischen {Approximationen}.
\newblock {\em Abh. Math. Semin. Univ. Hamb.}, 1:77--98, 1921.
\newblock \doi{10.1007/BF02940581}.

\bibitem[Ove23]{MR4588256}
N. Ovenhouse.
\newblock {$q$}-rationals and finite {S}chubert varieties.
\newblock {\em C. R. Math. Acad. Sci. Paris}, 361:807--818, 2023.
\newblock \doi{10.5802/crmath.446}.

\bibitem[Pro20]{zbMATH07181526}
J. Propp.
\newblock The combinatorics of frieze patterns and {Markoff} numbers.
\newblock {\em Integers}, 20:paper a12, 38, 2020.
\newblock \doi{10.5281/zenodo.10717499}.

\bibitem[Reu09]{MR2534916}
C. Reutenauer.
\newblock Christoffel words and {M}arkoff triples.
\newblock {\em Integers}, 9:A26, 327--332, 2009.
\newblock \doi{10.1515/INTEG.2009.027}.

\bibitem[Reu19]{MR3887697}
C. Reutenauer.
\newblock {\em From {C}hristoffel words to {M}arkoff numbers}.
\newblock Oxford University Press, Oxford, 2019.
\newblock \doi{10.1093/oso/9780198827542.001.0001}.

\bibitem[Sta86]{zbMATH03958127}
R.~P. Stanley.
\newblock Two poset polytopes.
\newblock {\em Discrete Comput. Geom.}, 1:9--23, 1986.
\newblock \doi{10.1007/BF02187680}.

\bibitem[Sta12]{MR2868112}
R.~P. Stanley.
\newblock {\em Enumerative combinatorics. {V}olume 1}, volume~49 of {\em
  Cambridge Studies in Advanced Mathematics}.
\newblock Cambridge University Press, Cambridge, second edition, 2012.
\newblock \doi{10.1017/CBO9781139058520}.

\bibitem[TU25]{topkara_equivariant_2025}
M. Topkara and A.~M. Uludag.
\newblock Equivariant {Modular} {Functions} and {Quantizations} of {Continued}
  {Fractions}.
\newblock September 2025.
\newblock \arxiv{2509.06036}.

\bibitem[Zie95]{MR1311028}
G.~M. Ziegler.
\newblock {\em Lectures on polytopes}, volume 152 of {\em Graduate Texts in
  Mathematics}.
\newblock Springer-Verlag, New York, 1995.
\newblock \doi{10.1007/978-1-4613-8431-1}.

\end{thebibliography}

\end{document}